\documentclass[11pt]{article}

\usepackage[dvips]{epsfig}
\usepackage{graphicx}
\usepackage{amssymb,graphics,amsmath,amsthm,amsopn,amstext,amsfonts,color,  algorithm,algorithmic}

\usepackage{subfigure}

\usepackage{multicol}
\newcommand{\gap}{\vspace{0.1in}}

\newcommand{\wt}{\widetilde}

\newcommand{\wh}{\widehat}

\newcommand{\ol}{\overline}

\newcommand{\SOL}{ \mbox{SOL} }

\newcommand{\Wcal}{\mathcal W}
\newcommand{\Ical}{\mathcal I}
\newcommand{\Jcal}{\mathcal J}
\newcommand{\Scal}{\mathcal S}

\newcommand{\Pcal}{\mathcal P}
\newcommand{\Ccal}{\mathcal C}
\newcommand{\Kcal}{\mathcal K}
\newcommand{\Lcal}{\mathcal L}

\newcommand{\mycut}[1]{{}}

\newcommand{\argmin}{\operatornamewithlimits{\arg\min}}

\newtheorem{theorem}{Theorem}[section] 
\newtheorem{lemma}{Lemma}[section] 
\newtheorem{corollary}{Corollary}[section] 
\newtheorem{proposition}{Proposition}[section] 
\newtheorem{remark}{Remark}[section]
\newtheorem{example}{Example}[section] 

\begin{document}

\title{Column Partition based Distributed Algorithms for Coupled Convex Sparse Optimization: Dual and Exact Regularization Approaches}


\author{Jinglai Shen\footnote{J. Shen and Eswar Kumar H.K. are with Department of Mathematics and Statistics, University of Maryland Baltimore County, Baltimore, MD 21250, U.S.A. Emails: {\tt shenj@umbc.edu} and {\tt eswar1@umbc.edu}.}, \ \ \ Jianghai Hu\footnote{J. Hu is with the School of
    Electrical and Computer Engineering, Purdue University, West Lafayette, IN 47907, U.S.A. E-mail: {\tt jianghai@purdue.edu}.}, \ \ and \ \ Eswar Kumar Hathibelagal Kammara}

\maketitle

\begin{abstract}
This paper develops column partition based distributed schemes for a class of large-scale convex sparse optimization problems, e.g., basis pursuit (BP), LASSO, basis pursuit denosing (BPDN),  and their extensions, e.g., fused LASSO. We are particularly interested in the cases where
the number of (scalar) decision variables is much larger than the number of (scalar) measurements, and each agent has limited memory or computing capacity such that it only knows a small number of columns of a measurement matrix. These problems in consideration are densely coupled and cannot be formulated as separable convex programs using column partition. To overcome this difficulty, we consider their dual problems which are separable or locally coupled. Once a dual solution is attained, it is shown that a primal
solution can be found from the dual of corresponding regularized BP-like problems
under suitable exact regularization conditions. A wide range of existing distributed
schemes can be exploited to solve the obtained dual problems. This yields two-stage column partition based distributed schemes for LASSO-like and BPDN-like problems; the overall convergence of these schemes is established using sensitivity analysis techniques. Numerical results illustrate the effectiveness of the proposed schemes.
%
\end{abstract}

%
\section{Motivation and Introduction} \label{sect:introduction}

Sparse modeling and approximation finds broad applications in numerous fields of contemporary interest, including signal and image processing, compressed sensing, machine learning, and high dimensional statistics and data analytics. Various efficient  schemes have been proposed for convex or nonconvex sparse signal recovery \cite{FoucartRauhut_book2013, ShenMou_Arx19}. To motivate the work of this paper, consider the well-studied LASSO problem:
%
%
$
\min_{x\in \mathbb R^N} \ \frac{1}{2} \| A x - b \|^2_2 + \lambda \, \| x \|_1,
$
%
%
where $A \in \mathbb R^{m\times N}$ is the measurement (or sensing) matrix, $b \in \mathbb R^m$ is the measurement (or sensing) vector, $\lambda >0$ is the penalty parameter, and $x\in \mathbb R^N$ is the decision variable. In the setting of sparse recovery, $N$ is much larger than $m$. Besides,  the measurement matrix $A$ usually satisfies certain uniform recovery conditions for recovery efficiency, e.g., the restricted isometry property \cite{FoucartRauhut_book2013}. As such, $A$ is often a dense matrix, namely, (almost) all of its elements are nonzero. We aim to develop distributed algorithms to solve the LASSO and other relevant problems, where each agent only knows the vector $b$ and a small subset of the columns of $A$. Specifically, let $\{\Ical_1, \ldots, \Ical_p\}$ be a disjoint union of $\{1, \ldots, N\}$ such that $\{ A_{\bullet\Ical_i} \}^p_{i=1}$ forms a column partition of $A$.  For each $i$, the $i$th agent only has the knowledge of $b$ and $A_{\bullet\Ical_i}$. By running the proposed distributed scheme, it is expected that each agent $i$ attains the subvector of an optimal solution of the LASSO corresponding to the index set $\Ical_i$, i.e., $x^*_{\Ical_i}$, at the end of the scheme, where $x^*$ denotes an optimal solution of the LASSO.

The distributed optimization task described above is inspired by the following two scenarios arising from big data and network systems, respectively. In the context of big data, a practitioner may deal with a {\em ultra-large} data set, e.g., $N$ is extremely large, so that it would be impossible to store a vector $x \in \mathbb R^N$ in a single computing device, let alone the measurement matrix $A$. When $m$ is relatively small compared with $N$,  the proposed distributed schemes can be used where each device only needs to store the vector $b$ and a small number of the columns of $A$. The second scenario arises from multi-agent network systems, where each agent is operated by a low cost computing device which has limited memory and computing capacities.  Hence, even when $N$ is moderately large, it would be impractical for the entire matrix $A$ to be stored or computed on one of these devices. Therefore, the proposed distributed schemes can be exploited in this scenario.
Besides, the proposed algorithms can be extended to other sub-matrix  partitions (in both row and column) of $A$, even if $m$ is large in the above scenarios.

Centralized algorithms for the LASSO and other related convex sparse optimization problems, e.g., BP and BPDN, have been extensively studied, and can be categorized into the first-order methods \cite{BeckTeb_SIAMG09, FoucartRauhut_book2013, LaiYin_SIAMG13, WrightNF_TSP09, Yin_SIMAGE10}, and the second-order methods \cite{ByrdCNO_MP16, LiSunToh_SIOPT18}. A number of effective distributed or decentralized schemes have also been developed, for example, \cite{HuXiaoLiu_CDC18, LiuWright_SIOPT15, Mota_ISP11, PengYYin_Asilomor13, ShiLWYin_TSP15, YuanLY_SIOP16}, just to name a few.
%
%
To the best of our knowledge, most of the current  distributed or decentralized schemes for the LASSO and BPDN require the central knowledge of the entire $A$ in at least one step of these schemes (which are referred to as {\em partially distributed}); exceptions include \cite{Mota_ISP11, YuanLY_SIOP16} for distributed basis pursuit using the dual approach.
%
%
In contrast, the distributed schemes developed in this paper do not require the knowledge of the entire $A$ for any agent throughout the schemes and are {\em fully distributed}.
%
%
A major difficulty of developing column partition based fully distributed schemes for the LASSO and BPDN problems is that they are densely coupled.
Recently,  distributed schemes are developed for locally coupled convex programs \cite{HuXiaoLiu_CDC18}. However, since $A$ is a dense matrix, the loss function $\| A x  - b\|^2_2$ cannot be written in a locally coupled manner over a general network. Hence, the technique in \cite{HuXiaoLiu_CDC18} cannot be directly applied to the LASSO and BPDN.

%
%

The development of the proposed distributed algorithms relies on several key techniques in convex optimization, including dual problems, solution properties of the LASSO and BPDN, exact regularization,  distributed computing of separable convex programs, and sensitivity analysis. First, motivated by the dual approach for distributed BP \cite{Mota_ISP11, YuanLY_SIOP16}, we consider the Lagrangian dual problems of LASSO and BPDN, which are separable or locally coupled and thus can be solved via column partition based distributed schemes. By using  the solution properties of the LASSO and BPDN, we show that a primal solution is a solution of a basis pursuit-like (BP-like) problem depending on a dual solution. Under exact regularization conditions, a primal solution can be obtained from the dual of a regularized BP-like problem which can be solved by another column partition based distributed scheme. This leads to two-stage, column partition based distributed schemes for the LASSO and BPDN, where many existing distributed schemes or methods (e.g., distributed consensus optimization and distributed averaging schemes) can be used at each stage. The overall convergence of the two-stage schemes is established  via sensitivity analysis of the BP-like problem. The proposed schemes are applicable to a broad class of generalized BP, LASSO and BPDN under mild assumptions on  communication networks; we only assume that a network is static, connected and bidirectional. Extensions to time-varying networks can be made.

%
%

The rest of the paper is organized as follows. In Section~\ref{sect:problem_formulation_properties}, we formulate the BP, LASSO, BPDN and their extensions, and discuss basic solution properties. Exact regularization of these problems is investigated in  Section~\ref{sect:exact_regu}. Section~\ref{sect:dual_problems} establishes their dual problems and studies  properties in connection with the primal problems. Based on these results, column partition based fully distributed schemes are developed in Section~\ref{sect:distributed_schemes}, including two-stage distributed schemes for the LASSO-like and BPDN-like problems whose overall convergence is shown in Section~\ref{sect:overall_converg} via sensitivity analysis tools. Finally, numerical results are given in Section~\ref{sect:numerical}, and conclusions are drawn in Section~\ref{sect:conclusions}.

{\it Notation}.
Let $A$ be an $m\times N$ real matrix.
For any index set $\Scal \subseteq \{1, \ldots, N\}$, let
%
%
$A_{\bullet\Scal}$ be the matrix formed by the columns of $A$ indexed by elements of $\Scal$. Similarly, for an index set $\alpha \subseteq \{1, \ldots, m\}$, $A_{\alpha\bullet}$ is the matrix formed by the rows of $A$ indexed by elements of $\alpha$.
 Let $\{\Ical_i\}^p_{i=1}$ form a disjoint union of $\{1, \ldots, N\}$, and $\{ x_{\Ical_i} \}^p_{i=1}$ form a partition of $x \in \mathbb R^N$.
%
%
%
%
%
%
For  $a \in \mathbb R^n$, let $a_+:=\max(a, 0) \ge 0$ and $a_-:=\max(-a, 0) \ge 0$. 
%
%
%
%
%
%
 For a closed convex set $\Ccal$ in $\mathbb R^n$, $\Pi_\Ccal$ denotes the Euclidean projection operator onto $\Ccal$. For  $u, v \in \mathbb R^n$, $u \perp v$ stands for the orthogonality of $u$ and $v$, i.e., $u^T v =0$.

%
\section{Problem Formulation and Solution Properties} \label{sect:problem_formulation_properties}

We consider a class of  convex sparse minimization problems and their generalizations or extensions whose formulations are given as follows.

\gap

\noindent $\bullet$ {\bf Basis Pursuit (BP) and Extensions}. This problem intends to recover a sparse vector from noiseless measurement $b$ given by the following linear equality constrained optimization problem
%
%
\begin{equation} \label{eqn:BP}
\text{BP}: \quad  \underset{x \in \mathbb R^N}{\text{min}} \ \|x\|_1  \quad \text{subject to} \quad Ax=b,
\end{equation}
where we assume $b\in R(A)$. Geometrically, this problem seeks to minimize the $1$-norm distance from the origin to the affine set defined by $Ax = b$.  A generalization of the BP (\ref{eqn:BP}) is $\min_{x \in \mathbb R^N} \| E x \|_1$ subject to $ A x = b$, where $E \in \mathbb R^{r\times N}$ is a matrix.

\gap

\noindent $\bullet$ {\bf Least Absolute Shrinkage and Selection Operator (LASSO) and Extensions}. The standard LASSO intends to minimize the loss function $\|A x - b\|^2_2$ along with the $\ell_1$-norm penalty on $x$ treated as a convex relaxation of the sparsity of $x$:
\begin{equation} \label{eqn:Lasso}
  \text{LASSO}: \quad \underset{x \in \mathbb R^N}{\text{min}} \ \frac{1}{2} \| A x - b \|^2_2 + \lambda \, \| x \|_1,
\end{equation}
where $\lambda >0$ is the penalty parameter. A generalized LASSO is given by
$
   \underset{x \in \mathbb R^N}{\text{min}} \ \frac{1}{2} \| A x - b \|^2_2 +  \| E x \|_1,
$
where $E \in \mathbb R^{r\times N}$ is a given matrix. It includes several extensions and variations of the standard LASSO:
\begin{itemize}
  \item [(i)] Fused LASSO: $\min_{x \in \mathbb R^N} \ \frac{1}{2} \| A x - b \|^2_2 + \lambda_1 \, \|  x \|_1 + \lambda_2 \| D_1 x \|_1$, where $D_1 \in \mathbb R^{(N-1)\times N}$ denotes the first order difference matrix. Letting $E:=\begin{bmatrix} \lambda_1 I_N \\ \lambda_2 D_1 \end{bmatrix}$, the fused LASSO can be converted to the generalized LASSO.
  \item [(ii)] Generalized total variation denoising and $\ell_1$-trend filtering \cite{KimKBG_SIREW09}. A generalized total variation denoising is a generalized LASSO with $E=\lambda D_1$ for $\lambda>0$, whereas the generalized $\ell_1$-trend filtering has $ E=\lambda D_2$, where $D_2$ is the second order difference matrix.
\end{itemize}
Another related LASSO problem is the group LASSO, which is widely used in statistics for model selection \cite{YuanLin_JRSS06}. For a vector partition $\{ x_{\Ical_i} \}^p_{i=1}$ of $x \in \mathbb R^N$,
%
%
the group LASSO with $\lambda_i>0$ is given by
\begin{equation} \label{eqn:group_Lasso}
   \underset{x \in \mathbb R^N}{\text{min}} \ \frac{1}{2} \| A x - b \|^2_2 + \sum^p_{i=1} \lambda_i \| x_{\Ical_i} \|_2.
\end{equation}

\noindent $\bullet$ {\bf Basis Pursuit Denoising (BPDN) and Extensions}. Consider the following constrained optimization problem which incorporates noisy signals:
\begin{equation} \label{eqn:BP_denoising01}
 \text{BPDN}: \quad \underset{x \in \mathbb R^N}{\text{min}} \ \| x \|_1 \ \quad \text{subject to} \quad \| A x - b \|_2 \le \sigma,
\end{equation}
where $\sigma > 0$ characterizes the bound of noise or errors.
Note that when $\sigma=0$, it reduces to the basis pursuit. We assume that $\|b\|_2 > \sigma$ since otherwise, $x_*=0$ is the trivial solution.
Similar to the LASSO, the BPDN has several generalizations and extensions.
For example, it can be extended to $\min_{x \in \mathbb R^N} \| E x \|_1$ subject to $\| A x - b \|_2 \le \sigma$, where $E \in \mathbb R^{r\times N}$ is a matrix.
%
%
%
%
%
%
%


We summarize some fundamental solution properties of the aforementioned problems to be used in the subsequent development. For the convenience of generalizations and extensions, we treat the aforementioned problems in a more general setting. Let the constant $q > 1$, $E \in \mathbb R^{r \times N}$ be a matrix, $\|\cdot \|_\star$ be a norm on the Euclidean space, and $\Ccal$ be a polyhedral set. Consider the following problems:
\begin{eqnarray}
  (P_1): & \quad \min_{x \in \mathbb R^N} \ \|E x \|_\star \quad \mbox{ subject to } \quad Ax = b, \ \mbox{ and } x \in \Ccal \label{eqn:gen_BP}  \\
  (P_2): &\quad \min_{x \in \mathbb R^N} \ \frac{1}{2}\|A x - b \|^q_q + \|E x \|_\star \quad \mbox{ subject to } \quad  x \in \Ccal \label{eqn:gen_Lasso} \\
  (P_3): & \qquad \qquad \ \ \min_{x \in \mathbb R^N} \  \|E x \|_\star \quad \mbox{ subject to } \quad  \|A x - b \|_q \le \sigma \ \mbox{ and } \quad x \in \Ccal. \label{eqn:gen_BPDN}
\end{eqnarray}
Note that all the BP, LASSO, and BPDN models introduced before can be formulated within the above framework. For example, in the group LASSO,  $\|x \|_\star :=\sum^p_{i=1}\lambda_i \|x_{\Ical_i}\|_2$ is a norm. Letting $E:=I_N$ and $q=2$, the group LASSO is a special case of $(P_2)$.
%
%

\begin{proposition} \label{prop:soln_property}
 Fix $q>1$, and assume that the problems $(P_1)$-$(P_3)$ are feasible. The following hold:
  \begin{itemize}
     \item [(i)] Each of the problems $(P_1)-(P_3)$ attains a minimizer;
     \item [(ii)] Let $\mathcal H_2$ be the solution set of $(P_2)$. Then $A x = A x'$ and $\| E x\|_\star=\| E x' \|_\star$ for all $x, x' \in \mathcal H_2$;
     \item [(iii)] Suppose in $(P_3)$ that $\| b\|_q > \sigma$, $0\in \Ccal$, and the optimal value of $(P_3)$ is positive.
      Then each minimizer $x_*$ of $(P_3)$ satisfies $\| A x_* - b \|_q = \sigma$ and $A x$ is constant on the solution set.
  \end{itemize}
\end{proposition}

\begin{proof}
  Statements (i) and (ii) follow from a similar proof of \cite[Theorem 4.1]{MouShen_COCV18} and \cite[Lemma 4.1]{ZhangYC_JOTA15}, respectively. For Statement (iii), by a similar argument of \cite[Lemma 4.2(3)]{ZhangYC_JOTA15}, we have that any minimizer $x_*$ of $(P_3)$ satisfies $\| A x_* - b \|_q = \sigma$. Since $\| \cdot \|^q_q$ is strictly convex for $q>1$ \cite[Appendix]{ShenMousavi_SIOPT18} and $\| A x - b\|^q_q$ is  constant (whose value is $\sigma^q$) on the solution set, we deduce that $A x$ is constant on the solution set.
\end{proof}

A sufficient condition for the optimal value of $(P_3)$ to be positive, along with the conditions that $\| b\|_q > \sigma$ and $0\in \Ccal$, is that $E$ has full column rank. In fact, when $\| b\|_q > \sigma$ and $0\in \Ccal$, any minimizer $x_*$ must be nonzero. If $E$ has full column rank, then $E x_* \ne 0$ so that $\| E x_* \|_\star >0$.

To compute a solution of the LASSO in (\ref{eqn:Lasso_primal}) using its dual solution,
%
%
we need the following result similar to \cite[Theorem 2.1]{ZhangYC_JOTA15} or \cite[Proposition 3.2]{MouShen_COCV18}.  To be self-contained, we present its proof below.

\begin{proposition} \label{prop:Lasso_to_BP}
The following hold:
\begin{itemize}
   \item [(i)] Let $x_*$ be a minimizer of $(P_2)$ given by (\ref{eqn:gen_Lasso}). Then $z_*$ is a minimizer of $(P_2)$ if and only if $z_*$ is a minimizer of the BP-like problem given by  (\ref{eqn:gen_BP}), i.e.,
  $
    (P_1): \   \min_{z \in \mathbb R^N} \ \| E z \|_\star  \mbox{ subject to }  A z = A x_*  \mbox{ and } \ z \in \Ccal.
  $
  Furthermore, the optimal value of $(P_1)$ equals $\|E x_* \|_\star$.
\item [(ii)]  Let $x_*$ be a minimizer of $(P_3)$ given by (\ref{eqn:gen_BPDN}) which satisfies:  $\| b\|_q > \sigma$, $0\in \Ccal$, and the optimal value of $(P_3)$ is positive. Then $z_*$ is a minimizer of $(P_3)$ if and only if $z_*$ is a minimizer of the BP-like problem  (\ref{eqn:gen_BP}) with $b:= A x_*$, and  the optimal value of this $(P_1)$ equals $\|E x_* \|_\star$.
\end{itemize}
\end{proposition}

\begin{proof}
(i)  Let $\mathcal H_2$ be the solution set of $(P_2)$ given by (\ref{eqn:gen_Lasso}). By Proposition~\ref{prop:soln_property}, $A x = A x_*$ and $\|E x \|_\star = \|E x_* \|_\star$ for any $x \in \mathcal H_2$. Let $J(x):= \frac{1}{2} \| A x - y \|^q_q +  \| E x \|_\star$ be the objective function.
 For the ``if'' part, let $z_*$ be a minimizer of $(P_1)$. Then $z_* \in \Ccal$, $A z_* = A x_*$ and $\|E z_*\|_\star \le \| E x_* \|_\star$. Hence, $J(z_*) \le J(x_*)$. On the other hand, $J(x_*) \le J(z_*)$ because $x_*$ is a minimizer of $(P_2)$. Therefore, $J(x_*)=J(z_*)$  so that $z_*$ is a minimizer of $(P_2)$. It also implies that $\|E z_*\|_\star =\|E x_*\|_\star$ or equivalently  the optimal value of $(P_1)$ equals $\|E x_* \|_\star$.
 To show the ``only if'' part, let $z_*$ be a minimizer of  $(P_2)$.  Suppose $z_*$ is not a minimizer of $(P_1)$. Then there exists $u \in \mathbb R^N$ such that $u \in \Ccal$, $A u = A x_*$ and $\| E u \|_\star < \|E z_*\|_\star$. Since $z_*$ is a minimizer of $(P_2)$, we have $A z_* = A x_*$ and $\|E z_*\|_\star = \|E x_* \|_\star$. Hence, $J(u) < J(z_*)$, yielding a contradiction.

(ii) Suppose $(P_3)$ satisfies the specified conditions, and let $\mathcal H_3$ denote its solution set. By statement (iii) Proposition~\ref{prop:soln_property}, we have $\mathcal H_3=\{x \in \Ccal \, | \, A x = A x_*, \, \|E x \|_\star = \|E x_* \|_\star \}$ for a minimizer $x_*$ of $(P_3)$. ``If'': suppose $z_*$ be a minimizer of $(P_1)$ with $b:=A x_*$. Then $z_*\in \Ccal$, $A z_* = A x_*$ and $\|E z_*\|_\star \le \| E x_* \|_\star$. This shows that $z_*$ is a feasible point of $(P_3)$ and hence a minimizer in view of $\|E z_*\|_\star=\|E x_*\|_\star$. ``Only if'':
since any feasible point of $(P_1)$ with $b:= A x_*$ is a feasible point of $(P_3)$ and since $x_*$ is a feasible point of this $(P_1)$, we see that the optimal value of this $(P_1)$ equals $\| Ex_*\|_\star$.
Suppose $z_*$ is a minimizer of $(P_3)$.
 Then $z_* \in \mathcal H_3$ such that $z_* \in \Ccal$, $A z_* = Ax_*$, and $\|E z_*\|_\star = \|E x_* \|_\star$. Hence $z_*$ is a feasible point of this $(P_1)$ and thus a minimizer of this $(P_1)$.
\end{proof}



%
\section{Exact Regularization} \label{sect:exact_regu}

A key step in the development of column partition based distributed algorithms is  using dual problems. To establish a relation between solutions of a primal problem and its dual, we consider regularization of the primal problem, which is expected to give rise to a solution of the original primal problem. This pertains to the exact regularization of the original primal problem \cite{FriedlanderTesng_SIOPT07}.

%

We briefly review the exact regularization of general convex programs given in \cite{FriedlanderTesng_SIOPT07}. Consider the convex minimization problem $(P)$ and its regularized problem $(P_\varepsilon)$ for some $\varepsilon \ge 0$:
\[
   (P): \quad \min_{x \in \Pcal} f(x); \qquad \qquad (P_\varepsilon): \quad \min_{x \in \Pcal} f(x)+ \varepsilon h(x),
\]
where $f, h:\mathbb R^N \rightarrow \mathbb R$ are real-valued convex functions, and $\Pcal$ is a closed convex set. It is assumed that $(P)$ has a solution, and $h$ is coercive such that $(P_\varepsilon)$ has a solution for each $\varepsilon>0$. A weaker assumption can be made for $h$; see \cite[Section 1.2]{FriedlanderTesng_SIOPT07} for details. We call the problem $(P)$ {\em exactly regularized} if there exists $\ol\varepsilon>0$ such that for any $\varepsilon \in (0, \ol\varepsilon]$, any solution of $(P_\varepsilon)$ is a solution of $(P)$. To establish the exact regularization, consider the following convex program: letting $f_*$ be the optimal value of $(P)$,
\[
   (P_h): \quad \min_{x \in \Pcal, \, f(x) \le f_*} \, h(x).
\]
Clearly,  the constraint set of $(P_h)$ is equivalent to $\{x \, | \, x \in \Pcal, f(x)=f_*\}$, which is the solution set of $(P)$. It is shown in \cite[Theorem 2.1]{FriedlanderTesng_SIOPT07} or \cite[Corollary 2.2]{FriedlanderTesng_SIOPT07} that $(P)$ is exactly regularized by the regularization function $h$ if and only if $(P_h)$ has a Lagrange multiplier $\mu_* \ge 0$, i.e., there exists a constant $\mu_* \ge 0$ such that
$
   \min_{x \in \Pcal, \, f(x) \le f_*} \, h(x)  =  \min_{x \in \Pcal } \, h(x) + \mu_* \big( \, f(x) - f_* \big).
$

\begin{corollary} \label{coro:P_h_Lagrange}
 The problem $(P_h)$ has a Lagrange multiplier $\mu_* \ge 0$ if and only if there exists a constant $\mu \ge 0$ such that  a minimizer $x_*$ of $(P_h)$ is a minimizer of $\min_{x \in \Pcal} h(x)+ \mu( f(x) - f_*)$.
\end{corollary}

\begin{proof}
``If'': suppose a constant $\mu \ge 0$ exists such that  a minimizer $x_*$ of $(P_h)$ is that of $\min_{x \in \Pcal} h(x)+ \mu( f(x) - f_*)$. Since $x_*$ is a feasible point of $(P_h)$, we have $x_* \in \Pcal$ and $f(x_*)\le f_*$ or equivalently $f(x_*)=f_*$. Hence, the optimal value of $\min_{x \in \Pcal} h(x)+ \mu( f(x) - f_*)$ is given by $h(x_*)$, which equals  $\min_{x \in \Pcal, \, f(x) \le f_*} \, h(x)$. Hence, $\mu_*:=\mu\ge 0$ is a Lagrange multiplier of $(P_h)$.

``Only If'': Let $\mu_* \ge 0$ be a Lagrange multiplier of $(P_h)$, and $x_*$ be a minimizer of $(P_h)$. Again, we have $x_* \in \Pcal$ and $f(x_*)=f_*$. This shows that $h(x_*)+ \mu_* (f(x_*) -f_*)=h(x_*)$. Hence,
\[
  h(x_*) \, =  \, \min_{x \in \Pcal, \, f(x) \le f_*} h(x) \, = \, \min_{x \in \Pcal } \, h(x) + \mu_* \big( \, f(x) - f_* \big) \le h(x_*).
\]
We thus deduce that $x_*$ is a minimizer of $\min_{x \in \Pcal} h(x)+ \mu( f(x) - f_*)$ with $\mu:=\mu_*$.
\end{proof}

%

%
%
%

%
\subsection{Exact Regularization of Convex Piecewise Affine Function based Optimization} \label{subsect:exact_reg_convx_PA}

We consider the exact regularization of convex piecewise affine functions based convex minimization problems with its applications to $\ell_1$-minimization given by the BP, LASSO, and BPDN.
A real-valued continuous  function $f:\mathbb R^N \rightarrow \mathbb R$ is piecewise affine (PA) if there exists a finite family of real-valued affine functions $\{f_i\}^\ell_{i=1}$ such that $h(x) \in \{f_i(x)\}^\ell_{i=1}$ for each $x\in \mathbb R^N$. 
A convex PA function $f:\mathbb R^N \rightarrow \mathbb R$ has the max-formulation \cite[Section 19]{Rockafellar_book70},
%
%
i.e., there exists a finite family of $(p_i, \gamma_i) \in \mathbb R^N \times \mathbb R, i=1, \ldots, \ell$ such that  $f(x) = \max_{i=1, \ldots, \ell} \, \big( \, p^T_i x + \gamma_i \, \big)$.
%
%
Convex PA functions represent an important class of nonsmooth convex functions in many applications, e.g., the $\ell_1$-norm $\|\cdot \|_1$, $f(x):=\| E x \|_1$ for a matrix $E$, a polyhedral gauge, and the $\ell_\infty$-norm; see \cite{MouShen_COCV18} for more discussions.
%
We first present a technical lemma whose proof is omitted.

\begin{lemma} \label{lem:P_extension}
  Let $f:\mathbb R^N \rightarrow \mathbb R$ and $h:\mathbb R^N \rightarrow \mathbb R$ be (not necessarily convex) functions and $\Pcal$ be a set such that $\min_{x \in \Pcal} f(x)$ attains a minimizer and its optimal value is denoted by $f_*$.
    Let the set $\mathcal W:= \{ (x, t) \, | \, x \in \Pcal, f(x) \le t \}$.
  Consider the following problems:
  \begin{align*}
    & (P_\varepsilon): \quad \min_{x \in \Pcal} f(x)+ \varepsilon h(x);  \qquad \  (P'_\varepsilon): \quad \min_{(x, t) \in \Wcal} t + \varepsilon h(x), \qquad \varepsilon \ge 0;  \\
    & (P_h): \quad \min_{x \in \Pcal, \, f(x) \le f_* } h(x);   \ \ \qquad (P'_h): \quad \min_{(x, t) \in \Wcal, \, t \le f_*} h(x).
 \end{align*}
Then the following hold:
 \begin{itemize}
   \item [(i)] Fix an arbitrary $\varepsilon \ge 0$. Then (a) if $x_*$ is an optimal solution of  $(P_\varepsilon)$, then $(x_*, f(x_*))$ is an optimal solution of $(P'_\varepsilon)$; (b) if $(x_*, t_*)$ is an optimal solution of $(P'_\varepsilon)$, then $t_*=f(x_*)$ and $x_*$ is an optimal solution of  $(P_\varepsilon)$.

  \item [(ii)] (a) If $x_*$ is an optimal solution of  $(P_h)$, then $(x_*, f_*)$ is an optimal solution of $(P'_h)$; (b) if $(x_*, t_*)$ is an optimal solution of $(P'_h)$, then $t_*=f_*$ and $x_*$ is an optimal solution of  $(P_h)$.
 \end{itemize}
\end{lemma}

\mycut{
\begin{proof}
  (i) Fix an $\varepsilon \ge 0$. (a) Suppose $x_*$ is an optimal solution of $(P_\varepsilon)$. Clearly, $(x_*, f(x_*))\in \Wcal$. For any $(x, t) \in \Wcal$, we have $f(x) \le t$. Since $x_*$ is an optimal solution of $(P_\varepsilon)$, we have $f(x_*) + \varepsilon h(x_*) \le f(x) + \varepsilon h(x) \le t + \varepsilon h(x)$ for all $(x, t) \in \Wcal$. This shows that $(x_*, f(x_*))$ is an optimal solution of $(P'_\varepsilon)$. To show Part (b), suppose $(x_*, t_*)$ is an optimal solution of $(P'_\varepsilon)$. Clearly,  $f(x_*) \le t_*$ by the definition of $\Wcal$. Since $(x_*, f(x_*)) \in \Wcal$, $t_*=\min_{(x_*, t) \in \Wcal} t \le f(x_*)$. Therefore, $t_*=f(x_*)$. Furthermore, for an arbitrary $x \in \Pcal$, we have $(x, f(x)) \in \Wcal$. Since $t_* + \varepsilon h(x_*) \le t + \varepsilon h(x)$ for any $(x, t) \in \Wcal$, we have $f(x_*)+  \varepsilon h(x_*) \le f(x)+  \varepsilon h(x)$. Hence, $x_*$ is an optimal solution of $(P_\varepsilon)$.

  (ii) For Part (a), let $x_*$ be an optimal solution of  $(P_h)$, i.e., $h(x_*)\le h(x)$ for all $x \in \Pcal$ with $f(x) \le f_*$. Consider an arbitrary $(x, t) \in \Wcal$ with $t\le f_*$. Then $x \in \Pcal$ and $f(x) \le t \le f_*$. This shows that $x$ is a feasible point of $(P_h)$, yielding $h(x_*) \le h(x)$. Noting that $f(x) \equiv f_*$ for all feasible points of $(P_h)$, we see that $(x_*, f_*)$ is an optimal solution of $(P'_h)$. Conversely, suppose $(x_*, t_*)$ is an optimal solution of $(P'_h)$, i.e., $x_* \in \Pcal, f(x_*) \le t_* \le f_*$. As $f(x_*)=f_*$, we have $t_*=f_*$. For any feasible point $x$ of $(P_h)$,  $(x, f(x))$ is a feasible point of $(P'_h)$. Thus $h(x_*) \le h(x)$ such that $x_*$ is a minimizer of $(P_h)$.
\end{proof}
}


The following proposition shows exact regularization for convex PA objective functions on a polyhedral set. This result has been mentioned in \cite{Yin_SIMAGE10} without a formal proof; we present a proof for completeness.

\begin{proposition} \label{prop:exact_reg_BP_like}
  Let $\Pcal$ be a polyhedral set, and $f:\mathbb R^N \rightarrow \mathbb R$ be a convex PA function  such that the problem $ (P): \min_{x \in \Pcal} f(x) $ has the nonempty solution set, and let
  $h:\mathbb R^N \rightarrow \mathbb R$ be a convex regularization function which is coercive.
    Then there exists $\ol \varepsilon >0$ such that for any $\varepsilon \in (0, \ol \varepsilon]$, any optimal solution of the regularized problem $(P_\varepsilon)$ is an optimal solution of $(P)$.
\end{proposition}

\begin{proof}
 Let $f_*$ be the optimal value of the problem $(P)$.
 In view of Lemma~\ref{lem:P_extension}, $(P)$ is equivalent to $(P'_0)$ and $(P_\varepsilon)$ is equivalent to $(P'_\varepsilon)$ for any $\varepsilon>0$ in the sense given by Lemma~\ref{lem:P_extension}. Hence, to show the exact regularization of $(P)$ via $(P_\varepsilon)$, it suffices to show the exact regularization of $(P'_0)$ via $(P'_\varepsilon)$. To show the latter, it follows from  \cite[Theorem 2.1]{FriedlanderTesng_SIOPT07} or \cite[Corollary 2.2]{FriedlanderTesng_SIOPT07} that we only need to show that $(P'_h)$ attains a Lagrange multiplier, namely, there exists a Lagrange multiplier $\mu_* \ge 0$ such that
$
  \min_{(x, t)\in \Wcal, \, t \le f_*} h(x)= \min_{(x, t) \in \Wcal} h(x) + \mu_* (t- f_*),
$
where we recall that $f_*$ is the optimal value of $(P)$ and $\mathcal W:= \{ (x, t) \, | \, x \in \Pcal, f(x) \le t \}$.
Suppose the convex PA function $f$ is given by $f(x) = \max_{i=1, \ldots, \ell} \, \big( \, p^T_i x + \gamma_i \, \big)$.
%
%
%
Then
$
 \Wcal \, = \,  \{ (x, t) \, | \, x \in \Pcal, \ p^T_i x + \gamma_i \le t, \ \forall \, i=1, \ldots, \ell \},
$
and $\Wcal$ is thus a polyhedral set.
Since $\Wcal$ is polyhedral and $t \le f_*$ is a linear inequality constraint, it follows from \cite[Proposition 5.2.1]{Bertsekas_book99} that there exists $\mu_* \ge 0$ such that $\min_{(x, t)\in \Wcal, \, t \le f_*} h(x)= \min_{(x, t) \in \Wcal} h(x) + \mu_* (t- f_*)$. By \cite[Corollary 2.2]{FriedlanderTesng_SIOPT07}, $(P'_\varepsilon)$ is the exact regularization of $(P'_0)$ for all small $\varepsilon>0$.
\end{proof}

The above proposition yields the exact regularization for the BP-like problem with the $\ell_1$-norm.

\begin{corollary} \label{coro:exact_reg_BPL1}
  Let $\Ccal$ be a polyhedral set.
  %
  %
   Then the following problem  attains the exact regularization of $(P_1)$ for all sufficiently small $\alpha>0$:
  \begin{eqnarray*}
  (P_{1, \alpha}): & \quad \min_{x \in \mathbb R^N} \ \|E x \|_1 + \frac{\alpha}{2} \|x\|^2_2 \quad \mbox{ subject to } \quad Ax = b, \ \mbox{ and } x \in \Ccal.
\end{eqnarray*}
\end{corollary}

\begin{proof}
  %
  Let $f(x):=\|E x\|_1$ which is a convex PA function, $h(x) := \|x\|^2_2$, and $\Pcal :=\{ x \, | \, A x = b, \ x \in \Ccal\}$. Then $\Pcal$ is a polyhedral set. Applying Proposition~\ref{prop:exact_reg_BP_like}, we conclude that the exact regularization holds.
\end{proof}

%
\subsubsection{Failure of Exact Regularization of the LASSO and BPDN Problems} \label{subsect:fail_exact_reg_Lasso_BPDN}


We investigate exact regularization of the LASSO and BPDN  when the $\ell_1$-norm is used. For simplicity, we focus on the standard problems (i.e., $\mathcal C=\mathbb R^N$) although the results developed here can be extended. It follows from Proposition~\ref{prop:soln_property} that the solution sets of the standard LASSO and BPDN are polyhedral. Hence, the constraint sets of $(P_h)$ associated with the LASSO and BPDN are polyhedral.
However, unlike the BP-like problem,  we show by examples that exact regularization fails in general. This motivates us to develop two-stage distributed algorithms in Section~\ref{sect:distributed_schemes} rather than directly using the regularized LASSO and BPDN.
%
%
%
 Our first example shows that in general, the standard LASSO (\ref{eqn:Lasso})
%
 %
 is {\em not} exactly regularized by the regularization function $h(x)=\| x\|^2_2$.

\begin{example} \rm
  Let $A =[ I_2 \ I_2 \ \cdots \ I_2]\in \mathbb R^{2\times N}$ with $N=2r$ for some $r\in \mathbb N$, and $b\in \mathbb R^2_{++}$. Hence, we can partition a vector $x \in \mathbb R^N$ as $x=(x^1, \ldots, x^r)$ where each $x^i\in \mathbb R^2$. When $0< \lambda <1$, it follows from the KKT condition:
  $ 0 \in A^T(A x_* -b) + \lambda \partial \|x_* \|_1$ and   a straightforward computation
   that a particular optimal solution $x_*$ is given by $x^i_*=\frac{1-\lambda}{r} b >0$ for all $i=1, \ldots, r$. Hence,
   the solution set $\mathcal H = \{ x =(x^1, \ldots, x^r) \, | \, \sum^r_{i=1} x^i = (1-\lambda)b, \, \|x\|_1 \le (1-\lambda) \|b\|_1 \}$.
 %
 %
    Consider the regularized LASSO for $\alpha>0$: $\min_{x\in \mathbb R^N}\frac{1}{2} \| A x - b \|^2_2 + \lambda \, \| x \|_1 + \frac{\alpha}{2} \|x\|^2_2$. For each $\alpha>0$, it can be shown that its unique optimal solution $x_{*, \alpha}$ is given by $x^i_{*, \alpha} = \frac{1-\lambda}{r+\alpha} b$ for each $i=1, \ldots, r$. Hence,  $x_{*, \alpha}\notin \mathcal H$ for any $\alpha>0$. \hspace{50pt} $\square$
\end{example}


%
In what follows, we show the failure of exact regularization of the standard BPDN (\ref{eqn:BP_denoising01}). Consider the convex minimization problem for a constant $\mu \ge 0$,
\[
 (P_\mu): \qquad \min_{\|A x - b\|_2 \le \sigma} \ \frac{1}{2} \|x\|^2_2 + \mu \| x\|_1,
\]
where $A \in \mathbb R^{m \times N}$, $b\in \mathbb R^m$, and $\sigma>0$ with $\|b\|_2 > \sigma$. Further, consider the max-formulation of $\ell_1$-norm, i.e., $\|x\|_1 = \max_{i=1, \ldots, 2^N} p^T_i x$, where each $p_i \in \big\{ (\pm 1, \pm 1, \ldots, \pm 1)^T \big\} \subset \mathbb R^N$; see \cite[Section 4.2]{MouShen_COCV18} for details.

\begin{lemma} \label{lem:P_mu_mini}
 A feasible point $x_*\in \mathbb R^N$ of $(P_\mu)$ is a minimizer of $(P_\mu)$ if and only if $\|A x_* - b\|_2=\sigma$ and
 \begin{equation} \label{eqn:P_mu_opt}
    \Big[ \, A u =0 \ \mbox{ or } \ g^T u < 0 \, \Big] \ \Rightarrow \ \Big (x_*^T u  + \mu \max_{i \in \Ical(x_*)} p_i ^T u \Big) \ge 0,
 \end{equation}
 where $g:=A^T (A x_* - b)$, and $\Ical(x_*):= \{ i \, | \, p^T_i x_* = \|x_* \|_1 \}$.
\end{lemma}

\begin{proof}
 It follows from a similar argument of \cite[Lemma 4.2(3)]{ZhangYC_JOTA15} that a minimizer $x_*$ of $(P_\mu)$ satisfies $\| A x_* - b \|_2 = \sigma$. The rest of the proof resembles that of \cite[Theorem 3.3]{MouShen_COCV18}; we present its proof for completeness. Since $(P_\mu)$ is a convex program, it is easy to see that $x_*$ is a minimizer of $(P_\mu)$ if and only if $u_*=0$ is a local minimizer of the following problem:
 \[
   (\wt P_\mu): \quad  \min_{u \in \mathbb R^N} \, \Big (x_*^T u  + \mu \max_{i \in \Ical(x_*)} p_i ^T u \Big), \quad \mbox{ subject to } \quad g^T u + \frac{1}{2} \|A u \|^2_2 \le 0.
 \]
 In what follows, we show that the latter holds if and only if the implication (\ref{eqn:P_mu_opt}) holds.
 Let $J(u):=x_*^T u + \mu \max_{i \in \Ical(x_*)} p_i^T u$.
 ``If'': Let $\mathcal U$ be a neighborhood of $u_*=0$. For any $u \in \mathcal U$ satisfying $g^T u + \frac{1}{2} \|A u \|^2_2 \le 0$, either $A u =0$ or $Au \ne 0$. For the latter, we have $g^T u <0$. Hence, in both cases, we deduce from (\ref{eqn:P_mu_opt}) that $J( u)  \ge 0=J(u_*)$. This shows that $u_*=0$ is a local minimizer of $(\wt P_\mu)$. ``Only If'': suppose $u_*=0$ is a local minimizer of $(\wt P_\mu)$. For any $u$ with $A u =0$, we have $g^T u=0$ such  that $v:=\beta u$ satisfies $g^T v + \frac{1}{2}\|A v \|^2_2=0$ for any $\beta > 0$. Hence, $\beta u$ is a locally feasible point of $(\wt P_\mu)$ for all small $\beta >0$ such that $J(\beta u) \ge J(u_*)=0$ for all small $\beta >0$. Since $J(\beta u)=\beta J(u)$ for all $\beta \ge 0$, we have $J(u) \ge 0$. Next consider a vector $u$ with $g^T u<0$. Clearly, $g^T( \beta u) + \frac{1}{2} \|A \beta u \|^2_2 < 0$ for all small $\beta>0$ so that $\beta u$ is a locally feasible point of $(\wt P_\mu)$. By a similar argument, we have $J(u) \ge 0$.
\end{proof}

\begin{proposition} \label{prop:BPDN_exact_reg}
 The problem $(BPDN_h): \min_{ \| A x - b\|_2 \le \sigma, \, \|x \|_1 \le f_*} \| x\|^2_2$ with $\|b\|_2>\sigma$ has a Lagrange multiplier if and only if there exist a constant $\mu \ge 0$ and a minimizer $x_*$ of $(BPDN_h)$ such that 
 \begin{itemize}
   \item [(i)] There exist $w \in \mathbb R^{|\Ical(x_*)|}_+$ with $\mathbf 1^T w =1$ and $v \in \mathbb R^m$ such that $x^* + \mu \sum_{i \in \Ical(x_*)} w_i p_i + A^T v=0$; and
   \item [(ii)] There exist $w' \in \mathbb R^{|\Ical(x_*)|}_+$ with $\mathbf 1^T w'=1$ and a constant $\gamma>0$ such that $(\mathbf 1^T w') x^* + \mu \sum_{i \in \Ical(x_*)} w'_i p_i + \gamma g=0$,
 \end{itemize}
 where $\mathbf 1$ denotes the vector of ones, $g:=A^T (A x_* - b)$, and $\Ical(x_*):= \{ i \, | \, p^T_i x_* = \|x_* \|_1 \}$. Furthermore, if $A$ has full row rank, then $(BPDN_h)$ has a Lagrange multiplier if and only if there exist a constant $\mu \ge 0$ and a minimizer $x_*$ of $(BPDN_h)$ such that
  \begin{itemize}
     \item [(ii')] There exist $\wh w \in \mathbb R^{|\Ical(x_*)|}_+$ with $\mathbf 1^T \wh w=1$ and a constant $\wh \gamma>0$ such that $x^* + \mu \sum_{i \in \Ical(x_*)} \wh w_i p_i + \wh \gamma g=0$.
  \end{itemize}
\end{proposition}

\begin{proof}
  It follows from Corollary~\ref{coro:P_h_Lagrange} 
   that $(BPDN_h)$ has  a Lagrange multiplier if and only if there exist a constant $\mu \ge 0$ and a minimizer $x_*$ of $(BPDN_h)$ such that $x_*$ is a minimizer of $(P_\mu)$. Note that any minimizer $x_*$ of $(BPDN_h)$ satisfies $\| A x_* -b\|_2 = \sigma$ such that $x_* \ne 0$ in light of $\|b\|_2 > \sigma$.
  By Lemma~\ref{lem:P_mu_mini}, we also deduce that $x_*$ is a  minimizer of $(P_\mu)$ if and only if $\| A x_* - b\|_2 = \sigma$ and the implication (\ref{eqn:P_mu_opt}) holds. Notice that the implication holds if and only if both the following linear inequalities have no solution:
  \[
     \mbox{(I)}: \quad A u = 0, \ \  \max_{i \in \Ical(x_*)} \big (x_* + \mu  p_i \big)^T u < 0; \qquad
     \mbox{(II)}: \quad g^T u < 0, \ \  \max_{i \in \Ical(x_*)} \big (x_* + \mu p_i \big)^T u < 0.
  \]
  By the Theorem of Alternative, we see that the inconsistency of the inequality (I) is equivalent to the existence of $(\wt w, \wt v)$ with $0 \ne \wt w \ge 0$ such that $\sum_{i \in \Ical(x_*)} \wt w_i (x_* + \mu p_i) + A^T \wt v=0$. Letting $w:= \wt w/(\mathbf 1^T \wt w)$ and $v:= \wt v/(\mathbf 1^T \wt w)$, we obtain condition (i). Similarly, the inconsistency of the inequality (II) is equivalent to the existence of $(\wt \gamma, \wt w')$ with $0 \ne (\wt \gamma, \wt w')\ge 0$ such that $\sum_{i \in \Ical(x_*)} \wt w'_i (x_* + \mu p_i) + \wt \gamma g=0$.
  Moreover, we deduce that $\wt \gamma >0$, since otherwise, we must have $0 \ne \wt w' \ge 0$ such that $0=x^T_* \sum_{i \in \Ical(x_*)} \wt w'_i (x_* + \mu p_i) = (\mathbf 1^T \wt w') ( \|x_*\|^2_2 + \mu \|x_*\|_1)$, where we use $p^T_i x_*=\|x_*\|_1$ for each $i \in \Ical(x_*)$, yielding a contradiction to $x_* \ne 0$.  Hence, by suitably scaling, we conclude that the inconsistency of the inequality (II) is equivalent to condition (ii). This completes the proof of the first part of the proposition.

 Suppose $A$ has full row rank. Then condition (i) holds trivially. Furthermore, since $A x_* - b \ne 0$ for a minimizer $x_*$ of $(BPDN_h)$, $g:=A^T (A x_* - b)$ is a nonzero vector. Hence, $w'$ in condition (ii) must be nonzero as $\gamma g \ne 0$. Setting $\wh w:= w'/(\mathbf 1^T w')$ and $\wh \gamma:=\gamma/(\mathbf 1^T w')$, we obtain condition (ii'), which is equivalent to condition (ii).
\end{proof}

By leveraging Proposition~\ref{prop:BPDN_exact_reg}, we construct the following example which shows that in general, the standard BPDN (\ref{eqn:BP_denoising01}) with the $\ell_1$-norm penalty is {\em not} exactly regularized by $h(x)=\| x\|^2_2$.

\begin{example} \rm
   Let $A =[ D \ D \ \cdots \ D]\in \mathbb R^{2\times N}$ with $N=2r$ for some $r\in \mathbb N$, where $D=\mbox{diag}(1, \beta) \in \mathbb R^{2\times 2}$ for a positive constant $\beta$. As before, we partition a vector $x \in \mathbb R^N$ as $x=(x^1, \ldots, x^r)$ where each $x^i\in \mathbb R^2$. Further, let $b=(b_1, b_2)^T \in \mathbb R^2$ and $\sigma=1$. We assume that $b \ge \mathbf 1$, which is a necessary and sufficient condition for $\| v - b\|_2 \le \sigma \Rightarrow v \ge 0$.

   We first consider the convex minimization problem: $\min_{u \in \mathbb R^2} \|u \|_1$ subject to $\|D u - b \|_2 \le 1$, which has a unique minimizer $u_*$ as $D$ is invertible for any $\beta >0$. Further, we must have $\|D u_* - b\|_2 =1$ and $u_* > 0$. In light of this, the necessary and sufficient optimality conditions for $u_*$ are: there exists $\lambda \in \mathbb R_+$ such that $\partial \| u_*\|_1 + \lambda D^T (D u_* - b)=0$, and $\|D u_* - b\|^2_2 =1$. Since $u_*>0$, we have $\lambda>0$ and the first equation becomes $\mathbf 1 + \lambda D^T (D u_* - b)=0$, which further gives rise to $Du_*= b - \frac{1}{\lambda} D^{-1} \mathbf 1$. Substituting it into the equation $\|D u_* - b\|_2 =1$, we obtain $\lambda = \frac{\sqrt{1+\beta^2}}{\beta}$. This yields $u_*=( b_1 - \frac{1}{\lambda}, \frac{1}{\beta} (b_2- \frac{1}{\beta\lambda }))^T$. Note that for all $\beta>0$, $0<\frac{1}{\lambda}<1$ and $\frac{1}{\beta \lambda} = \frac{1}{\sqrt{1+\beta^2}}$ so that $0<\frac{1}{\beta\lambda}<1$. Hence, $u_* >0$ in view of $b \ge \mathbf 1$.

   It can be shown that the solution set of the BPDN is given by
   \begin{align*}
      \mathcal H & = \{x_*=(x^1_*, \ldots, x^r_*) \, | \, \|x_*\|_1 = \| u_* \|_1, \ A x_* = Du_* \}
        = \{(x^1_*, \ldots, x^r_*) \, | \, \sum^r_{i=1} \|x^i_*\|_1 = \| u_* \|_1, \ \ \sum^r_{i=1} x^i_* = u_* \}  \\
        & = \{x_*=(x^1_*, \ldots, x^r_*) \, | \,  x^i_* = \lambda_i u_*,  \ \sum^r_{i=1} \lambda_i = 1, \ \lambda_i \ge 0, \forall \, i \}.
   \end{align*}
  Therefore, it is easy to show that the regularized BPDN with $h(x)=\|x \|^2_2$ has the unique minimizer $x_*=(x^i_*)$ with $x^*_i = \frac{u_*}{r}$ for each $i=1, \ldots, r$. Since $u_*>0$, we have $x_*>0$ such that $\Ical(x_*)$ is singleton with the single vector $p=\mathbf 1$. Since $A$ has full row rank, it follows from Proposition~\ref{prop:BPDN_exact_reg} that $(BPDN_h)$ has a Lagrange multiplier if and only if there exist constants $\mu \ge 0$, $\gamma>0$  such that $x_* + \mu  p + \gamma g=0$ for the unique minimizer $x_*$, where $p=\mathbf 1$ and $g=A^T (A x_* - b) = A^T (D u_* - b) = -\frac{1}{\lambda} \mathbf 1$, where $\lambda = \sqrt{1+\frac{1}{\beta^2}}$. Since $x_*=\frac{1}{r}(u_*, \ldots, u_*)$,  constants $\mu \ge 0$ and $\gamma>0$ exist if and only if $(u_*)_1=(u_*)_2$ or equivalently  $\beta (b_1 -\frac{1}{\lambda})=b_2 - \frac{1}{\beta \lambda}$. The latter is further equivalent to $b_2 = \beta b_1 + \frac{1-\beta^2}{\sqrt{1+\beta^2}}$. Hence, for any $\beta>0$, $(BPDN_h)$ has a Lagrange multiplier if and only if $b$ satisfies $b_2 = \beta b_1 + \frac{1-\beta^2}{\sqrt{1+\beta^2}}$ and $b \ge \mathbf 1$. The set of such $b$'s has zero measure in $\mathbb R^2$. For instance, when $\beta=1$, $(BPDN_h)$ has a Lagrange multiplier if and only if $b = \theta \cdot \mathbf 1$ for all $\theta \ge 1$.
  Thus the BPDN is {\em not} exactly regularized by $h(x)=\|x\|^2_2$ in general. \hspace{10pt} $\square$
\end{example}

%
\subsection{Exact Regularization of Grouped BP Problem  Arising From Group LASSO} \label{subsect:exact_reg_group_BP}


Motivated by  the group LASSO (\ref{eqn:group_Lasso}), we investigate exact regularization of the following BP-like problem: $\min \sum^p_{i=1} \|x_{\Ical_i} \|_2$ subject to $A x = b$, where $\{\Ical_i\}^p_{i=1}$ forms a disjoint union of $\{1, \ldots, N\}$. We call this problem the {\it grouped basis pursuit} or grouped BP. Here we set $\lambda_i$'s in the original group LASSO formulation (\ref{eqn:group_Lasso}) as one, without loss of generality. It is shown below that its exact regularization may fail.

\begin{example} \rm \label{example:grouped_BP}
Consider the grouped BP: $\min_{x, y \in \mathbb R^2} \|x\|_2+\|y \|_2$ subject to $ \begin{pmatrix} x_1 \\ x_2 \end{pmatrix} + \begin{pmatrix} \gamma y_1 \\ \beta y_2 \end{pmatrix} = b$, where $b=(b_1, b_2)^T\in \mathbb R^2$ is nonzero. Let $\gamma=0$ and $\beta>1$. Hence, $x^*_1= b_1$, $y^*_1=0$, and the grouped BP is reduced to $\min_{x_2, y_2} \sqrt{b^2_1 + x^2_2} + | y_2 |$ subject to $x_2 +\beta y_2 = b_2$, which is further equivalent to
\[
   (R_1): \quad \min_{x_2 \in \mathbb R} J(x_2):= \sqrt{b^2_1 + x^2_2} + \frac{|b_2 - x_2 |}{\beta}.
\]
It is easy to show that if $b_2 > \frac{b_1}{\sqrt{\beta^2-1}}>0$, then the above reduced problem attains the unique minimizer $x^*_2 = \frac{b_1}{\sqrt{\beta^2-1}}$ which satisfies $\nabla J(x^*_2)=0$. Hence, when $b_2 > \frac{b_1}{\sqrt{\beta^2-1}}>0$, the unique solution of the grouped BP is given by $x^*=(b_1, \frac{b_1}{\sqrt{\beta^2-1}})^T$ and $y^*=(0,  (b_2 - \frac{b_1}{\sqrt{\beta^2-1}})/\beta)^T$. Now consider the regularized problem for $\alpha>0$: $\min_{x, y \in \mathbb R^2} \|x\|_2+\|y \|_2 + \frac{\alpha}{2}\big( \|x\|^2_2 + \|y\|^2_2\big)$ subject to $ \begin{pmatrix} x_1 \\ x_2 \end{pmatrix} + \begin{pmatrix} 0 \\ \beta y_2 \end{pmatrix} = b$. Similarly, we must have $x^*_1= b_1$ and $y^*_1=0$ such that the reduced problem is given by
\[
 (R_2): \quad  \min_{x_2 \in \mathbb R} \sqrt{b^2_1 + x^2_2} + \frac{|b_2 - x_2 |}{\beta} + \frac{\alpha}{2}\Big( b^2_1 + x^2_2 + \frac{1}{\beta^2} (b_2 - x_2)^2 \Big).
\]
We claim that if $b_2 > \frac{b_1}{\sqrt{\beta^2-1}}>0$ with $b_2 \ne \frac{1+\beta^2}{\sqrt{\beta^2-1}} b_1$, then the exact regularization fails for any $\alpha>0$. We show this claim by contradiction. Suppose the exact regularization holds for some positive constant $\alpha$. Hence, $x^*_2= \frac{b_1}{\sqrt{\beta^2-1}}$ is the solution to the  reduced problem $(R_2)$. Since $\nabla J(x^*_2)=0$, we have
$\displaystyle
  \alpha \Big( x^*_2 + \frac{1}{\beta^2} \big( x^*_2 - b_2 \big) \Big) = 0.
$
This leads to $x^*_2 = \frac{b_2}{1+\beta^2}$, yielding a contradiction to $b_2 \ne \frac{1+\beta^2}{\sqrt{\beta^2-1}} b_1$. Hence, the exact regularization fails. \hspace{5in} $\square$
\end{example}

In spite of the failure of exact regularization in Example~\ref{example:grouped_BP}, it can be shown that the exact regularization holds for the following cases: (i) $\max(|\gamma|, |\beta|) < 1$; (ii) $\min(|\gamma|, |\beta|) > 1$;  (iii) $\gamma=0$, $\beta>1$, $b_1=0$, and $b_2 \ne 0$; and (iv) $\gamma=0$, $\beta=1$, and $b_1 \ne 0$. Especially, the first two cases hint that  the spectra of $A_{\bullet \Ical_i}$'s may determine exact regularization. Inspired by this example, we present certain sufficient conditions for which the exact regularization holds.

\begin{lemma}
 Consider a nonzero $b \in \mathbb R^m$ and a column partition $\{ A_{\bullet\Ical_i} \}^p_{i=1}$ of a matrix $A \in \mathbb R^{m\times N}$, where $\{ \Ical_i \}^p_{i=1}$ form a disjoint union of $\{1, \ldots, N\}$. Suppose $A_{\bullet\Ical_1}$ is invertible, $A^{-1}_{\bullet\Ical_1} A_{\bullet\Ical_i}$ is an orthogonal matrix for each $i=1, \ldots, s$, and
 $\| (A_{\bullet\Ical_i})^T (A_{\bullet\Ical_1})^{-T} A^{-1}_{\bullet\Ical_1} b\|_2 < \| A_{\bullet\Ical_1}^{-1} b\|_2$  for each $i=s+1, \ldots, p$.
%
%
 Then the exact regularization holds.
\end{lemma}

\begin{proof}
Since $A_{\bullet\Ical_1}^{-1} A b = A^{-1}_{\bullet\Ical_1} b$,
we may assume, without loss of generality, that $A_{\bullet\Ical_1}$ is the identity matrix. Hence, $A_{\bullet\Ical_i}$ is an orthogonal matrix for  $i=2, \ldots, s$.
 We claim that $x^*=(\frac{(A_{\bullet\Ical_1})^T b}{s}, \cdots, \frac{(A_{\bullet\Ical_s})^T b}{s}, 0, \ldots, 0)$ is an optimal solution to the grouped BP-like problem. Clearly, it satisfies the equality constraint. Besides, it follows from the KKT conditions that there exists a Lagrange multiplier $\lambda \in \mathbb R^m$ such that
 \[
    \frac{x^*_{\Ical_i}}{\|x^*_{\Ical_i}\|_2} + (A_{\bullet\Ical_i})^T \lambda =0, \quad \forall \ i=1, \ldots, s; \qquad 0 \in B_2(0, 1) + (A_{\bullet\Ical_j})^T \lambda, \quad \forall \ j=s+1, \ldots, p.
 \]
 Note that (i) $\lambda = -b/\|b\|_2$; and (ii) for each $j=s+1, \ldots, p$, $\|(A_{\bullet\Ical_j})^T \lambda \|_2 <1$ in view of $\|(A_{\bullet\Ical_j})^T b\|_2 < \|b\|_2$. Hence, $x^*$ is indeed a minimizer. Now consider the regularized grouped BP-like problem with the parameter $\alpha>0$. We claim that $x^*=(\frac{(A_{\bullet\Ical_1})^T b}{s}, \ldots, \frac{(A_{\bullet\Ical_s} )^T b}{s}, 0, \ldots, 0)$ is an optimal solution of the regularized problem for any sufficiently small $\alpha>0$. To see this, the KKT condition is given by
  \[
    \frac{x^*_{\Ical_i}}{\|x^*_{\Ical_i}\|_2} + \alpha x^*_{\Ical_i} +  (A_{\bullet\Ical_i})^T \wh\lambda =0, \quad \forall \ i=1, \ldots, s; \qquad 0 \in B_2(0, 1) + (A_{\bullet\Ical_j})^T \wh \lambda, \quad \forall \ j=s+1, \ldots, p.
 \]
 Hence, $\wh \lambda = -\big(\frac{1}{\|b\|_2} + \frac{\alpha }{s}\big) b$ such that $\|\wh \lambda\|_2 = 1 + \frac{\alpha}{s} \|b\|_2$. Since $\|(A_{\bullet\Ical_j})^T b\|_2< \|b\|_2$ for each $j=s+1, \ldots, p$, we have $\|(A_{\bullet\Ical_j})^T  \wh \lambda\|_2 = \big(\frac{1}{\|b\|_2} + \frac{\alpha }{s}\big) \|(A_{\bullet\Ical_j})^T   b \|_2  \le 1, \forall \, j=s+1, \ldots, p$ for all  sufficiently small $\alpha>0$. Hence, $x^*$ is a solution of the regularized problem for all small $\alpha>0$, and exact regularization holds.
\end{proof}

If the exact knowledge of $b \ne 0$ is unknown, the condition that $\| A^T_{\bullet\Ical_i} A^{-T}_{\bullet\Ical_1} A^{-1}_{\Ical_1} b\|_2 < \| A^{-1}_{\bullet\Ical_1} b\|_2$  for each $i=s+1, \ldots, p$ can be replaced by the following condition: $\| A^T_{\bullet\Ical_i} A^{-T}_{\bullet\Ical_1}\|_2 < 1$  for each $i=s+1, \ldots, p$.

%
\section{Dual Problems: Formulations and Properties} \label{sect:dual_problems}

We develop dual problems of the regularized BP as well as those of the LASSO and BPDN  in this section. These dual problems and their properties form a foundation for the development of column partition based distributed algorithms. As before, $\{ \Ical_i \}^p_{i=1}$ is a disjoint union of $\{1, \ldots, N\}$.

Consider the problems $(P_1)$-$(P_3)$ given by (\ref{eqn:gen_BP})-(\ref{eqn:gen_BPDN}), where $E \in \mathbb R^{r\times N}$ and $\|\cdot \|_\star$ is a general norm on $\mathbb R^r$. Let $\|\cdot \|_\diamond$ be the dual norm of $\|\cdot \|_\star$, i.e.,
$
   \| z \|_\diamond  :=  \sup \big \{ z^T v \, | \, \| v \|_\star \le 1 \big \}, \ \forall \, z \in \mathbb R^r.
$
As an example, the dual norm of the $\ell_1$-norm is the $\ell_\infty$-norm. When $\|x\|_\star :=\sum^p_{i=1} \|x_{\Ical_i}\|_2$ arising from the group LASSO, its dual norm is $\| z \|_\diamond = \max_{i=1, \ldots, p} \|z_{\Ical_i} \|_2$. Since the dual of the dual norm is the original norm, we have
$
    \|x \|_\star \, = \, \sup\big\{ x^T v \, | \, \| v\|_\diamond \le 1 \big\}, \forall \, x \in \mathbb R^r.
$
Further, let $B_{\diamond}(0, 1):=\{ v \, | \, \| v \|_\diamond  \le 1\}$ denote the closed unit ball centered at the origin with respect to $\| \cdot \|_\diamond$. Clearly, the subdifferential of $\|\cdot \|_\star$ at $x=0$ is  $B_{\diamond}(0, 1)$.

%
\subsection{Dual Problems: General Formulations} \label{subsect:general_results}

Strong duality will be exploited for the above mentioned problems and their corresponding dual problems. For this purpose, the following minimax result is needed.

\begin{lemma}  \label{lem:minimax}
Consider the convex program $(P): \inf_{z \in \mathcal P, A z = b, C z \le d} J(z)$, where $J(z):= \| E z \|_\star+ f(z) $, $f:\mathbb R^n \rightarrow \mathbb R$ is a convex function, $\mathcal P \subseteq \mathbb R^n$ is a polyhedral set,  $A, C, E$ are matrices, and $b, d$ are vectors. Suppose that $(P)$ is feasible and has a finite infimum. Then
\begin{eqnarray*}
 \lefteqn{ \inf_{z \in \Pcal} \Big( \sup_{y, \, \mu \ge 0, \, \| v\|_\diamond \le 1 }  \Big[ (E z)^T v + f(z) + y^T( A z  - b) + \mu^T(C z -d)  \Big]\Big)} \\
  &= &   \sup_{y, \, \mu \ge 0, \, \| v\|_\diamond \le 1 }   \Big( \inf_{z \in \Pcal} \Big[ (E z)^T v + f(z) + y^T( A z  - b) + \mu^T (C z - d)   \Big]\Big).
\end{eqnarray*}
\end{lemma}

\begin{proof}
Let $J_*>-\infty$ be the finite infimum of $(P)$. Since $\Pcal$ is polyhedral, it follows from \cite[Proposition 5.2.1]{Bertsekas_book99} that the strong duality holds, i.e., $J_*= \inf_{z \in \Pcal} \big[ \sup_{y, \mu \ge 0} J(z)+ y^T (A z - b) + \mu^T (C z - d)\big]= \sup_{y, \mu \ge 0} \big[\inf_{z \in \Pcal} J(z) + y^T (A z - b) + \mu^T (C z - d)\big]$, and the dual problem of $(P)$ attains an optimal solution $(y_*, \mu_*)$ with $\mu_* \ge 0$ such that $J_*= \inf_{z\in \Pcal} J(z)+ y^T_*(A z - b) + \mu^T_*(C z - d)$. Therefore,
 \begin{eqnarray*}
    J_* & = & \inf_{z\in \Pcal} \|E z \|_\star + f(z) + y^T_*(A z - b) + \mu^T_*(C z - d) \\
     &= & \inf_{z\in \Pcal}\Big( \sup_{ \| v\|_\diamond \le 1} \Big[ (E z)^T v + f(z) + y^T_*(A z - b) + \mu^T_*(C z - d) \Big] \Big) \\
    & = &  \sup_{ \| v\|_\diamond \le 1} \Big( \inf_{z\in \Pcal} \Big[ (E z)^T v + f(z) + y^T_*(A z - b) + \mu^T_*(C z - d) \Big] \Big) \\
    & \le & \sup_{y, \, \mu \ge 0, \, \| v\|_\diamond \le 1} \Big( \inf_{z\in \Pcal} \Big[  (E z)^T v + f(z) +
     y^T(A z - b) + \mu^T(C z - d)\Big] \Big) \\
    & \le & \inf_{z \in \Pcal} \Big( \sup_{y, \, \mu \ge 0, \, \| v\|_\diamond \le 1 }  \Big[ (E z)^T v + f(z) +
     y^T(A z - b) + \mu^T(C z - d)  \Big]\Big) \, = \, J_*,
 \end{eqnarray*}
 where the third equation follows from Sion's minimax theorem \cite[Corollary 3.3]{Sion_PJM58} and the fact that $B_{\diamond}(0, 1)$ is a convex compact set,
  and the second inequality is due to the weak duality. 
\end{proof}

In what follows, let $\mathcal C:=\{ x \in \mathbb R^N \, | \, C x \le d \}$ be a general polyhedral set unless otherwise stated, where $C \in \mathbb R^{\ell \times N}$ and $d \in \mathbb R^\ell$.

\gap

\noindent $\bullet$ {\bf Dual Problem of the Regularized BP-like Problem}  \
Consider the regularized BP-like problem for a fixed regularization parameter $\alpha>0$:
\begin{equation} \label{eqn:BP_primal}
  \min_{Ax = b, \, x \in \Ccal} \, \| E x \|_\star + \frac{\alpha}{2}\|x \|^2_2,
\end{equation}
where $b \in R(A) \cap A \mathcal C$ with $A \mathcal C:=\{ A x \, | \, x \in \mathcal C\}$. Let $\mu \in \mathbb R^\ell_+$ be the Lagrange multiplier for the polyhedral constraint $C x \le d$.
%
%
It follows from Lemma~\ref{lem:minimax} with $z=x$ and $\Pcal=\mathbb R^N$ that
\begin{eqnarray*}
\lefteqn{  \min_{A x = b, \, x \in \Ccal} \| E x \|_\star + \frac{\alpha}{2}\|x \|^2_2 } \\
 & =&
    \inf_{x} \Big( \sup_{ y, \, \mu \ge 0, \, \|v\|_\diamond \le 1} \Big[ (E x)^T v + \frac{\alpha}{2}\|x \|^2_2 + y^T( A x - b)+ \mu^T (C x - d) \, \Big]\Big) \\
  & = & \sup_{y, \, \mu \ge 0, \, \| v \|_\diamond \le 1} \Big( \inf_{x} \Big[ (E x)^T v + \frac{\alpha}{2}\|x \|^2_2  + y^T( A x - b) + \mu^T (C x - d)  \Big]  \Big) \\
 %
 %
  & = & \sup_{y, \, \mu \ge 0, \, \| v \|_\diamond \le 1} \Big( -b^T y-  \mu^T d + \sum^p_{i=1} \inf_{x_{\Ical_i}} \Big[ \frac{\alpha}{2} \|x_{\Ical_i}\|^2_2 + \big( ( A^T y + E^T v + C^T \mu )_{\Ical_i} \big)^T x_{\Ical_i} \Big] \Big) \\
  & = & \sup_{y, \, \mu \ge 0, \, \| v \|_\diamond \le 1} \Big( -b^T y - \mu^T d -
    \frac{1}{2\alpha} \sum^p_{i=1}  \big \| \big( A^T y + E^T v+ C^T \mu  \big)_{\Ical_i} \big\|^2_2 \Big),
 %
\end{eqnarray*}
This leads to the equivalent dual problem:
\begin{equation} \label{eqn:BP_dual}
  \mbox{(D)}: \quad \min_{y, \, \mu \ge 0, \, \| v \|_\diamond \le 1} \, \Big( b^T y + d^T \mu +
  \frac{1}{2\alpha} \sum^p_{i=1}  \big \| ( A^T y + E^T v + C^T \mu)_{\Ical_i} \big\|^2_2  \, \Big).
\end{equation}
%
%
Let $(y_*, \mu_*, v_*)\in \mathbb R^m \times \mathbb R^\ell_+ \times B_\diamond(0, 1)$ be an optimal solution of the dual problem; its existence is shown in the proof of Lemma~\ref{lem:minimax}. Consider the Lagrangian $L(x, y, \mu, v):=(E x)^T v + \frac{\alpha}{2}\|x \|^2_2 + y^T( A x - b)+\mu^T (Cx - d)$. Then by the strong duality given in Lemma~\ref{lem:minimax}, we  see from $\nabla_x L(x_*, y_*, \mu_*, v_*)=0$ that the unique optimal solution $x^*=(x^*_{\Ical_i})^p_{i=1}$ of (\ref{eqn:BP_primal}) is given by
\[
  x^*_{\Ical_i} \, = \,  -\frac{1}{\alpha} \Big( A^T y_* + E^T v_* + C^T \mu_* \Big)_{\Ical_i}, \qquad \forall \ i=1, \ldots, p.
\]


\noindent $\bullet$ {\bf Dual Problem of the LASSO-like Problem} \ Consider the  LASSO-like problem for $A \in \mathbb R^{m\times N}$, $b \in \mathbb R^m$, and $E \in \mathbb R^{r\times N}$:
%
%
\begin{equation} \label{eqn:Lasso_primal}
  \min_{x \in \Ccal} \, \frac{1}{2} \|A  x - b\|^2_2 +  \| E x \|_\star.
\end{equation}
It follows from Lemma~\ref{lem:minimax} with $z=(x, u)$ and $\Pcal=\mathbb R^N\times \mathbb R^m$ that
\begin{eqnarray*}
 \lefteqn{ \min_{x \in \Ccal} \frac{1}{2} \|A x - b\|^2_2+ \| E x \|_\star  \, = \, \inf_{ x \in \Ccal, \, u=A x -b } \frac{\|u \|^2_2}{2}+  \| E x \|_\star }  \\
 %
 %
  & = & \inf_{x, u} \Big( \sup_{y, \, \mu \ge 0, \,  \| v \|_\diamond \le 1} \Big \{  \, \frac{\|u \|^2_2}{2} +   (E x)^T v  + y^T( A x - b-u) +  \mu^T (Cx - d) \,   \Big\} \Big) \\
  & = & \sup_{y, \, \mu \ge 0, \, \| v \|_\diamond \le 1} \Big( \inf_{x, u} \,   \frac{\|u \|^2_2}{2} + (E x)^T v   + y^T( A x - b-u) +  \mu^T (Cx - d)  \,  \Big) \\
  & = & \sup_{y, \, \mu \ge 0, \, \| v \|_\diamond \le 1} \Big( -b^T y - \mu^T d + \inf_u  \Big( \frac{\|u \|^2_2}{2} - y^T u \Big) +  \sum^p_{i=1} \inf_{x_{\Ical_i}}  \big[ ( A^T y + E^T v+ C^T \mu )_{\Ical_i}\big]^T  x_{\Ical_i} \Big), \\
  & = & \sup_{ y, \, \mu \ge 0, \, \| v \|_\diamond \le 1} \Big\{  -b^T y- \frac{\|y \|^2_2}{2} -\mu^T d \, : \,  ( A^T y + E^T v + C^T \mu)_{\Ical_i} = 0, \ i=1, \ldots, p  \Big\}.
%
\end{eqnarray*}
This yields the equivalent dual problem
\begin{equation} \label{eqn:Lasso_dual}
  \mbox{(D)}: \quad \min_{y, \, \mu \ge 0, \, \| v \|_\diamond \le 1} \, \Big\{ \frac{\|y \|^2_2}{2} + b^T y + d^T \mu \, : \, ( A^T y + E^T v +C^T \mu)_{\Ical_i} = 0, \ i=1, \ldots, p \, \Big\}.
\end{equation}
By Lemma~\ref{lem:minimax}, the dual problem attains an optimal solution $(y_*, \mu_*, v_*)\in \mathbb R^m \times \mathbb R^\ell_+ \times B_\diamond(0, 1)$. Since the objective function of (\ref{eqn:Lasso_dual}) is strictly convex in $y$ and convex in $(\mu, v)$, $y_*$ is unique (but $(\mu_*, v_*)$ may not).
%
%

The following lemma establishes a connection between a primal solution and a dual solution, which is critical to distributed algorithm development.

\begin{lemma} \label{lem:primal_dual_Lasso}
      Let $(y_*, \mu_*, v_*)$ be an optimal solution to the dual problem (\ref{eqn:Lasso_dual}). Then for any optimal solution $x_*$ of the primal problem (\ref{eqn:Lasso_primal}),  $A x_* - b= y_*$. Further, if $\mathcal C$ is a polyhedral cone (i.e., $d=0$), then  $ \|E x_*\|_\star=  -(b+ y_*)^T y_*$.
\end{lemma}

\begin{proof}
Consider the equivalent primal problem for (\ref{eqn:Lasso_primal}): $\min_{x \in \Ccal, A x - b=u} \frac{1}{2} \| u \|^2_2+  \| E x\|_\star$, and let $(x_*, u_*)$ be its optimal solution.
%
%
%
Consider the Lagrangian
\[
   L(x, u, y, \mu, v) \, := \, \frac{\|u \|^2_2}{2} +  (E x)^T v  + y^T( A x - b-u) + \mu^T ( C x -d).
\]
In view of the strong duality shown in Lemma~\ref{lem:minimax},  $(x_*, u_*, y_*,\mu_*,  v_*)$ is a saddle point of $L$. Hence,
\begin{eqnarray*}
  L(x_*, u_*, y, \mu, v) \ \le \ L(x_*, u_*,  y_*, \mu_*, v_* ), \ \ & \quad \forall \ y \in \mathbb R^m, \ \mu \in \mathbb R^\ell_+, \  v \in B_\diamond(0, 1); \\
  L(x_*, u_*,  y_*, \mu_*, v_*) \ \le \ L(x, u,  y_*, \mu_*, v_*), \quad & \forall \ x \in \mathbb R^N, \ u \in \mathbb R^m.
\end{eqnarray*}
The former inequality implies that $\nabla_y L(x_*, u_*, y_*, \mu_*, v_*)=0$ such that $A x_* - b - u_*=0$; the latter inequality shows that $ \nabla_u L(x_*, u_*, y_*, \mu_*, v_*)=0$, which yields $u_* - y_*=0$. These results lead to $A x_* - b = y_*$.
Lastly, when $d=0$, it follows from the strong duality that $\frac{1}{2} \| A x_* - b \|^2_2+ \| E x_*\|_\star = - b^T y_* - \frac{1}{2}\|y_*\|^2_2$. Using $A x_* - b = y_*$, we have
$
 \| E x_* \|_\star = -b^T y_* - \| y_* \|^2_2 = -(b+ y_*)^T y_*.
$
\end{proof}


\noindent $\bullet$ {\bf Dual Problem of the BPDN-like Problem} \
Consider the BPDN-like problem with $\sigma>0$:
\begin{equation} \label{eqn:BPDN_primal}
  \min_{x \in \Ccal, \, \| Ax - b\|_2 \le \sigma} \, \| E x \|_\star \, = \, \inf_{ x \in \Ccal, \, u=A x -b, \, \| u\|_2 \le \sigma } \| E x \|_\star,
\end{equation}
where we assume that the problem is feasible and has a positive optimal value,
$\| b\|_2 > \sigma$, and  the polyhedral set $\Ccal$ satisfies $0\in \Ccal$. Note that $0\in \Ccal$ holds if and only if $d \ge 0$.

To establish the strong duality, we also assume that there is an $\wt x$ in the relative interior of $\Ccal$ (denoted by $\mbox{ri}(\Ccal)$) such that $\| A \wt x - b\|_2 < \sigma$ or equivalently, by \cite[Theorem 6.6]{Rockafellar_book70}, there exits $\wt u \in A ( \mbox{ri}(\Ccal))-b$ such that $\| \wt u \|_2 < \sigma$.
A sufficient condition for this assumption to hold is that $b \in A ( \mbox{ri}(\Ccal))$.
Under this assumption, it follows from \cite[Theorem 28.2]{Rockafellar_book70} that there exist $y_* \in \mathbb R^m$,  $\mu_* \ge 0$, and $\lambda_* \ge 0$ such that
%
$
\inf_{ x \in \Ccal, \, u=A x -b, \| u\|_2 \le \sigma } \| E x \|_\star = \inf_{x \in \Ccal, u} \| Ex \|_\star +y^T_*(A x - b - u) + \lambda_*(\|u\|^2_2-\sigma^2) +\mu_*^T(C x - d)$. By the similar argument for Lemma~\ref{lem:minimax}, we have
\begin{eqnarray*}
 \lefteqn{ \min_{x \in \Ccal, \| Ax - b\|_2 \le \sigma}  \| E x \|_\star \ = \ \inf_{ x \in \Ccal, \, u=A x -b, \| u\|_2 \le \sigma } \| E x \|_\star } \\
  & = & \inf_{x \in \Ccal, u} \Big( \sup_{y, \, \mu \ge 0, \, \| v \|_\diamond \le 1, \, \lambda \ge 0} \Big \{  \,  (E x)^T v  +  \lambda ( \| u \|^2_2 - \sigma^2) + y^T( A x - b-u) + \mu^T( C x -d) \,   \Big\} \Big) \\
  & = & \sup_{y, \, \mu \ge 0, \, \| v \|_\diamond \le 1, \lambda \ge 0} \Big( \inf_{x, u} \,  (E x)^T v  +  \lambda ( \| u \|^2_2 - \sigma^2) + y^T( A x - b-u) + \mu^T( C x -d) \, \Big)\\
 & \triangleq & \sup_{y, \, \mu \ge 0, \, \| v \|_\diamond \le 1, \lambda > 0} \Big( \inf_{x, u} \,  (E x)^T v  +  \lambda ( \| u \|^2_2 - \sigma^2) + y^T( A x - b-u) +  \mu^T( C x -d) \, \Big)\\
  & = & \sup_{y, \mu\ge 0, \| v \|_\diamond \le 1, \lambda > 0} \Big( -b^T y - \mu^T d - \lambda \sigma^2 + \inf_u  \big( \lambda\|u\|^2_2 - y^T u \big) +  \sum^p_{i=1} \inf_{x_{\Ical_i}}  \big( ( A^T y + E^T v + C^T \mu)_{\Ical_i} \big)^T  x_{\Ical_i} \Big] \Big) \\
  & \circeq & \sup_{y, \mu\ge 0, \| v \|_\diamond \le 1, \lambda > 0} \Big\{ -b^T y - \mu^T d - \lambda \sigma^2 - \frac{\|y \|^2_2}{4\lambda} \, : \,  ( A^T y + E^T v + C^T \mu)_{\Ical_i}  = 0, \ i=1, \ldots, p \, \Big\} \\
 %
 %
  & = & \sup_{y, \, \mu \ge 0, \, \| v \|_\diamond \le 1} \Big\{ -b^T y - \sigma \|y \|_2  -\sum^p_{i=1} \mu^T_i d_i \, : \,  ( A^T y + E^T v + C^T \mu)_{\Ical_i}  = 0, \ i=1, \ldots, p \, \Big\}.
\end{eqnarray*}
%
%
Here the reason for letting $\lambda>0$ in the 4th equation (marked with $\triangleq$) is as follows: suppose $\lambda=0$, then
\begin{eqnarray*}
\lefteqn{\sup_{y, \mu \ge 0, \| v \|_\diamond \le 1} \Big\{ \inf_{x} \Big[ (Ex)^T v + y^T (A x -b) +  \mu^T(Cx - d) \Big] + \inf_u \, y^T (- u) \Big\} } \\
%
%
& = & \sup_{y=0, \, \mu \ge 0, \, \| v \|_\diamond \le 1} \Big\{ \inf_{x} \Big[ (Ex)^T v + y^T (A x -b) +  \mu^T(Cx - d) \Big] + \inf_u \,y^T (- u) \Big\} \\
& = & \sup_{ \mu \ge 0, \, \| v \|_\diamond \le 1} \Big( \inf_x \Big[ (Ex)^T v + \mu^T(Cx - d) \Big] \Big) \\
& \le &  \inf_x \Big( \sup_{ \mu \ge 0, \, \| v \|_\diamond \le 1} \Big[ (Ex)^T v + \mu^T(Cx - d) \Big] \Big) \, = \, \inf_{x \in \Ccal} \| E x \|_\star \le 0,
\end{eqnarray*}
where we use the fact that $0\in \Ccal$. This shows that  the positive optimal value cannot be achieved when $\lambda=0$, and thus the constraint $\lambda \ge 0$ in the 3rd equation can be replaced by $\lambda>0$ without loss of generality. Besides, in the second-to-last equation (marked with $\circeq$), the constraint $y$ can be replaced with $y\ne 0$ because otherwise, i.e., $y=0$, then we have, in light of $\mu\ge 0$ and $d \ge 0$,
\begin{eqnarray*}
 \sup_{y=0, \mu\ge 0, \| v \|_\diamond \le 1, \lambda > 0} \Big\{ -b^T y -\mu^T d - \lambda \sigma^2 - \frac{\|y \|^2_2}{4\lambda} \, : \,  ( A^T y + E^T v + C^T \mu)_{\Ical_i}  = 0, \ i=1, \ldots, p \, \Big\}  \\
 =  \sup_{\mu \ge 0, \, \| v \|_\diamond \le 1, \lambda > 0} \Big\{   - \lambda \sigma^2-\mu^T d \, : \,  ( E^T v + C^T \mu)_{\Ical_i} = 0, \ i=1, \ldots, p \, \Big\} \, \le \, 0, \qquad
\end{eqnarray*}
which cannot achieve the positive optimal value.  Hence, we only consider $y\ne 0$ whose correponding optimal $\lambda_*=\frac{\| y \|_2 }{2\sigma}$ in the second-to-last equation is indeed positive and thus satifies the constraint $\lambda>0$. This gives rise to the equivalent dual problem
\begin{equation} \label{eqn:BPDN_dual}
  \mbox{(D)}: \quad \min_{y, \, \mu \ge 0, \, \| v \|_\diamond \le 1} \, \Big\{ \, b^T y + \sigma \|y \|_2  +
  d^T \mu \, : \, ( A^T y + E^T v + C^T \mu)_{\Ical_i} = 0, \ i=1, \ldots, p \, \Big\}.
\end{equation}

%
%

By the similar argument for Lemma~\ref{lem:minimax}, the dual problem attains an optimal solution $(y_*, \mu_*, v_*)\in \mathbb R^m \times \mathbb R^\ell_+ \times B_\diamond(0, 1)$ along with $\lambda_*\ge 0$.
%
%
The following lemma establishes certain solution properties of the dual problem and a connection between a primal solution and a dual solution, which is crucial to distributed algorithm development. Particularly, it shows that the $y$-part of a dual solution is unique when $\mathcal C$ is a polyhedral cone.

\begin{lemma} \label{lem:primal_dual_BPDN}
 Consider the BPDN (\ref{eqn:BPDN_primal}), where $\| b\|_2 > \sigma$, $0\in \Ccal$, and its optimal value is positive. Assume that the strong duality holds. The following hold:
 \begin{itemize}
   \item [(i)] Let $(y_*,  \mu_*, v_*)$ be a dual solution of (\ref{eqn:BPDN_dual}). Then $y_* \ne 0$, and
        for any solution $x_*$ of (\ref{eqn:BPDN_primal}),  $A x_* - b= \frac{\sigma y_*}{\|y_*\|_2}$. Further, if $\Ccal$ is a polyhedral cone (i.e., $d=0$),
        %
        %
        then  $ \|E x_*\|_\star= -b^T y_* - \sigma \|y_*\|_2$.
    \item [(ii)] Suppose $d=0$. Let $(y_*, \mu_*, v_*)$ and $(y'_*, \mu'_*, v'_*)$ be two arbitrary solutions of  (\ref{eqn:BPDN_dual}). Then $y_* = y'_*$.
  \end{itemize}
\end{lemma}

\begin{proof}
(i) Consider the equivalent primal problem for (\ref{eqn:BPDN_primal}): $\min_{x \in \Ccal,\, A x = b=u, \, \|u\|_2 \le \sigma } \| E x\|_\star$, and let $(x_*, u_*)$ be its optimal solution.
For a dual solution $(y_*,  \mu_*, v_*)$, we deduce that $y_* \ne 0$ since otherwise, we have $-(b^T y_* + \sigma \| y_*\|_2 + d^T \mu_* ) \le 0$, which contradicts its positive optimal value  by the strong duality.

Consider the Lagrangian
%
%
\[
   L(x, u, y, \mu,  v, \lambda ) \, := \,  (E x)^T v  + y^T( A x - b-u) + \lambda (\|u\|^2_2 - \sigma^2) +  \mu^T( C x - d). 
\]
 By the strong duality, $(x_*, u_*, y_*, \mu_*, v_*, \lambda_*)$ is a saddle point of $L$ such that
\begin{eqnarray*}
  L(x_*, u_*, y, \mu, v, \lambda) \ \le \ L(x_*, u_*,  y_*, \mu_*,  v_*, \lambda_* ), \ \ & \quad \forall \ y \in \mathbb R^m, \ \mu \in \mathbb R^\ell_+, \ v \in B_\diamond(0, 1), \ \lambda \in \mathbb R_+; \\
  L(x_*, u_*,  y_*, \mu_*, v_*, \lambda_*) \ \le \ L(x, u,  y_*, \mu_*, v_*, \lambda_*), \quad & \forall \ x \in \mathbb R^N, \ u \in \mathbb R^m.
\end{eqnarray*}
The former inequality implies that $\nabla_y L(x_*, u_*, y_*, \mu_*, v_*, \lambda_*)=0$, yielding $A x_* - b - u_*=0$, and the latter shows that $ \nabla_u L(x_*, u_*, y_*, \mu_*, v_*, \lambda_*)=0$, which gives rise to $2 \lambda_* u_* = y_*$. Since $y_* \ne 0$, we have $\lambda_*>0$ which implies $\|u_*\|_2 - \sigma=0$ by the complementarity relation. It thus follows from $2 \lambda_* u_* = y_*$ and $\|u_*\|_2 = \sigma$ that $\lambda_*=\frac{\|y_*\|_2}{2 \sigma}$. This leads to $u_* = \frac{y_*}{2 \lambda_*}=  \frac{\sigma y_*}{\|y_*\|_2}$. Therefore, $A x_* - b = u_*= \frac{\sigma y_*}{\|y_*\|_2} $.
Finally, when $d=0$, we deduce via the strong duality that $ \| E x_*\|_\star = - b^T y_* - \sigma \|y_*\|_2$.

(ii) Suppose $d=0$. Let $(y_*, \mu_*, v_*)$ and $(y'_*, \mu'_*, v'_*)$ be two solutions of  the dual problem (\ref{eqn:BPDN_dual}), where $y_* \ne 0$ and $y'_* \ne 0$. Then $b^T y_* + \sigma \| y_*\|_2 = b^T y'_* + \sigma \|y'_*\|_2=-\|E x_*\|_\star<0$. Therefore, $\|y_*\|_2 \big( b^T \frac{y_*}{\|y_*\|_2} + \sigma \big) = \|y'_*\|_2 \big( b^T \frac{y'_*}{\|y'_*\|_2} + \sigma \big)$, and $b^T \frac{y_*}{\|y_*\|_2} + \sigma < 0$.
It follows from Proposition~\ref{prop:soln_property} that for any solution $x_*$ of the primal problem (\ref{eqn:BPDN_primal}), $A x_* - b$ is constant. By the argument for Part (i), we have $A x_*-b=u_*$ and $A x'_*-b=u'_*$ such that $u_*=u'_*$, and $u_*= \frac{\sigma y_*}{\|y_*\|_2}$ and $u'_*= \frac{\sigma y'_*}{\|y'_*\|_2}$.
Hence, $\frac{y_*}{\|y_*\|_2} = \frac{y'_*}{\|y'_*\|_2}$ such that $b^T \frac{y_*}{\|y_*\|_2} + \sigma = b^T \frac{y'_*}{\|y'_*\|_2} + \sigma  < 0$. In light of $\|y_*\|_2 \big( b^T \frac{y_*}{\|y_*\|_2} + \sigma \big) = \|y'_*\|_2 \big( b^T \frac{y'_*}{\|y'_*\|_2} + \sigma \big)$, we have $\|y_*\|_2 = \|y'_*\|_2$. Using $\frac{y_*}{\|y_*\|_2} = \frac{y'_*}{\|y'_*\|_2}$ again, we obtain $y_* = y'_*$.
\end{proof}

\begin{remark} \rm  \label{remark:decoupled_polyhedral_set}
The above dual problem formulations for a general polyhedral set $\Ccal$ are useful for distributed computation when $\ell \ll N$, even if $C \in \mathbb R^{\ell \times N}$ is a dense matrix; see Section~\ref{sect:distributed_schemes}. When both $N$ and $\ell$ are large, e.g., $\Ccal =\mathbb R^N_+$, decoupling properties of $\Ccal$ are preferred. In particular,
%
%
consider
the following polyhedral set of certain decoupling structure:
\begin{equation}
   \mathcal C \, := \, \{ x =(x_{\Ical_i})^p_{i=1} \in \mathbb R^N \, | \, C_{\Ical_i} \, x_{\Ical_i} \le d_{\Ical_i}, \ \ i=1, \ldots, p \, \},  \label{eqn:polyhedral_set_C}
\end{equation}
where $C_{\Ical_i} \in \mathbb R^{\ell_i \times |\Ical_i|}$ and $d_{\Ical_i} \in \mathbb R^{\ell_i}$ for each $i=1, \ldots, p$. Let $\ell:=\sum^p_{i=1} \ell_i$. Also, let $\mu=(\mu_{\Ical_i})^p_{i=1}$ with $\mu_{\Ical_i} \in \mathbb R^{\ell_i}_+$ be the Lagrange multiplier for $\mathcal C$. The dual problems in (\ref{eqn:BP_dual}), (\ref{eqn:Lasso_dual}), and (\ref{eqn:BPDN_dual}) can be easily extended to the above set $\Ccal$ by replacing $\mu^T d$ with $\sum^p_{i=1} \mu^T_{\Ical_i} d_{\Ical_i}$ and $( A^T y + E^T v + C^T \mu)_{\Ical_i}$ with $( A^T y + E^T v )_{\Ical_i} + C^T_{\Ical_i} \mu_{\Ical_i}$, respectively.
For example, the dual problem of the regularized problem (\ref{eqn:BP_primal}) is:
\[
  \mbox{(D)}: \quad \min_{y, \, \mu \ge 0, \, \| v \|_\diamond \le 1} \, \Big( b^T y +  \sum^p_{i=1} \mu^T_{\Ical_i} d_{\Ical_i} +
   \frac{1}{2\alpha} \sum^p_{i=1} \big \| ( A^T y + E^T v )_{\Ical_i} + C^T_{\Ical_i} \mu_{\Ical_i} \big\|^2_2 \, \Big).
\]
Moreover, letting $(y_*, \mu_*, v_*)$ be a dual solution, the unique primal solution $x^*=(x^*_{\Ical_i})^p_{i=1}$ is given by
$
  x^*_{\Ical_i} \, = \,  -\frac{1}{\alpha} \Big[\big( A^T y_* + E^T v_*  \big)_{\Ical_i} + C^T_{\Ical_i} (\mu_*)_{\Ical_i} \Big], \forall \, i=1, \ldots, p.
$
Further, Lemmas~\ref{lem:primal_dual_Lasso} and \ref{lem:primal_dual_BPDN} also hold for a primal solution $x_*$ and a dual solution $y_*$.
\end{remark}

\noindent $\bullet$ {\bf Reduced Dual Problems for Box Constraints} \
Consider the box constraint set $\mathcal C:=  [l_1, u_1] \times \cdots \times [l_N, u_N]$, where $-\infty \le l_i < u_i \le +\infty$ for each $i=1, \ldots, N$. We assume $0 \in \mathcal C$ or equivalently $l_i \le 0 \le u_i$ for each $i$,  which often holds for sparse signal recovery. We may write $\mathcal C=\{ x \in \mathbb R^N \, | \, \mathbf l \le x \le \mathbf u \}$, where $\mathbf l :=(l_1, \ldots, l_N)^T$ and $\mathbf u :=(u_1, \ldots, u_N)^T$.
The dual problems for such $\Ccal$ can be reduced by removing the dual variable $\mu$ as shown below.

For any given $l_i \le 0 \le u_i$ with $l_i < u_i$ for $i=1, \ldots, N$, define the function $\theta_i:\mathbb R \rightarrow \mathbb R$ as
\begin{equation} \label{eqn:q_i}
   \theta_i(t) \, := \,  \  t^2- (t - \Pi_{[l_i, u_i]}(t))^2 = t^2 - \big( \min ( t - l_i, \, (u_i - t)_- ) \big)^2, \qquad \ \forall \, t \in \mathbb R.
\end{equation}
%
%
%
Hence, $\theta_i$ is $C^1$ and convex \cite[Theorem 1.5.5, Exercise 2.9.13]{FPang_book03}, and $\theta_i$ is increasing on $\mathbb R_+$ and decreasing on $\mathbb R_-$, and its minimal value   on $\mathbb R$ is zero. When $\Ccal=\mathbb R^N$, $\theta_i(s)=s^2, \forall \, i$; when $\Ccal=\mathbb R^N_+$, $\theta_i(s)=(s_+)^2, \forall \, i$.

Define the index sets $\Lcal_\infty:=\{ i \, | \, l_i=-\infty, \ u_i=+\infty\}$, $\Lcal_+:=\{ i \, | \, l_i \mbox{ is finite}, \ u_i=+\infty\}$,  $\Lcal_-:=\{ i \, | \, l_i=-\infty, \ u_i \mbox{ is finite} \}$, and $\Lcal_b :=\{1, \ldots, N\} \setminus (\Lcal_\infty \cup \Lcal_+\cup \Lcal_-)$. Further, define the polyhedral cone
\[
  \Kcal \, := \big\{ (y, v) \in \mathbb R^m \times \mathbb R^r \ | \ (A^T y + E^T v )_{\Lcal_\infty}=0, \  (A^T y + E^T v )_{\Lcal_+} \ge 0, \ (A^T y + E^T v )_{\Lcal_-} \le 0 \big\}, 
\]
and the extended real valued convex PA function
\begin{eqnarray}
   g(y, v)  & \, := & \sum_{i\in \Lcal_+} (- l_i) \cdot \big[ ( A^T y + E^T v)_i \big]_+ + \sum_{i\in \Lcal_-} u_i \cdot \big[ ( A^T y + E^T v)_i \big]_-  \notag \\ 
     & & \quad + \sum_{i\in \Lcal_b} \Big\{ (- l_i) \cdot \big[ ( A^T y + E^T v)_i \big]_+ +  u_i \cdot \big[ ( A^T y +  E^T v)_i \big]_- \Big\}, \quad \forall \ (y, v) \in \Kcal, \notag
\end{eqnarray}
and $g(y, v) :=+\infty$ for each $(y, v) \not\in \Kcal$. Note that $g(y, v)\ge 0, \forall \, (y, v) \in \Kcal$. When the box constraint set $\Ccal$ is a cone, then $\Kcal=\{ (y, v) \, | \, A^T y + E^T v \in \mathcal C^*\}$ (where $\Ccal^*$ is the dual cone of $\Ccal$), and the corresponding $g(y, v)=0$ for all $(y, v) \in \mathcal K$. Using these results, we obtain the following  reduced dual problems:

(i) The  dual of the regularized BP-like problem (\ref{eqn:BP_primal}):
%
%
$
 \min_{y, \| v \|_\diamond \le 1} \, b^T y +
   \frac{\alpha}{2} \sum^N_{i=1} \theta_i \big( -\frac{1}{\alpha} \big( A^T y + E^T v \big)_{i} \big).
$
%
%

(ii) The  dual of the  LASSO-like problem (\ref{eqn:Lasso_primal}):
$
 \min_{ \| v\|_\diamond \le 1, \, (y, v) \in  \Kcal } \big( b^T y + \frac{\|y \|^2_2}{2} + g(y, v) \big).
$

(iii) Under similar assumptions, the  dual of the  BPDN-like problem (\ref{eqn:BPDN_primal}): \\
$
  \min_{ \| v\|_\diamond \le 1, \, (y, v) \in  \Kcal } \big( b^T y + \sigma \| y\|_2 + g(y, v) \big).
$

\gap

The dual problems developed in this subsection can be further reduced or simplified for specific norms or polyhedral constraints. This will be shown in the subsequent  subsections.

%
\subsection{Applications to the $\ell_1$-norm based Problems} \label{subsect:dual_problem_L1}

Let $\|\cdot \|_\star$ be the $\ell_1$-norm; its dual norm is the $\ell_\infty$-norm. As before, $\Ccal$ is a general polyhedral set defined by $C x \le d$ unless otherwise stated.
%
%
%

\gap

\noindent $\bullet$ {\bf Reduced Dual Problem of the Regularized BP-like Problem} \ Consider two cases as follows:

{\bf Case (a)}: $E=I_N$.  The dual variable $v$ in (\ref{eqn:BP_dual}) can be removed using the soft thresholding or shrinkage operator $S_\kappa:\mathbb R \rightarrow \mathbb R$ with the parameter $\kappa>0$ given by
\[
  S_\kappa (s) \, : = \, \argmin_{t \in \mathbb R} \frac{1}{2}(t - s)^2 + \kappa | t| = \left\{ \begin{array}{ll} s - \kappa, & \mbox{ if } s \ge \kappa \\ 0, & \mbox{ if } s \in [-\kappa, \kappa] \\ s + \kappa, & \mbox{ if } s \le -\kappa \end{array} \right.
\]
When $\kappa=1$, we write $S_\kappa(\cdot)$ as $S(\cdot)$ for notational convenience. It is known that $S^2(\cdot)$ is convex and $C^1$.
Further, for a vector $v=(v_1, \ldots, v_k)^T \in \mathbb R^k$, we let $S(v):=(S(v_1), \ldots, S(v_k))^T \in \mathbb R^n$.
In view of $\min_{|t|\le 1}(t -s)^2 = S^2(s), \forall \, s \in \mathbb R$ whose optimal solution is given by
$t_*=\Pi_{[-1, 1]}(s)$, the dual problem (\ref{eqn:BP_dual}) reduces to
\begin{equation} \label{eqn:BP_like_1norm}
  \mbox{(D)}: \quad \min_{y, \, \mu \ge 0} \, \Big( b^T y + \mu^T d +
  \frac{1}{2\alpha} \sum^p_{i=1} \|S\big( - (A^T y + C^T \mu )_{\Ical_i} \big)\|^2_2  \, \Big).
\end{equation}
Letting $(y_*, \mu_*)$ be an optimal solution of the above reduced dual problem, it can be shown via the strong duality that for each $\Ical_i$, $(v_*)_{\Ical_i} = \psi \big( - (A^T y_* + C^T \mu_*)_{\Ical_i} \big)$, where $\psi(v):=(\Pi_{[-1, 1]}(v_1), \ldots, \Pi_{[-1, 1]}(v_k))$ for $v\in \mathbb R^k$. Thus the unique primal solution $x^*$ is given by
\begin{equation} \label{eqn:solution_x*_L1}
   x^*_{\Ical_i}  = -\frac{1}{\alpha} \Big[ ( A^T y_* + C^T \mu_* )_{\Ical_i}  + \psi \big( - (A^T y_*+ C^T \mu_*)_{\Ical_i} \big)  \Big]=  -\frac{1}{\alpha} S\Big(( A^T y_* + C^T \mu_* )_{\Ical_i} \Big),   \quad \forall \, i=1, \ldots, p.
\end{equation}
When $\Ccal$ is a box constraint set, the equivalent dual problem further reduces to
\begin{equation} \label{eqn:BP_dual_box_constraint_L1}
\mbox{(D)}: \quad    \min_{y \in \mathbb R^m} \Big[ b^T y + \frac{\alpha}{2}  \sum^N_{i=1} \theta_i \circ \Big( - \frac{1}{\alpha}  S \Big( \big( A^T y \big)_i \Big) \Big) \, \Big].
\end{equation}
Letting $y_*$ be a dual solution, the unique primal solution $x^*$ is given by 
\[
   x^*_i \, = \,  \max\left\{ l_i, \, \min\Big( - \frac{1}{\alpha} S\big( (A^T y_*)_i \big), \, u_i \Big) \right\},
    \qquad \forall \ i=1, \ldots, N.
\]

{\bf Case (b)}: $E=\begin{bmatrix} I_N \\ F \end{bmatrix}$ for some matrix $F \in \mathbb R^{k\times N}$. Such an $E$ appears in the $\ell_1$ penalty of the fused LASSO. Let $v=(v', \wt v)$. Noting that $\| v\|_\infty \le 1 \Leftrightarrow \|v'\|_\infty \le 1, \| \wt v\|_\infty \le 1$, and $E^T v = v' + F^T \wt v$, we have
\begin{equation} \label{eqn:BP_dual_L1_fuse_Lasso}
  \mbox{(D)}: \quad \min_{y, \, \mu \ge 0, \, \|\wt v\|_\infty \le 1} \, \Big( b^T y + \mu^T d
 %
 %
  + \frac{1}{2\alpha} \sum^p_{i=1} \big\|S\big( - (A^T y+F^T \wt v+C^T \mu)_{\Ical_i} \big) \big\|^2_2  \, \Big).
\end{equation}
Letting $(y_*, \mu_*, \wt v_*)$ be an optimal solution of the above reduced dual problem, it can be shown via the similar argument for Case (a) that the unique primal solution $x^*$ is given by
\begin{equation} \label{eqn:solution_x*_fused_Lasso}
   x^*_{\Ical_i} \, = \,  -\frac{1}{\alpha} S\Big(( A^T y_* + F^T \wt v_* + C^T \mu_* )_{\Ical_i}  \Big),   \quad \forall \ i=1, \ldots, p.
\end{equation}
Similarly, when $\Ccal$ is a box constraint set, the equivalent dual problem further reduces to
\[
   \mbox{(D)}: \quad  \min_{y \in \mathbb R^m, \, \| \wt v \|_\infty \le 1 \,} \Big[ \, b^T y + \frac{\alpha}{2}  \sum^N_{i=1} \theta_i \circ \Big( - \frac{1}{\alpha}  S \Big( \big( A^T y + F^T \wt v \big)_i \Big) \Big) \, \Big],
\]
and  the primal solution $x^*$ is expressed in term of a dual solution $(y_*, \wt v_*)$ as
\[
   x^*_i \, = \,  \max\left\{ l_i, \, \min\Big( - \frac{1}{\alpha} S\big( (A^T y_*+ F^T \wt v_*)_i \big), \, u_i \Big) \right\},
   \quad \forall \ i=1, \ldots, N.
\]

\gap

\noindent $\bullet$ {\bf Reduced Dual Problem of the LASSO-like Problem} \ Consider the following cases:

{\bf Case (a)}: $E= \lambda I_N$ for a positive constant $\lambda$. The dual problem (\ref{eqn:Lasso_dual}) reduces to
\begin{equation} \label{eqn:dual_Lasso_L1}
  \mbox{(D)}: \quad \min_{y, \, \mu \ge 0} \, \Big\{ \frac{\|y \|^2_2}{2} + b^T y + d^T \mu \, : \, \| (A^T y+ C^T \mu)_{\Ical_i}  \|_\infty \le  \lambda, \ i=1, \ldots, p \, \Big\}.
\end{equation}
Particularly, when $\mathcal C=\mathbb R^N$, it further reduces to $\min_{\|A^T y \|_\infty \le \lambda} \frac{\|y \|^2_2}{2} + b^T y$; when $\mathcal C=\mathbb R^N_+$, in light of the fact that  the inequality  $w + v \ge 0$ and $\| v\|_\infty \le 1$ is feasible for a given vector $w$ if and only if $w \ge -\mathbf 1$, we see that the dual problem (\ref{eqn:Lasso_dual}) further reduces to $\min_{A^T y \ge - \lambda \mathbf 1} \frac{\|y \|^2_2}{2} + b^T y$.


{\bf Case (b)}: $E=\begin{bmatrix} \lambda I_N \\ F \end{bmatrix}$ for some matrix $F \in \mathbb R^{k\times N}$ and $\lambda>0$. Such an $E$ appears in the $\ell_1$ penalty of the fused LASSO. The dual problem (\ref{eqn:Lasso_dual}) reduces to
\begin{equation} \label{eqn:dual_fused_Lasso_L1}
  \mbox{(D)}: \quad \min_{y, \, \mu \ge 0, \, \| \wt v\|_\infty \le 1} \, \Big\{ \frac{\|y \|^2_2}{2} + b^T y + d^T \mu \, : \, \| (A^T y+ F^T \wt v + C^T \mu)_{\Ical_i}  \|_\infty \le  \lambda, \ i=1, \ldots, p \, \Big\}.
\end{equation}
Particularly, when $\mathcal C=\mathbb R^N$, it further reduces to $\min_{\|A^T y + F^T \wt v \|_\infty \le \lambda, \, \| \wt v\|_\infty\le 1} \frac{\|y \|^2_2}{2} + b^T y$.

\gap

\noindent $\bullet$ {\bf Reduced Dual Problem of the BPDN-like Problem} \ Consider the following cases under the similar assumptions given below (\ref{eqn:BPDN_primal}) in Section~\ref{subsect:general_results}:

{\bf Case (a)}: $E=  I_N$. The equivalent dual problem (\ref{eqn:BPDN_dual}) becomes
\begin{equation} \label{eqn:dual_BPDN_L1}
  \mbox{(D)}: \quad \min_{y, \, \mu \ge 0} \, \Big\{ \, b^T y + \sigma \|y \|_2  +
  d^T \mu \, : \, \| (A^T y + C^T \mu)_{\Ical_i} \|_\infty \le 1, \ i=1, \ldots, p \, \Big\}.
\end{equation}
When $\mathcal C=\mathbb R^N$, it further reduces to $\min_{\|A^T y \|_\infty \le 1} b^T y + \sigma \|y \|_2 $; when $\mathcal C=\mathbb R^N_+$, the dual problem (\ref{eqn:Lasso_dual}) further reduces to $\min_{A^T y \ge - \mathbf 1}  b^T y+\sigma \|y\|_2$.


{\bf Case (b)}: $E=\begin{bmatrix} I_N \\ F \end{bmatrix}$ for some $F \in \mathbb R^{k\times N}$. The equivalent dual problem (\ref{eqn:BPDN_dual}) reduces to
\begin{equation} \label{eqn:dual_fused_BPDN_L1}
  \mbox{(D)}: \quad \min_{y, \, \mu \ge 0, \, \| \wt v\|_\infty \le 1} \, \Big\{  b^T y + \sigma \|y \|_2 + d^T \mu \, : \, \| (A^T y+ F^T \wt v+ C^T \mu)_{\Ical_i}  \|_\infty \le  1, \ i=1, \ldots, p \, \Big\}.
\end{equation}
Particularly, when $\mathcal C=\mathbb R^N$, it further reduces to $\min_{\|A^T y + F^T \wt v \|_\infty \le 1, \, \| \wt v\|_\infty\le 1}  b^T y+\sigma \|y\|_2$.

%
\subsection{Applications to Problems Associated with the Norm from Group LASSO} \label{subsect:dual_group_Lasso}

Consider the norm $\|x\|_\star :=\sum^p_{i=1} \|x_{\Ical_i}\|_2$ arising from the group LASSO, where its dual norm $\| x \|_\diamond = \max_{i=1,\ldots, p} \|x_{\Ical_i} \|_2$.
%
%
%

\gap

\noindent $\bullet$ {\bf Reduced Dual Problem of the Regularized BP-like Problem} \ We consider $E=I_N$ as follows.

{\bf Case (a)}: $\Ccal$ is a general polyhedral set defined by $C x \le d$. Given a vector $w$, we see that
\[
   \min_{\|v\|_\diamond \le 1} \sum^p_{i=1} \big\| (v-w)_{\Ical_i} \big \|^2_2 = \min_{ ( \max_{i=1,\ldots, p} \|v_{\Ical_i} \|_2)  \, \le 1} \sum^p_{i=1} \big \| (v-w)_{\Ical_i} \big \|^2_2 = \sum^p_{i=1} \min_{ \| v_{\Ical_i} \|_2 \le 1} \big \| v_{\Ical_i} - w_{\Ical_i} \big\|^2_2.
\]
Let $S_{\|\cdot\|_2}(z)  : =    \big(1 - \frac{1}{\| z\|_2} \big)_+ z,  \forall  z \in \mathbb R^n$ denote the soft thresholding operator with respect to the $\ell_2$-norm, and let $B_2(0, 1):=\{ z \, | \, \|z\|_2 \le 1\}$. It is known that given $w$, $z_* := \Pi_{B_2(0,1)}(w)= w - S_{\|\cdot\|_2}(w)$ and $\|z_* -w \|^2_2= \| S_{\|\cdot\|_2}(w) \|^2_2 = [ (\|w\|_2 -1)_+ ]^2$. Applying these results to (\ref{eqn:BP_dual}), we obtain the reduced dual problem
\begin{equation} \label{eqn:BP_like_group_Lasso}
  \mbox{(D)}: \quad \min_{y, \, \mu \ge 0} \, \Big( b^T y + \mu^T d +
  \frac{1}{2\alpha} \sum^p_{i=1} \Big[ \big(\big\| (A^T y + C^T \mu )_{\Ical_i} \big\|_2 -1 \big)_+ \Big]^2  \, \Big).
\end{equation}
Letting $(y_*, \mu_*)$ be an optimal solution of the problem (D), the primal solution is given by
\begin{equation} \label{eqn:solution_x*_group_Lasso}
   x^*_{\Ical_i}  =   -\frac{1}{\alpha} S_{\|\cdot \|_2} \Big(( A^T y_* + C^T \mu_* )_{\Ical_i} \Big),   \quad \forall \, i=1, \ldots, p.
\end{equation}
The above results can be easily extended to the decoupled polyhedral constraint set given by (\ref{eqn:polyhedral_set_C}).

\gap

{\bf Case (b)}: $\Ccal$ is a box constraint with $0 \in \Ccal$. In this case, the dual variable $\mu$ can be removed. In fact, it follows from the results at the end of Section~\ref{subsect:general_results} that the reduced dual problem is
\begin{equation} \label{eqn:BP_dual_group_Lasso_box_constr}
  \min_{y, \, (v_{\Ical_i})^p_{i=1}} \sum^p_{i=1} \Big[ \, \frac{b^T y}{p} + \frac{\alpha}{2}  \sum_{j\in \Ical_i} \theta_j\Big( - \frac{1}{\alpha} \big( (A_{\bullet\Ical_i})^T y + v_{\Ical_i} \big)_j \Big) \Big], \quad \mbox{ subject to } \ \|v_{\Ical_i} \|_2 \le 1, \ \ i=1, \ldots, p,
\end{equation}
where the functions $\theta_j$'s are defined in (\ref{eqn:q_i}). Given a dual solution $(y_*, v_*)$,
 the primal solution $x^*_{\Ical_i}  = \max\big(\mathbf l_{\Ical_i}, \min(-\frac{  (A_{\bullet\Ical_i})^T y_* + (v_*)_{\Ical_i}}{\alpha}, \mathbf u_{\Ical_i}) \big)$  for $i=1, \ldots, p$.
When the box constraint set $\mathcal C$ is a cone, the above problem can be further reduced by removing $v$. For example,
when $\Ccal=\mathbb R^N$,  the reduced dual problem becomes
$
 \min_{y\in \mathbb R^m} \big ( b^T y + \frac{1}{2\alpha} \sum^p_{i=1} \big[ \big(\| (A_{\bullet\Ical_i})^T y\|_2 - 1 \big)_+\big]^2 \, \big),
$
and the primal solution $x^*$ is given in term of a dual solution $y_*$ by
$x^*_{\Ical_i} =  -\frac{1}{\alpha}S_{\|\cdot\|_2}( (A_{\bullet\Ical_i})^T y_*)$ for  $i=1, \ldots, p$.
When $\Ccal=\mathbb R^N_+$, the reduced dual problem becomes:
$
    \min_{y\in \mathbb R^m} \big ( b^T y + \frac{1}{2\alpha} \sum^p_{i=1} \big[ \big(\| [(A_{\bullet\Ical_i})^T y ]_-\|_2 - 1 \big)_+\big]^2 \, \big).
$
Given a dual solution $y_*$,  the unique primal solution $x^*$ is given by:
%
%
\[
   x^*_{\Ical_i} \, = \,  \Big[ - \frac{1}{\alpha}  ((A_{\bullet\Ical_i})^T y_*)_+ +\frac{1}{\alpha} S_{\|\cdot\|_2} \big( ((A_{\bullet\Ical_i})^T y_*)_- \big) \Big]_+, \qquad \forall \ i=1, \ldots, p.
\]

\gap

\noindent $\bullet$ {\bf Reduced Dual Problem of the LASSO-like Problem} \  Let $E= \lambda I_N$ for a positive constant $\lambda$. For a general polyhedral constraint $C x \le d$, the dual problem (\ref{eqn:Lasso_dual}) reduces to
\begin{equation} \label{eqn:dual_Lasso_norm_group_Lasso}
  \mbox{(D)}: \quad \min_{y, \, \mu \ge 0} \, \Big\{ \frac{\|y \|^2_2}{2} + b^T y + d^T \mu \, : \, \| (A^T y+ C^T \mu)_{\Ical_i}  \|_2 \le  \lambda, \ i=1, \ldots, p \, \Big\}.
\end{equation}
When $\mathcal C=\mathbb R^N$, the dual problem becomes $\min_{y} \big( b^T y + \frac{\|y \|^2_2}{2}  \big)$ subject to $\|(A_{\bullet \Ical_i})^T y \|_2 \le \lambda, i=1, \ldots, p$.
When $\mathcal C=\mathbb R^N_+$, the dual problem is $\min_{ y} \big( b^T y + \frac{\|y \|^2_2}{2}  \big)$ subject to $ \| v\|_\diamond \le 1, A^T y + \lambda v \ge 0$, which is equivalent to $\|v_{\Ical_i}\|_2 \le 1, (A^T y)_{\Ical_i} + v_{\Ical_i} \ge 0$ for all $i=1, \ldots, p$. Note that for a given vector $w \in \mathbb R^k$, the inequality system $v+w \ge 0, \|v\|_2 \le 1$ is feasible if and only if $ w \in B_2(0, 1) + \mathbb R^k_+$.
%
%
Hence, when $\mathcal C=\mathbb R^N_+$, the dual problem is given by $\min_{ y} \big( b^T y + \frac{\|y \|^2_2}{2}  \big)$ subject to $ (A^T y)_{\Ical_i} \in B_2(0, \lambda) + \mathbb R^{|\Ical_i|}_+$ for all $i=1, \ldots, p$.

\gap

\noindent $\bullet$ {\bf Reduced Dual Problem of the BPDN-like Problem}  Let $E= I_N$. Suppose the similar assumptions indicated in Section~\ref{subsect:general_results} hold. For a general polyhedral set $\Ccal$, the dual problem (\ref{eqn:BPDN_dual}) reduces to
\begin{equation} \label{eqn:dual_BPDN_norm_group_Lasso}
  \mbox{(D)}: \quad \min_{y, \, \mu \ge 0} \, \Big\{ \, b^T y + \sigma \|y \|_2  +
  d^T \mu \, : \, \| (A^T y + C^T \mu)_{\Ical_i} \|_2 \le 1, \ i=1, \ldots, p \, \Big\}.
\end{equation}
When $\mathcal C=\mathbb R^N$, the dual problem is $\min_{ y} \big( b^T y + \sigma \|y \|_2  \big)$ subject to $\|(A_{\bullet \Ical_i})^T y \|_2 \le 1, i=1, \ldots, p$.
When $\mathcal C=\mathbb R^N_+$, the dual problem is $\min_{ y} \big( b^T y + \sigma \|y \|_2  \big)$ subject to $ (A^T y)_{\Ical_i} \in B_2(0, 1) + \mathbb R^{|\Ical_i|}_+$ for all $i=1, \ldots, p$.
%





%
\section{Column Partition based Distributed Algorithms: A Dual Approach} \label{sect:distributed_schemes}

To elaborate the development of distributed algorithms, recall that the index sets $\Ical_1, \ldots, \Ical_p$ form a disjoint union of $\{1, \ldots, N\}$ such that  $\{ A_{\bullet\Ical_i} \}^p_{i=1}$ forms a column partition of $A\in \mathbb R^{m\times N}$. We assume that there are $p$ agents in a network, and the $i$th agent possesses $A_{\bullet\Ical_i}$ and $b$ (with possibly additional information) but does not know the other $A_{\bullet\Ical_j}$'s with $j\ne i$.
Further, we consider a general network topology modeled by an undirected graph $\mathcal G(\mathcal V, \mathcal E)$, where $\mathcal V=\{1, \ldots, p\}$ is the set of agents, and $\mathcal E$
represents the set of bidirectional edges connecting two agents in $\mathcal V$, i.e., if two agents are connected by an edge in $\mathcal E$, then two agents can exchange information. We assume that $\mathcal G(\mathcal V, \mathcal E)$ is connected.

%
%

Motivated by the large-scale problems arising from various applications indicated in Section~\ref{sect:introduction}, we are especially interested in the cases where $N$ is large whereas each agent has a limited or relatively small memory capacity. Hence, we consider certain classes of polyhedral sets $\Ccal:=\{ x \in \mathbb R^N \, | \, Cx \le d\}$ with $C \in \mathbb R^{\ell \times N}$ and $d \in \mathbb R^\ell$, for example, those with $\ell \ll N$ or those with decoupling structure given in Remark~\ref{remark:decoupled_polyhedral_set}. Under these conditions, we will show that the dual problems obtained in Section~\ref{sect:dual_problems} can be easily formulated as separable or locally decoupled convex problems to which a wide range of existing distributed schemes can be effectively applied using column partition. This is particularly important to the development of two-stage distributed schemes for the densely coupled LASSO-like and BPDN-like problems (cf. Section~\ref{subsect:algorithm_Lasso_BPDN}).
For the purpose of illustration, we consider operator splitting method based synchronous distributed schemes including the Douglas-Rachford (D-R) algorithm and its variations \cite{DavisYin_SVA17, HuXiaoLiu_CDC18}. It should be pointed out that it is {\em not} a goal of this paper to improve the performance of the existing schemes or seek the most efficient existing scheme
%
%
%
but rather to demonstrate their applicability to the obtained dual problems. In fact, many other synchronous or asynchronous distributed schemes can be exploited under even weaker assumptions, e.g., time-varying networks.

%
%
%
%

%

%
\subsection{Column Partition based Distributed Schemes for Regularized BP-like Problems} \label{subsect:Distributed_BP}

Consider the regularized BP-like problem (\ref{eqn:BP_primal}) with the regularization parameter $\alpha>0$. It follows from Section~\ref{sect:dual_problems} that this problem can be solved from its dual, which can be further solved via column partition based distributed schemes. We discuss these schemes for specific norms $\|\cdot\|_\star$ as follows.

%
\subsubsection{Regularized BP-like Problems with $\ell_1$-norm} \label{subsubsect:BP_scheme_L1}

When $\|\cdot\|_\star$ is the $\ell_1$-norm, Corollary~\ref{coro:exact_reg_BPL1} shows the exact regularization holds, i.e., the regularized problem attains a solution of the original BP-like problem for all sufficiently small $\alpha>0$.

{\bf Case (a)}: $E=I_N$ and $\ell$ is small (e.g., $\ell \ll N$). To solve the problem (\ref{eqn:BP_like_1norm}),   let $J_i:\mathbb R^m \times \mathbb R^\ell \rightarrow \mathbb R$ be
\[
   J_i (\mathbf y_i, {\boldsymbol\mu}_i) \, := \, \frac{b^T \mathbf y_i + d^T {\boldsymbol\mu}_i}{p} \,  + \frac{1}{2\alpha} \Big \| S\big(- (A_{\bullet \Ical_i})^T \mathbf y_i - ( C_{\bullet \Ical_i})^T {\boldsymbol\mu}_i \big) \Big \|^2_2, \qquad i=1, \ldots, p,
\]
 the consensus subspace $\mathcal A_{\mathbf y}:=\{ \mathbf y \, | \, \mathbf y_i = \mathbf y_j, \, \forall \, (i, j) \in \mathcal E\}$, where $\mathbf y :=(\mathbf y_1, \ldots, \mathbf y_p) \in \mathbb R^{mp}$, and the consensus cone $\mathcal A_{\boldsymbol \mu}:=\{ \boldsymbol \mu \ge 0 \, | \, \boldsymbol \mu_i = \boldsymbol \mu_j, \, \forall \, (i, j) \in \mathcal E\}$, where $\boldsymbol \mu:=(\boldsymbol \mu_1, \ldots, \boldsymbol \mu_p) \in \mathbb R^{\ell p}$. Hence, the dual problem (\ref{eqn:BP_like_1norm}) can be equivalently written as the following separable convex minimization problem:
\[
    \min_{(\mathbf y, \boldsymbol\mu) \in \mathcal A_{\mathbf y} \times \mathcal A_{\boldsymbol \mu} } \, \sum^p_{i=1} J_i(\mathbf y_i, \boldsymbol\mu_i).
    %
\]
A specific operator splitting based scheme for solving this problem is the Douglas-Rachford algorithm \cite{HuXiaoLiu_CDC18}: given suitable constants $\eta \in (0, 1)$ and $\rho>0$,
\begin{subequations} \label{eqn:BP_scheme_L1}
\begin{eqnarray}
   \mathbf w^{k+1} & = & 
   \Pi_{\mathcal A_{\mathbf y} \times \mathcal A_{\boldsymbol \mu}} (\mathbf z^k), \label{eqn:BP_scheme_L1_a}\\
   %
   %
   %
   \mathbf z^{k+1}_i & = & \mathbf z^k_i + 2\eta \Big( \, \mbox{prox}_{\rho J_i} \big(2 \mathbf w^{k+1}_i - \mathbf z^k_i \big) - \mathbf w^{k+1}_i \Big), \quad i=1, \ldots, p, \label{eqn:BP_scheme_L1_b}
\end{eqnarray}
\end{subequations}
where $\mathbf w^k=(\mathbf y^k, {\boldsymbol \mu}^k) \in \mathbb R^{mp}\times \mathbb R^{\ell p}$,  $\mathbf z^k=(\mathbf z^k_\mathbf y, \mathbf z^k_{\boldsymbol \mu}) \in \mathbb R^{mp}\times \mathbb R^{\ell p}$, $\mathbf w^k_i=(\mathbf y^k_i, {\boldsymbol \mu}^k_i)\in \mathbb R^{p}\times \mathbb R^{\ell}$, $\mathbf z^k_i=( (\mathbf z^k_\mathbf y)_i, (\mathbf z^k_{\boldsymbol \mu})_i )\in \mathbb R^{m}\times \mathbb R^{\ell}$ for $i=1, \ldots, p$,
and $\mbox{prox}_{\rho f}(\cdot)$ denotes the proximal operator for a convex function $f$.
This scheme can be implemented distributively, where each agent $i$ only knows $A_{\bullet \Ical_i}, C_{\bullet \Ical_i}$ and other constant vectors or parameters, e.g., $b, d$, and $\alpha$, and has the local variables $\mathbf w^k_i$ and $\mathbf z^k_i$ at step $k$. For any $\mathbf z=(\mathbf z_\mathbf y, \mathbf z_{\boldsymbol \mu})$, we have $\Pi_{\mathcal A_{\mathbf y} \times \mathcal A_{\boldsymbol \mu}} (\mathbf z)=\big( \ol{\mathbf z_\mathbf y}, (\ol {\mathbf z_{\boldsymbol \mu}})_+\big)$, where $\ol{\mathbf z_\mathbf y}:=  \mathbf 1 \otimes [\frac{1}{p}\sum^p_{i=1} (\mathbf z_\mathbf y)_i]  $ denotes the averaging of $\mathbf z_\mathbf y$, and the similar notation holds for $\ol {\mathbf z_{\boldsymbol \mu}}$. Therefore, the first step given by (\ref{eqn:BP_scheme_L1_a}) can be implemented via distributed averaging algorithms \cite{XiaoBoyd_SCL04}, and the second step given by (\ref{eqn:BP_scheme_L1_b}) can be also computed in a distributed manner due to the separable structure of $J_i$'s.

The scheme given by (\ref{eqn:BP_scheme_L1}) generates a sequence $( \mathbf z^k)$ that converges to $\mathbf z_*$. Further, $(\mathbf y_*, {\boldsymbol \mu_*}) =\Pi_{\mathcal A_{\mathbf y} \times \mathcal A_{\boldsymbol \mu}} (\mathbf z_*)$, where $\mathbf y_*=(y_*, \ldots, y_*)\in \mathcal A_{\mathbf y}$ for some $y_* \in \mathbb R^m$ and ${\boldsymbol \mu_*}=(\mu_*, \ldots, \mu_*) \in \mathcal A_{\boldsymbol \mu}$ for some $\mu_* \in \mathbb R^\ell_+$ in view of the connectivity of $\mathcal G$. Once the dual solution $(y_*, \mu_*)$ is found, it follows from (\ref{eqn:solution_x*_L1}) that the primal solution $x^*$ is given  by  $x^*_{\Ical_i} = -\frac{1}{\alpha} S\big( (A_{\bullet \Ical_i})^T y_* + ( C_{\bullet \Ical_i})^T  \mu_*  \big)$ for each $i=1, \ldots, p$.


\gap

{\bf Case (b)}: $E=I_N$, $\ell$ is large, and $\Ccal$ is given by (\ref{eqn:polyhedral_set_C}). It follows from Remark~\ref{remark:decoupled_polyhedral_set} and Section~\ref{subsect:dual_problem_L1} that the equivalent dual problem is given by: recalling that $\mu:=(\mu_{\Ical_i})^p_{i=1} \in \mathbb R^N$,
\[
  \mbox{(D)}: \quad \min_{y, \, \mu \ge 0} \, \Big( b^T y +  \sum^p_{i=1} \mu^T_{\Ical_i} d_{\Ical_i} +
   \frac{1}{2\alpha} \sum^p_{i=1} \big \| S( -(A_{\bullet\Ical_i})^T y -C^T_{\Ical_i} \mu_{\Ical_i}) \big\|^2_2 \, \Big).
\]
Letting $J_i(\mathbf y_i, \mu_{\Ical_i}) := b^T \mathbf y_i/p + \mu^T_{\Ical_i} d_{\Ical_i} + \frac{1}{2\alpha} \big \| S( -(A_{\bullet\Ical_i})^T \mathbf y_i -C^T_{\Ical_i} \mu_{\Ical_i}) \big\|^2_2$,
%
%
 an equivalent form of (D) is
\[
 \min_{\mathbf y \in \mathcal A_{\mathbf y}, \, \mu \ge 0} \, \sum^p_{i=1} J_i(\mathbf y_i, \mu_{\Ical_i}).
\]
It can be solved  by the scheme (\ref{eqn:BP_scheme_L1})  distributively by replacing $\mathcal A_{\boldsymbol \mu}$ in (\ref{eqn:BP_scheme_L1_a}) with $\mathbb R^N_+$.

\gap

{\bf Case (c)}: $E=I_N$, and $\Ccal$ is a box constraint set with $0\in \Ccal$ given right above (\ref{eqn:q_i}). To solve the reduced dual problem given by (\ref{eqn:BP_dual_box_constraint_L1}) with the variable $y$ only, let $J_i:\mathbb R^m \rightarrow \mathbb R$ be
$
   J_i (\mathbf y_i)  :=  \frac{b^T \mathbf y_i}{p} \,  + \frac{\alpha}{2}  \sum_{j \in \Ical_i} \theta_j \circ \Big( - \frac{1}{\alpha}  S \big( \big( (A_{\bullet \Ical_i} )^T \mathbf y_i \big)_j \big) \Big)
$
such that an equivalent form of the dual problem is: $\min_{\mathbf y \in \mathcal A_{\mathbf y} } \sum^p_{i=1} J_i(\mathbf y_i)$, which can also be solved via the scheme (\ref{eqn:BP_scheme_L1}) by removing $\mathcal A_{\boldsymbol \mu}$ from (\ref{eqn:BP_scheme_L1_a}).

\gap

{\bf Case (d)}: $E=\begin{bmatrix} I_N \\ \gamma D_1 \end{bmatrix}$, where $D_1 \in \mathbb R^{(N-1)\times N}$ is the first order difference matrix, and $\gamma$ is a positive constant. This case arises from the fused LASSO with $F=\gamma D_1$ and $\wt v =(\wt v_1, \ldots, \wt v_{N-1}) \in \mathbb R^{N-1}$ (cf. Section~\ref{sect:problem_formulation_properties}). We first consider a general polyhedral set $\Ccal$ with a small $\ell$.
Let $n_s:=\sum^s_{i=1} |\Ical_i|$ for $s=1, \ldots, p$.
Without loss of generality, let $\Ical_1=\{1, \ldots, n_1\}$, and $\Ical_{i+1}=\{ n_i+1, \ldots,  n_i+|\Ical_{i+1}|\}$ for each $i =1, \ldots, p-1$. We also assume that $(i, i+1) \in \mathcal E$ for any $i=1, \ldots, p-1$.
%
%
We introduce more notation.
Let $\mathbf v_1:=\wt v_{\Ical_1}$, $\mathbf v_i:=(\wt v_{n_{i-1}}, \wt v_{\Ical_i})$ for each $i=2, \ldots, p-1$, and $\mathbf v_p :=(\wt v_i: i=n_{p-1}, \ldots, N-1)$. Thus for $i=1, \ldots, p-1$, $\mathbf v_{i}$ and $\mathbf v_{i+1}$ overlap on one variable $\wt v_{n_i}$. Let  $\mathbf v:=(\mathbf v_i)^p_{i=1} \in \mathbb R^{N+p-2}$.
Further, let $r_1:=|\Ical_1|$, $r_i:=|\Ical_i|+1$ for each $i=2, \ldots, p-1$, and $r_p:=|\Ical_p|$.
Define $J_i:\mathbb R^m \times \mathbb R^\ell \times \mathbb R^{r_i} \rightarrow \mathbb R$ as
\[
   J_i (\mathbf y_i, \boldsymbol \mu_i, \mathbf v_i ) \, := \, \frac{b^T \mathbf y_i + d^T \boldsymbol \mu_i}{p}  + \frac{1}{2\alpha} \Big \| S\big(- (A_{\bullet \Ical_i})^T \mathbf y_i - ( C_{\bullet \Ical_i})^T {\boldsymbol\mu}_i  -  \gamma (D_1^T)_{\Ical_i \bullet} \mathbf v_i \big) \Big \|^2_2, \quad i=1, \ldots, p.
\]
Due to the structure of $D_1$ and the network topology assumption (i.e., $(i, i+1) \in \mathcal E$ for any $i=1, \ldots, p-1$), $J_i(\mathbf y_i, \boldsymbol \mu_i,\mathbf v_i)$'s are locally coupled \cite{HuXiaoLiu_CDC18}. Hence,  the dual problem (\ref{eqn:BP_dual_L1_fuse_Lasso}) can be equivalently written as the following locally coupled convex minimization problem:
\[
    \min_{(\mathbf y, \boldsymbol\mu) \in \mathcal A_{\mathbf y} \times \mathcal A_{\boldsymbol \mu}, \,  \|\mathbf v_i\|_\infty \le 1, \, i=1, \ldots, p } \, \sum^p_{i=1} J_i(\mathbf y_i, \boldsymbol \mu_i, \mathbf v_i), \ \ \mbox{ subject to }  \ \
    %
    %
    (\mathbf v_i)_{r_i}=(\mathbf v_{i+1})_1, \ \forall \, i=1, \ldots, p-1.
\]
Let $B_\infty(0, 1):=\{ \mathbf v \, | \, \| \mathbf v \|_\infty \le 1\}$ and $\mathcal A_C:=\{ \mathbf v \in \mathbb R^{N+p-2} \, | \, (\mathbf v_i)_{r_i}=(\mathbf v_{i+1})_1, \ \forall \, i=1, \ldots, p-1 \}$. The following three-operator splitting scheme \cite[Algorithm 1]{DavisYin_SVA17} can be used for distributed computation:
\begin{subequations} \label{eqn:BP_scheme_fused_Lasso}
\begin{eqnarray}
   \wt{\mathbf w}^k & =  & \Pi_{\mathcal A_{\mathbf y} \times \mathcal A_{\boldsymbol \mu}\times \mathcal A_C}(\mathbf z^{k}),
%
%
   \label{eqn:BP_scheme_fused_Lasso_a}\\
   \wh{\mathbf w}^{k} & = & \Pi_{\mathbb R^{m p} \times \mathbb R^{\ell p} \times B_\infty(0, 1)} \Big(2 \wt{\mathbf w}^k- \mathbf z^k - \eta \sum^p_{i=1} \nabla J_i\big( (\wt{\mathbf w}^k)_i \big) \Big),  \label{eqn:BP_scheme_fused_Lasso_b} \\
   \mathbf z^{k+1} & = & \mathbf z^k + \lambda \big(  \wh{\mathbf w}^{k} - \wt{\mathbf w}^k \big), \label{eqn:BP_scheme_fused_Lasso_c}
\end{eqnarray}
\end{subequations}
where $\mathbf z^k=(\mathbf z^k_\mathbf y, \mathbf z^k_{\boldsymbol \mu}, {\mathbf z}^{k}_{\mathbf v}) \in \mathbb R^{mp}\times \mathbb R^{\ell p} \times \mathbb R^{N+p-2}$, and $\eta, \lambda$ are suitable positive constants depending on the Lipschitz constant of $\sum^p_{i=1}\nabla J_i$; see \cite[Thoerem 1]{DavisYin_SVA17} for details. For a distributed implementation of this scheme, each  agent $i$ has the local variable $(\mathbf z^{k}_i, \wt{\mathbf w}^k_i, \wh{\mathbf w}^k_i)$, and it is easy to see that the projections in (\ref{eqn:BP_scheme_fused_Lasso_a}) and (\ref{eqn:BP_scheme_fused_Lasso_b}) can be computed   distributively due to the separable structure of $J_i$'s and $B_\infty(0, 1)$ using the distributed averaging and other techniques. The scheme (\ref{eqn:BP_scheme_fused_Lasso}) yields a sequence that converges to $(\wt{ \mathbf w}_*,  \wh{ \mathbf w}_*, \mathbf z_*)$. A dual solution $(y_*, \mu_*, \wt v_*)$ can be retrieved from $\wh{ \mathbf w}_*$ in a similar way as shown in Case (a). Finally,  the primal solution $(x^*_{\Ical_i})^p_{i=1}$ is obtained using (\ref{eqn:solution_x*_fused_Lasso}).

Column partition based distributed schemes similar to (\ref{eqn:BP_scheme_fused_Lasso}) can be developed to the decoupled constraint set given by (\ref{eqn:polyhedral_set_C}) and a box constraint set. Moreover,  similar schemes can be developed for the generalized total variation denoising and $\ell_1$-trend filtering where $E=\lambda D_1$ or $E=\lambda D_2$ with $\lambda>0$.

%
%

%
\subsubsection{Regularized BP-like Problems with the Norm from Group LASSO}

Let $\|x\|_\star :=\sum^p_{i=1} \|x_{\Ical_i}\|_2$, and its dual norm $\| x \|_\diamond = \max_{i=1,\ldots, p} \|x_{\Ical_i} \|_2$.
We assume that exact regularization holds if needed; see Section~\ref{subsect:exact_reg_group_BP}.
Consider $E=I_N$ and a general polyhedral set $\Ccal$.
%
%
The dual problem  (\ref{eqn:BP_like_group_Lasso}) can be written as the  separable convex program: $\min_{(\mathbf y, \boldsymbol\mu) \in \mathcal A_{\mathbf y} \times \mathcal A_{\boldsymbol \mu} } \, \sum^p_{i=1} J_i(\mathbf y_i, \boldsymbol\mu_i)$, where
\[
   J_i (\mathbf y_i, {\boldsymbol\mu}_i) \, := \, \frac{b^T \mathbf y_i + d^T {\boldsymbol\mu}_i}{p} \,  + \frac{1}{2\alpha} \Big[  \big(\big\| (A_{\bullet \Ical_i})^T \mathbf y_i + ( C_{\bullet \Ical_i})^T {\boldsymbol\mu}_i \big\|_2 -1 \big)_+  \Big]^2, \qquad \forall \ i=1, \ldots, p,
\]
and $\mathcal A_{\mathbf y}, \mathcal A_{\boldsymbol \mu}$ are defined in Case (a) in Section~\ref{subsubsect:BP_scheme_L1}. Thus the distributed scheme (\ref{eqn:BP_scheme_L1}) can be applied.

When $\mathcal C$ is a box constraint set, consider the dual problem (\ref{eqn:BP_dual_group_Lasso_box_constr}).  By introducing $p$ copies of $y$'s given by $\mathbf y_i$ and imposing the consensus condition on $\mathbf y_i$'s, this problem can be converted to  a convex  program of the variable $(\mathbf y_i, v_{\Ical_i})^p_{i=1}$ with a separable objective function and separable constraint sets which have nonempty interiors. Thus by Slater's condition, the D-R scheme or three-operator splitting based column distributed schemes similar to (\ref{eqn:BP_scheme_fused_Lasso}) can be developed. If, in addition, $\mathcal C$ is a cone, the dual problems can be further reduced to unconstrained problems of the variable $y$ only, e.g.,  those for $\Ccal=\mathbb R^N$ and $\Ccal=\mathbb R^N_+$ given in Case (b) of Section~\ref{subsect:dual_group_Lasso}.  These problems can be formulated as consensus convex programs and solved by column partition based distributed schemes. The primal solution $x^*_{\Ical_i}$ can be computed distributively using a dual solution $y_*$ and the operator $S_{\|\cdot\|_2}$ (cf. Section~\ref{subsect:dual_group_Lasso}). We omit these details here.
%

%
\subsection{Two-stage, Column Partition based Distributed Algorithms for LASSO-like and BPDN-like Problems: A Dual Approach} \label{subsect:algorithm_Lasso_BPDN}

The LASSO-like problem (\ref{eqn:Lasso_primal}) and the BPDN-like problem (\ref{eqn:BPDN_primal}) are not exactly regularized in general (cf. Section~\ref{subsect:fail_exact_reg_Lasso_BPDN}). Their objective functions or constraints are densely coupled without separable or locally coupled structure, making the development of column partition based distributed schemes particularly difficult. By leveraging their solution properties, we develop dual based two-stage distributed schemes.

We first outline key ideas of the two-stage distributed schemes. It follows from Proposition~\ref{prop:Lasso_to_BP} that if $A x_*$ is known for a minimizer $x_*$ of the LASSO or BPDN, then an exact primal solution can be solved by a regularized BP-like problem shown in Section~\ref{subsect:Distributed_BP} using column partition of $A$, assuming that exact regularization holds.
 To find $A x_*$, we deduce from Lemmas~\ref{lem:primal_dual_Lasso} and \ref{lem:primal_dual_BPDN} that $A x_*= b + y_*$ or $A x_* = b + \frac{\sigma y_*}{\|y_*\|_2}$, where $y_*$ is a dual solution of the LASSO or BPDN. Since the dual problems of LASSO and BPDN can be solved distributively using column partition of $A$, this yields the two-stage distributed schemes; see Algorithm~\ref{algo:distributed_two_stage}.

\begin{algorithm}
\caption{Two-stage Distributed Algorithm for LASSO-like (resp. BPDN-like) Problem}
\begin{algorithmic}[1]
\label{algo:distributed_two_stage}
\STATE Initialization
\STATE \fbox{Stage 1} Compute a dual solution $y_*$ of the LASSO-like problem (\ref{eqn:Lasso_primal})  (resp. BPDN-like problem (\ref{eqn:BPDN_primal}))
  using a column partition based distributed scheme;
\STATE \fbox{Stage 2} Solve the following regularized BP-like problem for a sufficiently small $\alpha>0$ using $y_*$ and a column partition based distributed scheme:
\begin{equation} \label{eqn:BP_Lasso}
  \mbox{BP}_{\mbox{LASSO}}: \quad \min_{x\in \mathbb R^N} \, \quad  \|E x\|_\star+\frac{\alpha}{2}\|x\|^2_2, \ \ \mbox{ subject to } \ Ax = b+y_*, \ \ x \in \Ccal
\end{equation}
or
\begin{equation} \label{eqn:BP_BPDN}
  \mbox{BP}_{\mbox{BPDN}}: \quad \min_{x \in \mathbb R^N} \, \quad  \|E x\|_\star+\frac{\alpha}{2}\|x\|^2_2, \ \ \mbox{ subject to } \ Ax = b+ \frac{\sigma y_*}{\| y_*\|_2}, \ \ \ x \in \Ccal
\end{equation}

\STATE Output: obtain the subvectors $x^*_{\Ical_i}$ for each $i=1, \ldots, p$
%
%

\end{algorithmic}
\end{algorithm}

We discuss column partition based distributed schemes indicated in Stage 1 as follows; distributed schemes in Stage 2 have been discussed in Section~\ref{subsect:Distributed_BP}. For each fused problem involving the matrix $D_1$ discussed below, we assume that its graph satisfies  $(i, i+1)\in \mathcal E, \forall \, i=1, \ldots, p-1$.


%
\subsubsection{Column Partition based Distributed Algorithms for the Dual of LASSO-like Problem} \label{subsubsect:scheme_dual_Lasso}

Let $\Ccal$ be a general polyhedral set given by $C x \le d$ unless otherwise stated. Consider $\|\cdot\|_\star=\|\cdot \|_1$ first.

{\bf Case (a)}: $E=\lambda I_N$ for a positive constant $\lambda$. Suppose $\ell$ is small. Recall that $\mathbf y =(\mathbf y_1, \ldots, \mathbf y_p) \in \mathbb R^{mp}$, and $\boldsymbol \mu=(\boldsymbol \mu_1, \ldots, \boldsymbol \mu_p) \in \mathbb R^{\ell p}$. Define the set $\Wcal:=\{ (\mathbf y, \boldsymbol \mu) \, | \, \| (A_{\bullet \Ical_i})^T \mathbf y_i + (C_{\bullet \Ical_i})^T \boldsymbol \mu_i \|_\infty \le \lambda, \ \forall \, i=1, \ldots, p \}$ and the functions $J_i(\mathbf y_i, \boldsymbol \mu_i) := \big(\frac{\|\mathbf y_i \|^2_2}{2} + b^T \mathbf y_i+ d^T  \boldsymbol \mu_i)/p$ for $i=1, \ldots, p$. Using $\mathcal A_{\mathbf y}$ and $\mathcal A_{\boldsymbol \mu}$ introduced in Section~\ref{subsubsect:BP_scheme_L1}, the dual problem (\ref{eqn:dual_Lasso_L1}) can be written as the following consensus convex program:
\[
   \min_{(\mathbf y, \boldsymbol\mu) \in \mathcal A_{\mathbf y} \times \mathcal A_{\boldsymbol \mu} } \, \sum^p_{i=1} J_i(\mathbf y_i, \boldsymbol\mu_i), \quad \mbox{ subject to } \quad (\mathbf y, \boldsymbol\mu) \in \Wcal.
\]
A three-operator splitting scheme \cite[Algorithm 1]{DavisYin_SVA17} can be used for solving this problem:
\begin{equation} \label{eqn:scheme_Lasso_L1}
    \wt{\mathbf w}^k =   \Pi_{\mathcal A_{\mathbf y} \times \mathcal A_{\boldsymbol \mu}}(\mathbf z^{k} ),  \ \ \ \wh{\mathbf w}^{k}  =  \Pi_{\mathcal W  } \Big(2 \wt{\mathbf w}^k- \mathbf z^k - \eta \sum^p_{i=1} \nabla J_i\big( (\wt{\mathbf w}^k)_i \big) \Big),  \ \ \ \mathbf z^{k+1}  =  \mathbf z^k + \lambda \big(  \wh{\mathbf w}^{k} - \wt{\mathbf w}^k \big),
\end{equation}
where $\mathbf z^k=(\mathbf z^k_\mathbf y, \mathbf z^k_{\boldsymbol \mu}) \in \mathbb R^{mp}\times \mathbb R^{\ell p}$, and $\eta>0$ and $\lambda>0$ are suitable constants. Due to the separable structure of $\Wcal$ and $J_i$'s, this scheme can be implemented distributively via similar techniques discussed  in Section~\ref{subsubsect:BP_scheme_L1}. It can be extended to the decoupled constraint set in (\ref{eqn:polyhedral_set_C}) by replacing $\mathcal A_{\boldsymbol \mu}$ with $\mathbb R^N_+$. For some important special cases, e.g., $\Ccal=\mathbb R^N$ or $\Ccal=\mathbb R^N_+$, the variable $\mu$ or $\boldsymbol\mu$ can be removed; see the discussions below (\ref{eqn:dual_Lasso_L1}). Especially, the resulting consensus convex programs for $\Ccal=\mathbb R^N$ and $\Ccal=\mathbb R^N_+$ have strongly convex objective functions. Hence, an accelerated operator splitting method \cite[Algorithm 2]{DavisYin_SVA17} can be used to develop column partition based distributed schemes  with the convergence rate $O(1/k)$.
Since $\mathcal W$ is separable and polyhedral, an alternative scheme for (\ref{eqn:scheme_Lasso_L1}) is to drop the constraint $\Wcal$, replace $J_i$ by the sum of $J_i$ and the  indicator function of the corresponding $\mathcal W_i$, and then use the D-R scheme.

\begin{remark} \rm \label{remark:Lasso_scaled_BP}
  When $E=\lambda I_N$ and $\Ccal$ is a polyhedral cone (i.e., $d=0$), let $y_*$ be the unique dual solution of the problem (\ref{eqn:Lasso_primal}) obtained from the first stage. We discuss a variation of the BP formulation in the second stage by exploiting solution properties of (\ref{eqn:Lasso_primal}).
  In view of Lemma~\ref{lem:primal_dual_Lasso} and $E=\lambda I_N$, we deduce that $A x_* = b+ y_*$ and $\lambda \| x_*\|_1 = -y^T_* (b+ y_*)$ for any minimizer $x_*$ of the problem (\ref{eqn:Lasso_primal}), noting that $\| x_*\|_1$ is constant on the solution set by  Proposition~\ref{prop:soln_property}. Suppose $x_* \ne 0$ or equivalently $b+y_* \ne 0$.  Then
$
 \|x_*\|_1 = -\frac{1}{\lambda} y_*^T (y_* + b)$, and  $\frac{ A x_*}{ \|x_*\|_1} = -\frac{\lambda (y_* + b)}{ y_*^T (y_* + b) }$.
Consider the scaled regularized BP for a small $\alpha>0$:
\begin{equation} \label{eqn:scaled_BP}
     \mbox{Scaled r-BP}: \quad \min_{z \in \mathbb R^N} \ \| z \|_1 + \frac{\alpha}{2} \| z \|^2_2 \quad \mbox{ subject to } \quad A z = - \frac{\lambda (y_* + b)}{ y^T_* (y_* + b) }, \quad z \in \Ccal.
\end{equation}
Once  the unique minimizer $z_*$ of the above regularized BP is obtained (satisfying $\|z_*\|_1=1$), then the least 2-norm minimizer $x_*$ is given by $x_* = -\frac{1}{\lambda} y^T_* (y_* + b) z_*$.

The advantages of using the scaled regularized BP (\ref{eqn:scaled_BP}) are two folds. First, since $\|x_*\|_1$ may be small or near zero in some applications, a direct application of the $\mbox{BP}_{\mbox{LASSO}}$ using $y_*$ in Algorithm~\ref{algo:distributed_two_stage} may be sensitive to round-off errors. However, using the scaled BP (\ref{eqn:scaled_BP}) can avoid such a problem. Second, the suitable value of $\alpha$ achieving exact regularization is often unknown, despite the existential result in theory. A simple rule  for choosing such an $\alpha$ is given in \cite{LaiYin_SIAMG13}: $\alpha \le \frac{1}{10 \| \wh x \|_\infty}$, where  $\wh x \ne 0$ is a  sparse vector to be recovered. Assuming that the solution set of the problem (\ref{eqn:Lasso_primal}) contains $\wh x$, an estimate of the upper bound of $\alpha$ is $\frac{1}{10 \| \wh x \|_1}$ in view of $\| \wh x \|_1 \ge \| \wh x \|_\infty$. Again, when $\|x_*\|_1$ is small, this upper bound may have a large numerical error. Instead, when the scaled BP (\ref{eqn:scaled_BP}) is used, we can simply choose $\alpha \le \frac{1}{10}$.
\end{remark}


{\bf Case (b)}: $E=\begin{bmatrix} \lambda I_N \\ \gamma D_1 \end{bmatrix}$  for positive constants $\lambda$ and $ \gamma$, and $\ell$ is small. This case is an extension of the fused LASSO. To solve its dual problem in (\ref{eqn:dual_fused_Lasso_L1}) with $F=\gamma D_1$, recall the notation $\mathbf v$ and $\mathcal A_C$ introduced for Case (d) in Section~\ref{subsubsect:BP_scheme_L1}. Define the set $\mathcal U:=\{ (\mathbf y, \boldsymbol \mu, \mathbf v) \, | \, \| \mathbf v_i \|_\infty \le 1, \ \| (A_{\bullet \Ical_i})^T \mathbf y_i + (C_{\bullet \Ical_i})^T \boldsymbol \mu_i + \gamma (D_1)^T_{\Ical_i \bullet} \mathbf v_i\|_\infty \le \lambda, \, \forall \, i=1, \ldots, p \}$ and the functions $J_i(\mathbf y_i, \boldsymbol \mu_i, \mathbf v_i) := \big(\frac{\|\mathbf y_i \|^2_2}{2} + b^T \mathbf y_i+ d^T  \boldsymbol \mu_i)/p$ for $i=1, \ldots, p$. Hence, the dual problem (\ref{eqn:dual_fused_Lasso_L1}) can be formulated as the  locally coupled convex program:
\[
   \min_{(\mathbf y, \boldsymbol\mu, \mathbf v) \in \mathcal A_{\mathbf y} \times \mathcal A_{\boldsymbol \mu} \times \mathcal A_C } \, \sum^p_{i=1} J_i(\mathbf y_i, \boldsymbol\mu_i, \mathbf v_i), \quad \mbox{ subject to } \quad (\mathbf y, \boldsymbol\mu, \mathbf v) \in \mathcal U.
\]
%
 %
Replacing the step (\ref{eqn:BP_scheme_fused_Lasso_b}) by $\wh{\mathbf w}^k = \Pi_{\mathcal U}(\mathbf z^{k})$, the scheme (\ref{eqn:BP_scheme_fused_Lasso}) can be applied. Since $\mathcal U$ is a decoupled constraint set, the projection $\Pi_{\mathcal U}$ can be computed distributively.
%
%
This leads to a distributed scheme for the above convex program. These schemes can be extended to the generalized total variation denoising  and generalized $\ell_1$-trend filtering with $E=\lambda D_1$ or $E=\lambda D_2$.

\gap

We then consider the norm $\|x\|_\star=\sum^p_{i=1} \|x_{\Ical_i}\|_2$ arising from group LASSO. Suppose $E=\lambda I_N$ with $\lambda>0$. In view of the dual formulation (\ref{eqn:dual_Lasso_norm_group_Lasso}),  the distributed scheme (\ref{eqn:scheme_Lasso_L1}) can be applied by replacing $\Wcal$ with the set $\wt\Wcal:=\{ (\mathbf y, \boldsymbol \mu) \, | \, \| (A_{\bullet \Ical_i})^T \mathbf y_i + (C_{\bullet \Ical_i})^T \boldsymbol \mu_i \|_2 \le \lambda, \ \forall \, i=1, \ldots, p \}$, which has nonempty interior.
%

%
\subsubsection{Column Partition based Distributed Algorithms for the Dual of BPDN-like Problem}

Suppose the assumptions given below (\ref{eqn:BPDN_primal}) in Section~\ref{subsect:general_results} hold. Consider the dual problem (\ref{eqn:BPDN_dual}) with a general polyhedral set $\Ccal$. As shown in Lemma~\ref{lem:primal_dual_BPDN}, a dual solution $y_* \ne 0$ under these assumptions. Hence, the function $\|y\|_2$ is always differentiable near $y_*$. In what follows, consider $\|\cdot \|_\star=\| \cdot \|_1$ first.

{\bf Case (a)}: $E= I_N$. Suppose $\ell$ is small first. In light of the dual formulation (\ref{eqn:BPDN_dual}), it is easy to verify that the distributed scheme (\ref{eqn:scheme_Lasso_L1}) can be applied by setting $\lambda$ in $\mathcal W$ as one and replacing the functions $J_i$ with $\wt J_i(\mathbf y_i, \boldsymbol \mu_i) := \big(\sigma \|\mathbf y_i \|_2 + b^T \mathbf y_i+ d^T  \boldsymbol \mu_i)/p$ for $i=1, \ldots, p$.
 This scheme can be extended to the decoupled constraint set in (\ref{eqn:polyhedral_set_C}) by replacing $\mathcal A_{\boldsymbol \mu}$ with $\mathbb R^N_+$. When $\Ccal=\mathbb R^N$ or $\Ccal=\mathbb R^N_+$, the variable $\mu$ or $\boldsymbol\mu$ can be removed and the proposed scheme can be easily adapted for these cases; see the discussions below (\ref{eqn:dual_BPDN_L1}).
 Moreover, when $E= I_N$ and $\Ccal$ is a polyhedral cone (i.e., $d$=0), it follows from Lemma~\ref{lem:primal_dual_BPDN} and the assumption that the optimal value of (\ref{eqn:BPDN_primal}) is positive that $-b^T y_* - \sigma \|y_*\|_2>0$. Hence, by a similar argument in Remark~\ref{remark:Lasso_scaled_BP}, we deduce that a primal solution  $x_*$ can be solved from the following scaled regularized BP using the (unique) dual solution $y_*$:
\[
   \mbox{Scaled r-BP}: \quad \min_{z \in \mathbb R^N} \ \| z \|_1 + \frac{\alpha}{2} \| z \|^2_2 \quad \mbox{ subject to } \quad A z = - \frac{b + \frac{\sigma y_*}{\|y_*\|_2} }{b^T y_* + \sigma \|y_*\|_2}, \ \ z \in \Ccal.
\]
Once  the unique minimizer $z_*$ of the above regularized BP is obtained (satisfying $\|z_*\|_1=1$), then the least 2-norm minimizer $x_*$ of the BPDN is given by $x_* = -(b^T y_* + \sigma \|y_*\|_2) z_*$.

\gap

{\bf Case (b)}: $E=\begin{bmatrix} I_N \\ \gamma D_1 \end{bmatrix}$  for a positive constant $\gamma$, and $\ell$ is small. To solve its dual problem in (\ref{eqn:dual_fused_BPDN_L1}) with $F=\gamma D_1$, define the set $\wt{\mathcal U}:=\{ (\mathbf y, \boldsymbol \mu, \mathbf v) \, | \, \| \mathbf v_i \|_2 \le 1, \ \| (A_{\bullet \Ical_i})^T \mathbf y_i + (C_{\bullet \Ical_i})^T \boldsymbol \mu_i + \gamma (D_1)^T_{\Ical_i \bullet} \mathbf v_i\|_2 \le 1, \, \forall \, i=1, \ldots, p \}$ and the functions $\wt J_i(\mathbf y_i, \boldsymbol \mu_i, \mathbf v_i) := \big(\sigma \|\mathbf y_i \|_2 + b^T \mathbf y_i+ d^T  \boldsymbol \mu_i)/p$ for $i=1, \ldots, p$. Thus $\wt{\mathcal U}$ has nonempty interior.
The dual problem (\ref{eqn:dual_fused_BPDN_L1}) can be formulated as the  locally coupled convex program:
\[
   \min_{(\mathbf y, \boldsymbol\mu, \mathbf v) \in \mathcal A_{\mathbf y} \times \mathcal A_{\boldsymbol \mu} \times \mathcal A_C } \, \sum^p_{i=1} \wt J_i(\mathbf y_i, \boldsymbol\mu_i, \mathbf v_i), \quad \mbox{ subject to } \quad (\mathbf y, \boldsymbol\mu, \mathbf v) \in \wt {\mathcal U}.
\]
By a similar argument for Case (b) of Section~\ref{subsubsect:scheme_dual_Lasso},  the scheme (\ref{eqn:BP_scheme_fused_Lasso}) can be applied by replacing the step (\ref{eqn:BP_scheme_fused_Lasso_b}) with $\wh{\mathbf w}^k = \Pi_{\wt {\mathcal U}}(\mathbf z^{k})$ and $J_i$'s with $\wt J_i$'s.

We then consider the norm $\|x\|_\star=\sum^p_{i=1} \|x_{\Ical_i}\|_2$ arising from group LASSO. Suppose $E=I_N$. Similarly, the distributed scheme (\ref{eqn:scheme_Lasso_L1}) can be applied to the dual formulation (\ref{eqn:dual_BPDN_norm_group_Lasso}) by replacing $J_i$'s with $\wt J_i$'s defined above and $\Wcal$ with the set $\wt\Wcal:=\{ (\mathbf y, \boldsymbol \mu) \, | \, \| (A_{\bullet \Ical_i})^T \mathbf y_i + (C_{\bullet \Ical_i})^T \boldsymbol \mu_i \|_2 \le 1, \ \forall \, i=1, \ldots, p \}$.

%
\section{Overall Convergence of the  Two-stage  Distributed Algorithms} \label{sect:overall_converg}

In this section, we analyze the overall convergence of the two-stage distributed algorithms proposed in Section~\ref{sect:distributed_schemes}, assuming that a distributed algorithm in each stage is convergent.
To motivate the overall convergence analysis, it is noted that an  algorithm of the first-stage generates an approximate solution $y^k$ to the solution $y_*$ of the dual problem, and this raises the question of whether using this approximate solution in the second stage  leads to significant discrepancy when solving the second-stage problem (\ref{eqn:BP_Lasso}) or (\ref{eqn:BP_BPDN}). Inspired by this question and its implication to the overall convergence of the two-stage algorithms, we  establishes the continuity of the solution of the regularized BP-like problem (\ref{eqn:BP_primal}) in $b$, which is closely related to sensitivity analysis of the problem  (\ref{eqn:BP_primal}).  We first present some technical preliminaries.

\begin{lemma} \label{lem:uniform_bdd_subdf_norm}
 Let $\|\cdot\|_\star$ be a norm on $\mathbb R^n$ and $\|\cdot \|_\diamond$ be its dual norm. Then for any $x \in \mathbb R^n$, $\|v\|_\diamond \le 1$ for any $v \in \partial \| x \|_\star$.
\end{lemma}

\begin{proof}
 Fix $x \in \mathbb R^n$, and let $v \in  \partial \| x \|_\star$. Hence, $\|y\|_\star \ge \|x \|_\star + \langle v, y - x\rangle $ for all $y\in \mathbb R^n$. Since $\| y - x\|_\star \ge \big | \| y\|_\star - \|x\|_\star \big| \ge \langle v, y - x\rangle$, we have  $\langle v, \frac{y - x}{\|y -x \|_\star} \rangle \le 1$ for any $y \ne x$. This shows that $\|v\|_\diamond \le 1$.
\end{proof}

Another result we will use  is concerned with the Lipschitz property of the linear complementarity problem (LCP) under certain singleton property. Specifically, consider the LCP $(q, M)$: $0\le u \perp M u + q \ge 0 $ for a given matrix $M \in \mathbb R^{n\times n}$ and a vector $q \in \mathbb R^n$. Let $\SOL(q, M)$ denote its solution set. The following theorem is an extension of  a well-known fact in the LCP and variational inequality theory, e.g.,   \cite[Propositioin 4.2.2]{FPang_book03},  \cite[Theorem 10]{GowdaS_MP96}, and \cite{ShenPang_ACC07}.

\begin{theorem} \label{thm:LCP_singleton}
Consider the LCP $(q, M)$. Suppose a matrix $E \in \mathbb R^{p\times n}$ and a set $\Wcal \subseteq \mathbb R^n$ are such that  for any $q \in \Wcal$, $\SOL(q, M)$ is nonempty  and $E\SOL(q, M)$ is singleton. The following hold:
\begin{itemize}
  \item [(i)] $E\SOL(\cdot, M)$ is locally Lipschitz at each $q \in \Wcal$, i.e., there exist a constant $L_q>0$ and a neighborhood $\mathcal N$ of $q$ such that $\|E\SOL(q', M) - E\SOL(q, M)\|\le L_q \|q'-q\|$ for any $q'\in \mathcal N \cap \Wcal$;
  \item [(ii)] If $\Wcal$ is a convex set, then $E\SOL(\cdot, M)$ is (globally) Lipschitz continuous on $\Wcal$, i.e., there exists a constant $L>0$ such that $\|E\SOL(q, M) - E\SOL(q', M)\|\le L \|q-q'\|$ for all $q', q \in \Wcal$.
\end{itemize}
%
%
\end{theorem}


We apply the above results to the regularized BP-like problem subject to a generic polyhedral constraint, in addition to the linear equality constraint, i.e.,
\begin{equation}  \label{eqn:BP_gen_polyhedral}
  \min_{x \in \mathcal C, \, A x = b} \| E x\|_\star + \frac{\alpha}{2} \|x \|^2_2,
\end{equation}
where $\alpha$ is a positive constant, $E \in \mathbb R^{r \times N}$, $A \in \mathbb R^{m\times N}$, the polyhedral set $\mathcal C:=\{ x \in \mathbb R^N \, | \, Cx \le d \}$ for some $C \in \mathbb R^{\ell \times N}$ and $d \in \mathbb R^\ell$, and $b\in \mathbb R^m$ with $b \in A \mathcal C:=\{ A x \, | \, x \in \mathcal C\}$. We shall show that its unique optimal solution is continuous in $b$, where we assume that $A \ne 0$ without loss of generality. To achieve this goal,
consider the necessary and sufficient optimality condition for the unique solution $x_*$ of (\ref{eqn:BP_gen_polyhedral}), namely, there exist (possibly non-unique) multipliers $\lambda \in \mathbb R^m$ and $\mu \in \mathbb R^\ell_+$ such that
 \begin{equation} \label{eqn:optimality_cond}
    0 \in E^T \partial \| E x_* \|_\star + \alpha \, x_* + A^T \lambda + C^T \mu, \quad A x_* = b, \quad 0 \le \mu \perp C x_* - d \le 0.
 \end{equation}
 When we need to emphasize the dependence of $x_*$ on $b$, we write it as $x_*(b)$ in the following development. For a given $b \in A \mathcal C$ and its corresponding unique minimizer $x_*$ of (\ref{eqn:BP_gen_polyhedral}), define the set
\[
 \Scal(x_*) \, := \, \Big\{ (w, \lambda, \mu) \ \big| \ w \in \partial \| E x_* \|_\star, \ \  E^T w + \alpha x_* + A^T \lambda + C^T \mu =0, \ \ 0 \le \mu \perp C x_* -d \le 0  \Big\}.
\]
This set contains all the sub-gradients $w$ and the multipliers $\lambda, \mu$ satisfying the optimality condition at $x_*$, and it is often unbounded due to possible unboundeness of $\lambda$ and $\mu$ (noting that by Lemma~\ref{lem:uniform_bdd_subdf_norm}, $w$'s are bounded). To overcome this difficulty in continuity analysis, we present the following proposition.

\begin{proposition} \label{prop:boundness_seq}
  The following hold for the minimization problem (\ref{eqn:BP_gen_polyhedral}):
 \begin{itemize}
   \item [(i)] Let $\mathcal B$ be a bounded set in $\mathbb R^m$. Then  $\{ x_*(b) \, | \, b \in A \mathcal C\, \cap \, \mathcal B \}$ is a bounded set;
   \item [(ii)] Let $(b^k)$ be a convergent sequence in $A \mathcal C \cap \mathcal B$. Then there exist a constant $\gamma>0$ and an index subsequence $(k_s)$ such that  for each $k_s$, there exists $(w^{k_s}, \lambda^{k_s}, \mu^{k_s}) \in \Scal(x_*(b^{k_s}))$ satisfying $\|(\lambda^{k_s}, \mu^{k_s})\| \le \gamma$.
 \end{itemize}
\end{proposition}

\begin{proof}
 (i) Suppose $\{ x_*(b) \, | \, b \in A \mathcal C \, \cap \, \mathcal B \}$ is unbounded. Then there exists a sequence $(b^k)$ in $A \mathcal C \, \cap \, \mathcal B$ such that the sequence $\big( x_*(b^k) \big)$ satisfies $\| x_*(b^k) \| \rightarrow \infty$.
  For notational simplicity, we let $x^k_*:= x_*(b^k)$ for each $k$.
  Without loss of generality, we assume that $\big( \frac{x^k_*}{\|x^k_*\|} \big)$ converges to $v_* \ne 0$. In view of $A \frac{x^k_*}{\|x^k_*\|} = \frac{b^k}{\|x^k_*\|}$, $ C  \frac{x^k_*}{\|x^k_*\|} \le \frac{d}{\|x^k_*\|}$, and the fact that $(b^k)$ is bounded, we have $A v_* =0$ and $C v_* \le 0$. Further, for each $k$, there exist $\lambda^k \in \mathbb R^m$ and $\mu^k \in \mathbb R^\ell_+$ and $w^k \in \partial \| E x^k_* \|_\star$ such that $ E^T w^k + \alpha x^k_* + A^T \lambda^k + C^T \mu^k=0$, $A x^k_* = b^k$, and $0\le \mu^k \perp C x^k_* - d \le 0$ for each $k$.  We claim that $(C v_*)^T \mu^k = 0$ for all large $k$. To prove this claim, we note that, by virtue of $C v_* \le 0$, that for each index $i$, either $(C v_*)_i=0$ or $(C v_*)_i <0$. For the latter, it follows from $\big( C  \frac{x^k_*}{\|x^k_*\|} - \frac{d}{\|x^k_*\|}\big)_i \rightarrow (C v_*)_i$ that $( C x^k_* - d)_i <0$ for all large $k$. Hence, we deduce from the optimality condition (\ref{eqn:optimality_cond}) that $\mu^k_i=0$ for all large $k$. This shows that $(C v_*)_i \cdot \mu^k_i=0, \forall \, i$ for all large $k$. Hence, the claim holds.
  In view of this claim and $A v_*=0$, we see that
   left multiplying $v^T_*$  to the equation $ E^T w^k + \alpha x^k_* + A^T \lambda^k + C^T \mu^k =0$ leads to $(E v_*)^T w^k + \alpha (v_*)^T x^k_*=0$, or equivalently $(E v_*)^T \frac{ w^k}{\| x^*_k\|} + \alpha (v_*)^T  \frac{x^k_*}{\|x^k_*\|}=0$, for all large $k$. Since $(w^k)$ is bounded by Lemma~\ref{lem:uniform_bdd_subdf_norm}, we have, by taking the limit, that $\alpha \|v_* \|^2_2=0$, leading to $v_*=0$, a contradiction. Hence, $\{ x_*(b) \, | \, b \in A \mathcal C \, \cap \, \mathcal B\}$ is bounded.


 (ii) Given a convergent sequence $(b^k)$ in $A \mathcal C$, we use $x^k_*:= x_*(b^k)$ for each $k$ again. Consider a sequence $\big( (w^{k}, \lambda^{k}, \mu^{k}) \big)$,  where $ (w^{k}, \lambda^{k}, \mu^{k}) \in \Scal (x^k_*)$ is arbitrary for each $k$. In view of the boundedness of $(x^k_*)$ proven in (i) and Lemma~\ref{lem:uniform_bdd_subdf_norm}, we assume by taking a suitable subsequence that $(w^k, x^k_*) \rightarrow (\wh w, \wh x)$. Let the index set $\wh \Ical_\mu:=\{ i  \, | \, (C^T \wh x -d)_i <0  \}$. If there exists an index $i \not\in \wh \Ical_\mu$ such that $(\mu^k_i)$ has a zero subsequence $(\mu^{k'}_i)$, then let $\wh\Ical'_\mu  := \wh\Ical_\mu \cup \{ i \}$. We then consider the subsequence $(\mu^{k'})$. If there exists an index $j \notin \wh\Ical'_\mu$ such that $(\mu^{k'}_j)$ has a zero subsequence $(\mu^{k''}_j)$, then  let $\wh \Ical''_\mu  := \wh\Ical'_\mu \cup \{ j \}$ and consider the subsequence $(\mu^{k''})$. Continuing this process in finite steps, we obtain  an index subsequence $(k_s)$  and an index set $\Ical_\mu$ such that $(C^T x^{k_s}_* -d)_{\Ical_\mu} <0$ and $\mu^{k_s}_{\Ical^c_\mu}>0$ for all $k_s$'s, where $\Ical^c_\mu := \{1, \ldots, N\} \setminus \Ical_\mu$. By the complementarity condition in (\ref{eqn:optimality_cond}), we have $\mu^{k_s}_{\Ical_\mu}=0$ and $\mu^{k_s}_{\Ical^c_\mu}>0$ for each $k_s$.

Since $A \ne 0$, there exits  an index subset $\Jcal \subseteq \{1, \ldots, m\}$ such that the columns of $(A^T)_{\bullet \Jcal}$ (or equivalently $(A_{\Jcal \bullet})^T$) form a basis for  $R(A^T)$. Hence, for each $\lambda^{k_s}$, there exists a unique vector $\wt \lambda^{k_s}$ such that $A^T \lambda^{k_s} = (A_{\Jcal \bullet})^T \wt \lambda^{k_s}$. In view of the equations $E^T w^{k_s} + \alpha x^{k_s}_* + (A_{\Jcal \bullet})^T \wt \lambda^{k_s} + C^T \mu^{k_s}=0$ and $A_{\Jcal \bullet} x^{k_s}_* = b^{k_s}_\Jcal$, we obtain via  a straightforward computation that
  \begin{align}
    \wt \lambda^{k_s} & \, = \,  -\big( A_{\Jcal \bullet} (A_{\Jcal \bullet})^T \big)^{-1}  \Big[ \, \alpha b^{k_s}_\Jcal +A_{\Jcal \bullet} (E^T w^{k_s} + C^T \mu^{k_s}) \, \Big], \notag \\
    x^{k_s}_* & \, = \, (A_{\Jcal \bullet})^T \big( A_{\Jcal \bullet} (A_{\Jcal \bullet})^T \big)^{-1} b^{k_s}_\Jcal + \frac{1}{\alpha} \Big[  (A_{\Jcal \bullet})^T \big( A_{\Jcal \bullet} (A_{\Jcal \bullet})^T \big)^{-1} A_{\Jcal \bullet} - I \Big] \big(E^T w^{k_s} + C^T \mu^{k_s} \big), \label{eqn:x_k_s}
  \end{align}
  where  $C^T \mu^{k_s} = (C_{\Ical^c_\mu \bullet})^T \mu^{k_s}_{\Ical^c}$  for each $k_s$ in view of $\mu^{k_s}_{\Ical^c_\mu}>0$ and $\mu^{k_s}_{\Ical_\mu}=0$.
  Substituting $x^{k_s}_*$ into the complementarity condition $  0 \le \mu^{k_s} \perp d - C x^{k_s}_*  \ge 0$, we deduce that $\mu^{k_s}_{\Ical^c_\mu}$ satisfies the following conditions:
  \[
     0 \le \mu^{k_s}_{ \Ical^c_\mu} \perp  d_{\Ical^c_\mu} - C_{\Ical^c_\mu \bullet} x^{k_s}_*    \ge 0, \quad C_{\Ical_\mu \bullet} x^{k_s}_* - d_{\Ical_\mu} = H \, \mu^{k_s}_{\Ical^c_\mu} + h^{k_s} \le 0,
  \]
  %
 %
   where $d_{\Ical^c_\mu} - C_{\Ical^c_\mu \bullet} x^{k_s}_*  = G \, \mu^{k_s}_{\Ical^c_\mu} + g^{k_s}$, and the matrices $G, H$ and the vectors $g^{k_s}, h^{k_s}$ are given by
  \begin{eqnarray*}
    G & := & \frac{1}{\alpha} C_{\Ical^c_\mu \bullet}\Big[ I- (A_{\Jcal \bullet})^T \big( A_{\Jcal \bullet} (A_{\Jcal \bullet})^T \big)^{-1} A_{\Jcal \bullet}  \Big] (C_{\Ical^c_\mu \bullet})^T, \\
   g^{k_s} & := & \frac{1}{\alpha} C_{\Ical^c_\mu \bullet}\Big[ I - (A_{\Jcal \bullet})^T \big( A_{\Jcal \bullet} (A_{\Jcal \bullet})^T \big)^{-1} A_{\Jcal \bullet}  \Big] E^T w^{k_s} - C_{\Ical^c_\mu \bullet} (A_{\Jcal \bullet})^T \big( A_{\Jcal \bullet} (A_{\Jcal \bullet})^T \big)^{-1} b^{k_s}_\Jcal + d_{\Ical^c_\mu}, \\
    H & := & \frac{1}{\alpha} C_{\Ical_\mu \bullet}\Big[ I- (A_{\Jcal \bullet})^T \big( A_{\Jcal \bullet} (A_{\Jcal \bullet})^T \big)^{-1} A_{\Jcal \bullet}  \Big] (C_{\Ical^c_\mu \bullet})^T, \\
   h^{k_s} & := & \frac{1}{\alpha} C_{\Ical_\mu \bullet}\Big[ I - (A_{\Jcal \bullet})^T \big( A_{\Jcal \bullet} (A_{\Jcal \bullet})^T \big)^{-1} A_{\Jcal \bullet}  \Big] E^T w^{k_s} - C_{\Ical_\mu \bullet} (A_{\Jcal \bullet})^T \big( A_{\Jcal \bullet} (A_{\Jcal \bullet})^T \big)^{-1} b^{k_s}_\Jcal + d_{ \Ical_\mu}.
  \end{eqnarray*}
%
%
%
Since $\mu^{k_s}_{\Ical^c_\mu}>0$, we must have $G \mu^{k_s}_{\Ical^c_\mu} + g^{k_s}=0$.
%
%
For the  matrices  $G, H$ and given vectors $g, h$, define the polyhedral set $\Kcal(G, H, g, h):=\{ z \, | \, z \ge 0, \ G z + g =0,  \ H z + h \le 0 \} = \{ z \, | \,  D z + v \ge 0 \}$, where $D := \begin{bmatrix} I \\ G \\ -G \\ - H \end{bmatrix}$ and $v:=\begin{bmatrix} 0 \\ g \\ -g \\ -h \end{bmatrix}$.  Hence, for each $k_s$, $\Kcal(G, H, g^{k_s}, h^{k_s})$ contains the vector $\mu^{k_s}_{\Ical^c_\mu}>0$ and thus is nonempty.
We write $v$ as $v^{k_s}$ when $(g, h)=(g^{k_s}, h^{k_s})$.
Let $\wt z^{k_s}$ be the least 2-norm point of $\Kcal(G, H, g^{k_s}, h^{k_s})$, i.e., $\wt z^{k_s}$ is the unique solution to $\min \frac{1}{2} \|z\|^2_2$ subject to $D z + v^{k_s} \ge 0$. Since its underlying optimization problem has a (feasible) polyhedral constraint, its necessary and sufficient optimality condition is:  $\wt z^{k_s} - D^T \nu =0, \ 0\le \nu \perp D \wt z^{k_s} + v^{k_s} \ge 0$ for some (possibly non-unique) multiplier $\nu$. Let $\SOL(v^{k_s}, D D^T)$ be the solution set of the LCP: $0\le \nu \perp  v^{k_s} + D D^T \nu \ge 0$. By the uniqueness of $\wt z^{k_s}$, $\wt z^{k_s}=D^T \SOL(v^{k_s}, D D^T)$ such that $D^T \SOL(v^{k_s}, D D^T)$ is singleton.  

Since  $g^{k_s}$ and $h^{k_s}$ are affine functions of $(w^{k_s}, b^{k_s})$ and the sequences $(w^{k_s})$ and  $(b^{k_s})$ are convergent, $( v^{k_s} )$ is convergent and we let $v^*$ be its limit. We show as follows that the polyhedral set $\{ z \, | \, D z + v^* \ge 0 \}$ is nonempty. Suppose not. Then it follows from a version of Farkas' lemma \cite[Theorem 2.7.8]{CPStone_book92} that there exists $w \ge 0$ such that $D^T w =0$ and $w^T v^* <0$. Since $(v^{k_s}) \rightarrow v^*$, we see that $w^T v^{k_s}<0$ for all large $k_s$. By \cite[Theorem 2.7.8]{CPStone_book92} again, we deduce that $D z + v^{k_s} \ge 0$ has no solution $z$ for all large $k_s$, yielding a contradiction. This shows that $\{ z \, | \, D z + v^* \ge 0 \}$ is nonempty. Thus $\SOL(v^*, D^T D)$ is nonempty and $D^T \SOL(v^*, D D^T)$ is singleton. Define the function $R(v):=D^T\SOL(v, D D^T)$. By Theorem~\ref{thm:LCP_singleton}, $R(\cdot)$ is locally Lipschitz continuous at $v^*$, i.e., there exist a constant $L_*>0$ and a neighborhood $\mathcal V$ of $v^*$ such that for any $v\in \mathcal V$ satisfying that $\{ z \, | \, D z + v \ge 0\}$ is nonempty, $\|R(v) - R(v^*)\| \le L_* \| v - v^*\|$. This, along with the convergence of $(v^{k_s})$ to $v^*$, show that  $\{ \wt z^{k_s} \, | \, \wt z^{k_s}=R(v^{k_s}), \forall \, k_s \}$ is bounded.
%
%
%
%
%
%
%
For each $k_s$, let $\wh \mu^{k_s} := ( \wh \mu^{k_s}_{\Ical_\mu}, \wh \mu^{k_s}_{\Ical^c_\mu})=(0, \wt z^{k_s})$.
Hence, $(\wh \mu^{k_s})$ is a bounded sequence.
Further, let $\wt \lambda^{k_s}  : =  -\big( A_{\Jcal \bullet} (A_{\Jcal \bullet})^T \big)^{-1}  \big[\alpha b^{k_s}_\Jcal +A_{\Jcal \bullet} (E^T w^{k_s} + C^T \wh \mu^{k_s})  \big]$, and $\wh \lambda^{k_s}:= \big( \wh \lambda^{k_s}_{\Jcal}, \wh \lambda^{k_s}_{\Jcal^c} \big)=( \wt \lambda^{k_s}, 0)$.
This implies that $( \wt \lambda^{k_s} )$, and thus $( \wh \lambda^{k_s})$, is bounded. Hence, $\big( (\wh \lambda^{k_s}, \wh\mu^{k_s})  \big)$ is a  bounded sequence, i.e., there exists $\gamma >0$ such that $\|(\wh \lambda^{k_s}, \wh \mu^{k_s})\| \le \gamma$ for all $k_s$.

Lastly, we show that $(w^{k_s}, \wh\lambda^{k_s}, \wh \mu^{k_s}) \in \Scal(x^{k_s}_*)$ for each $k_s$. In view of (\ref{eqn:x_k_s}), define
\[
  \wh x^{k_s} \, := \, (A_{\Jcal \bullet})^T \big( A_{\Jcal \bullet} (A_{\Jcal \bullet})^T \big)^{-1} b^{k_s}_\Jcal + \frac{1}{\alpha} \Big[  (A_{\Jcal \bullet})^T \big( A_{\Jcal \bullet} (A_{\Jcal \bullet})^T \big)^{-1} A_{\Jcal \bullet} - I \Big] \big(E^T w^{k_s} + C^T \wh \mu^{k_s} \big).
\]
Therefore, $A_{\Jcal\bullet} \wh x^{k_s} = b^{k_s}_\Jcal$ for each $k_s$. Since the columns of $(A_{\Jcal \bullet})^T$ form a basis for  $R(A^T)$ and $b^{k_s} \in A \Pcal$, we have $A\wh x^{k_s} = b^{k_s}$. Moreover, based on the constructions of $\wh \lambda^{k_s}$ and $\wh \mu^{k_s}$, it is easy to show that $E^T w^{k_s} + \alpha \wh x^{k_s} + A^T \wh \lambda^{k_s} + C^T \wh \mu^{k_s}=0$,  $ (C \wh x^{k_s} - d)_{\Ical_\mu} = H \wt z^{k_s} + h^{k_s} \le 0$, and $ (C \wh x^{k_s} - d)_{\Ical^c_\mu} = G \wt z^{k_s} + g^{k_s} = 0$ for each $k_s$. In light of $\wh \mu^{k_s} = ( \wh \mu^{k_s}_{\Ical_\mu}, \wh \mu^{k_s}_{\Ical^c_\mu})=(0, \wt z^{k_s}) \ge 0$, we have $0\le \wh \mu^{k_s} \perp C \wh x^{k_s} - d \le 0$ for each $k_s$. This implies that  $(w^{k_s}, \wh\lambda^{k_s}, \wh \mu^{k_s}) \in \Scal(\wh x^{k_s})$ for each $k_s$. Since the optimization problem (\ref{eqn:BP_gen_polyhedral}) has a unique solution for each $b^{k_s}$, we must have $\wh x^{k_s} = x^{k_s}_*$. This shows that $(w^{k_s}, \wh\lambda^{k_s}, \wh \mu^{k_s}) \in \Scal(x^{k_s}_*)$ for each $k_s$.
\end{proof}

With the help of Proposition~\ref{prop:boundness_seq}, we are ready to show the desired continuity.

\begin{theorem} \label{thm:continuity_x_*}
  Let $\alpha>0$,  $E \in \mathbb R^{r \times N}$ $A \in \mathbb R^{m\times N}$, $\mathcal C:=\{ x \in \mathbb R^N \, | \, Cx \le d \}$ for some $C \in \mathbb R^{\ell\times N}$ and $d \in \mathbb R^\ell$, and $b\in \mathbb R^m$ with $b \in A \mathcal C:=\{ A x \, | \, x \in \mathcal C\}$.  Then the unique solution $x_*$ of the minimization problem (\ref{eqn:BP_gen_polyhedral}) is continuous in $b$ on $A \mathcal C$.
\end{theorem}

\begin{proof}
  Fix an arbitrary $b\in A\mathcal C$. Suppose $x_*(\cdot)$ is discontinuous at this $b$. Then there exist $\varepsilon_0>0$ and a sequence $(b^k)$ in $A \mathcal C$ such that $(b^k)$ converges to $b$ but $\| x^k_* - x_*(b) \| \ge \varepsilon_0$ for all $k$, where $x^k_*:= x_*(b^k)$. By Statement (i) of Proposition~\ref{prop:boundness_seq}, $(x^k_*)$ is  bounded and hence attains a convergent subsequence  which, without loss of generality, can be itself. Let the limit of $(x^k_*)$ be $\wh x$. Further,
 as shown in Statement (ii) of Proposition~\ref{prop:boundness_seq}, there exists a bounded subsequence $\big( (w^{k_s}, \lambda^{k_s}, \mu^{k_s}) \big) $ such that $(w^{k_s}, \lambda^{k_s}, \mu^{k_s}) \in \Scal(x^{k_s}_*)$ for each $k_s$. Without loss of generality, we assume that $\big( (w^{k_s}, \lambda^{k_s}, \mu^{k_s}) \big)$ converges to $(\wh w, \wh \lambda, \wh \mu)$. Since $(E x^{k_s}_*) \rightarrow E \wh x$ and $(w^{k_s})\rightarrow\wh w$ with $w^{k_s} \in \partial \|E x^{k_s}_*\|_\star$ for each $k_s$, it follows from \cite[Proposition B.24(c)]{Bertsekas_book99} that $\wh w \in \partial\| E \wh x \|_\star$.
 By taking the limit, we deduce that $(\wh x, \wh w, \wh \lambda, \wh \mu)$ satisfies $E^T \wh w + \alpha \wh x + A^T \wh \lambda+C^T \wh \mu=0$, $A \wh x = b$, and $0 \le \wh \mu \perp C \wh x - d \le 0$, i.e., $(\wh w, \wh \lambda, \wh \mu) \in \Scal(\wh x)$.
   This shows that $\wh x$ is a solution to (\ref{eqn:BP_gen_polyhedral}) for the given $b$. Since this solution is unique, we must have $\wh x = x_*(b)$. Hence, $(x^{k_s}_*)$ converges to $x_*(b)$, a contradiction to $\| x^{k_s}_* - x_*(b) \| \ge \varepsilon_0$ for all $k_s$. This yields the continuity of $x_*$ in $b$ on $A \mathcal C$.
\end{proof}


%
%

%
%

When the norm $\|\cdot \|_\star$ in the objective function of the optimization problem (\ref{eqn:BP_gen_polyhedral}) is given by the $\ell_1$-norm or a  convex PA function in general,  the continuity property shown in Theorem~\ref{thm:continuity_x_*} can be enhanced. Particularly, the following result establishes the Lipschitz continuity of $x_*$ in $b$, which is useful in deriving the overall convergence rate of the two-stage distributed algorithm.

\begin{theorem} \label{thm:continuity_x_*_convex_PA}
 Let $f:\mathbb R^n \rightarrow \mathbb R$ be a convex piecewise affine function, $A \in \mathbb R^{m\times N}$, $\mathcal C:=\{ x \in \mathbb R^N \, | \, Cx \le d \}$ for some $C \in \mathbb R^{\ell\times N}$ and $d \in \mathbb R^\ell$, and $b\in \mathbb R^m$ with $b \in A \mathcal C:=\{ A x \, | \, x \in \mathcal C\}$.
 Then for any $\alpha>0$,  $\min_{x \in \mathcal C} f(x) + \frac{\alpha}{2} \|x \|^2_2$ subject to $A x = b$ has a unique minimizer $x_*$. Further, $x_*$ is  Lipschitz continuous in $b$ on $A \mathcal C$, i.e., there exists a constant $L>0$ such that $\|x_*(b') - x_*(b)\|\le L\|b'-b\|$ for any $b, b' \in A \mathcal C$.
\end{theorem}

\begin{proof}
We first show the solution existence and uniqueness. Consider a real-valued convex PA function $f(x) \, = \, \max_{i=1, \ldots, r} \, \big( \, p^T_i x + \gamma_i \, \big)$ for a finite family of $(p_i, \gamma_i) \in \mathbb R^N \times \mathbb R, i=1, \ldots, r$.
%
%
Note that for any given $\alpha>0$ and any nonzero $x$,
\[
 f(x) + \frac{\alpha}{2} \|x \|^2_2 \, = \, \|x\|^2_2 \cdot \left[ \, \frac{\alpha}{2} + \max_{i=1, \ldots, r} \Big( p^T_i \frac{x}{\|x\|^2_2} + \frac{\gamma_i}{\|x\|^2_2} \Big) \, \right].
\]
Hence, $f(x) + \frac{\alpha}{2} \|x \|^2_2$ is coercive. Since it is continuous and strictly convex and the constraint set is closed and convex, the underlying optimization problem attains a unique minimizer.

To prove the Lipschitz property of the unique minimizer $x_*$ in $b$, we consider the following equivalent form of the underlying optimization problem:
\begin{equation}
    \min_{t_+, t_-, x} \, t_+ - t_- + \frac{\alpha}{2} \|x \|^2_2 \ \ \, \mbox{subject to} \ \ t_+ \ge 0, \ t_- \ge 0, \, A x = b, \, C x \le d, \ p^T_i x + \gamma_i \le t_+ - t_-, \  i=1, \ldots, r. \label{eqn:opt_problem_Lipschtiz}
\end{equation}
Define the matrix $W:=\begin{bmatrix} p^T_1 \\ \vdots \\ p^T_r \end{bmatrix} \in \mathbb R^{r \times N}$ and the vector $\Gamma:=\begin{bmatrix} \gamma_1 \\ \vdots \\ \gamma_r \end{bmatrix} \in \mathbb R^{r}$. Then the constraints can be written as $t_+ \ge 0, \ t_- \ge 0, \ A x = b, \ C x \le d$, and $W x + \Gamma - t_+ \mathbf 1 + t_-\mathbf 1 \le 0$, where $\mathbf 1$ denotes the vector of ones. Given $b \in A \mathcal C$, the necessary and sufficient optimality conditions for the minimizer $x_*$ are described by a mixed linear complementarity problem, i.e., there exist Lagrange multipliers $\lambda \in \mathbb R^m, \mu \in \mathbb R^\ell_+, \nu \in \mathbb R^r_+, \theta_+ \in \mathbb R_+$ and $ \theta_- \in \mathbb R_+$ such that
\begin{align*}
  \alpha x_* + A^T \lambda + C^T \mu + W^T \nu = 0,   \quad A x_* = b, & \quad   1- \mathbf 1^T \nu -\theta_+=0, \quad  -1 + \mathbf 1^T \nu - \theta_-=0,    \\
   0 \le \mu \perp C x_* - d \le 0, \ 0 \le \nu \perp W x_* + \Gamma - t_+ \mathbf 1 + t_-\mathbf 1 \le 0, &  \quad 0\le \theta_+ \perp t_+ \ge 0, \quad 0\le \theta_- \perp t_- \ge 0.
\end{align*}
Note that  the first and second equations are equivalent to the first equation and $\alpha b + A A^T \lambda + A C^T \mu + A W^T \nu =0$. Further, it is noticed that $\theta_+=\theta_-=0$, and $\lambda = \lambda_+ - \lambda_-$ with $0\le \lambda_+ \perp \lambda_- \ge 0$. Hence, by adding two slack variables $\vartheta$ and $\varphi$, the above mixed linear complementarity problem is equivalent to
\begin{eqnarray*}
  x_* & = & -\frac{1}{\alpha} \Big( A^T \lambda_+ - A^T \lambda_- + C^T \mu + W^T \nu \Big), \\
 0 \le \mu & \perp & -\frac{C}{\alpha} \Big( A^T \lambda_+ - A^T \lambda_- + C^T \mu + W^T \nu \Big)  - d \le 0, \\
 0 \le \nu & \perp & -\frac{W}{\alpha} \Big( A^T \lambda_+ - A^T \lambda_- + C^T \mu + W^T \nu \Big) + \Gamma - t_+ \mathbf 1 + t_-\mathbf 1 \le 0, \\
  0\le t_+ & \perp & 1 - \mathbf 1^T \nu \ge 0, \\
  0\le t_- & \perp & -1 + \mathbf 1^T \nu \ge 0, \\
  0\le \lambda_+ & \perp & \lambda_- \ge 0, \\
 0 \le \vartheta & \perp &  \alpha b + A A^T (\lambda_+ - \lambda_-) + A C^T \mu + A W^T \nu \ge 0, \\
 0 \le \varphi & \perp &  \alpha b + A A^T (\lambda_+ - \lambda_-) + A C^T \mu + A W^T \nu \le 0.
\end{eqnarray*}
The latter seven complementarity conditions in the above formulation yield the following linear complementarity problem (LCP): $0\le u \perp M u + q \ge 0$, where $u=(\mu, \nu, t_+, t_-, \lambda_+, \lambda_-,  \vartheta, \varphi) \in \mathbb R^\ell_+ \times \mathbb R^r_+  \times \mathbb R_+\times \mathbb R_+\times  \mathbb R^m_+\times \mathbb R^m_+ \times \mathbb R^m_+\times \mathbb R^m_+$, $M$ is a constant matrix of order $(\ell+r+4m+2)$ that depends on $A, C, W, \alpha$ only, and the vector $q=(d, -\Gamma, 1, 1, 0, 0, \alpha b, -\alpha b)\in \mathbb R^\ell \times \mathbb R^r \times  \mathbb R\times \mathbb R\times \mathbb R^m\times \mathbb R^m \times \mathbb R^m\times \mathbb R^m$. Denote this LCP by LCP$(q, M)$.
For any given $b\in A \mathcal C$, LCP$(q, M)$ attains a solution $u$ which pertains to the Lagrange multipliers $\lambda, \mu, \nu$, $t_+, t_-$ and the slack variables  $\vartheta$ and $\varphi$. This shows that for any given $b\in A \mathcal C$, LCP$(q, M)$ has a nonempty solution set $\SOL(q, M)$. Further, for any $\wt u=(\wt\mu, \wt\nu, \wt t_+, \wt t_-, \wt \lambda_+, \wt\lambda_-,  \wt\vartheta, \wt\varphi) \in \SOL(q, M)$, if follows from the last two complementarity conditions that $\wt x:=-\frac{1}{\alpha} \big( A^T \wt \lambda_+ - A^T \wt \lambda_- + C^T \wt \mu + W^T \wt \nu \big)$ satisfies $A \wt x = b$. Besides,  $\wt \lambda:=\wt \lambda_+-\wt \lambda_-$, $\wt \mu, \wt \nu$, and $\wt \theta_+ = \wt\theta_-=0$ satisfy
%
%
%
 the optimality conditions of the underlying optimization problem (\ref{eqn:opt_problem_Lipschtiz}) at $(\wt t_+, \wt t_-, \wt x)$ for the given $b \in A \mathcal C$.
Define the matrix
\[
   E \, := \, -\frac{1}{\alpha}\begin{bmatrix} C^T & W^T & 0 & 0 & A^T & -A^T  & 0 & 0 \end{bmatrix} \in \mathbb R^{N\times  (\ell+r+4m+2)}.
\]
It follows from the solution uniqueness of the underlying optimization problem (\ref{eqn:opt_problem_Lipschtiz}) that for any $b \in A \mathcal C$, $E\SOL(q, M)$ is singleton. Define the function $F(q):=E\SOL(q, M)$. Hence,  $F(\cdot)$ is singleton on the closed convex set $\Wcal:=\{ q=(d, -\Gamma, 1, 1, 0, 0, \alpha b, -\alpha b) \, | \, b \in A\mathcal C\}$ and $x_*(b)=F(q)$. By Theorem~\ref{thm:LCP_singleton}, $F$ is Lipscthiz on $\Scal$, i.e., there exists $ L>0$ such that $\|F(q') - F(q)\|_2 \le L \| q' - q\|_2$ for all $q', q \in \Wcal$. Since $\|q' - q\|_2 = \sqrt{2} \alpha \|b' -b\|_2$ for any $b', b \in A \mathcal C$, the desired (global) Lipschitz property of $x_*$ in $b$ holds.
\end{proof}

For a general polyhedral set $\mathcal C$, it follows from Lemmas~\ref{lem:primal_dual_Lasso} and \ref{lem:primal_dual_BPDN} that $y_* + b \in A \mathcal C$ (respectively  $\frac{\sigma y_*}{\|y_*\|_2}+b \in A \mathcal C$), where $y_*$ is a solution to the dual problem (\ref{eqn:Lasso_dual}) (respectively (\ref{eqn:BPDN_dual})). Practically, $y_*$ is approximated by a numerical sequence $(y^k)$ generated in the first stage. For the LASSO-like problem (\ref{eqn:Lasso_primal}), one uses $y^k+b$ (with a large $k$) instead of $y_*+b$ in the $\mbox{BP}_{\mbox{LASSO}}$ (\ref{eqn:BP_Lasso}) in the second stage. This raises the question of whether $y^k+b \in A \mathcal C$ for all large $k$, which pertains to the feasibility of $b\in A \mathcal C$ subject to perturbations. The same question also arises for the BPDN-like problem  (\ref{eqn:BPDN_primal}).
We discuss a mild sufficient condition on $A$ and $\Ccal$ for the feasibility under perturbations for a given $b$. Suppose $\Ccal$ has a nonempty interior and $A$ has full row rank, which holds for almost all $A \in \mathbb R^{m\times N}$ with $N \ge m$. In view of $\mbox{ri}(A\Ccal)=A \mbox{ri}(\Ccal)=A \mbox{int}(\Ccal)$ \cite[Theorem 6.6]{Rockafellar_book70}, we see that $A\Ccal$ has nonempty interior given by $A\mbox{ri}(\Ccal)=A \mbox{int}(\Ccal)$. Thus if $\wh b:=y_* + b$ is such that $\wh b=A \wh x$ for some $\wh x \in \mbox{int}(\Ccal)$, then there exists a neighborhood $\mathcal N$ of $\wh b$ such that  $b \in A \mathcal C$ for any $b \in \mathcal N$. Additional sufficient conditions independent of $b$ can also be established. For example, suppose $\Ccal$ is unbounded, and consider its recession cone $\Kcal:=\{ x \, | \, C x \le 0\}$. Let $h_i\in \mathbb R^N$ be generators of $\Kcal$, i.e., $\Kcal=\mbox{cone}\{ h_1, \ldots, h_s\}$. Define the matrix $H:=[h_1, \ldots, h_s]$. A sufficient condition for $A \Ccal$ to be open is $A \Kcal=\mathbb R^m$, which is equivalent to $A H \mathbb R^s_+ = \mathbb R^m$. By the Theorem of Alternative, the latter condition is further equivalent to (i) $AH$ has full row rank; and (ii) there exists a nonnegative matrix $Q$ such that $AH(I+Q)=0$. Some simplified conditions can be derived from it for special cases. For instance, when $\mathcal C = \mathbb R^N$, $A$ need to have full row rank; when $\mathcal C=\mathbb R^N_+$, $A$ need to have full row rank and $ A (I + Q)=0$ for a nonnegative matrix $Q$.

Based on the previous results, we establish the overall convergence of the two-stage algorithms.

\begin{theorem}
 Consider the two-stage distributed algorithms for the LASSO-like problem (\ref{eqn:Lasso_primal}) (resp. the BPDN-like problem (\ref{eqn:BPDN_primal})) with the norm $\|\cdot \|_\star$. Let $(y^k)$ be a sequence generated in the first stage such that $(y^k) \rightarrow y_*$ as $k \rightarrow \infty$ and $b+y^k\in A \mathcal C$ (resp. $b+\frac{\sigma y^k}{\|y^k\|_2} \in A \mathcal C$) for all large $k$, where $y_*$ is a solution to the dual problem (\ref{eqn:Lasso_dual}) (resp. (\ref{eqn:BPDN_dual})), and $(x^s)$ be a convergent sequence in the second stage for solving (\ref{eqn:BP_Lasso}) (resp. (\ref{eqn:BP_BPDN})). Then the following hold:
 \begin{itemize}
   \item [(i)]  $(x^s) \rightarrow x_*$ as $k, s \rightarrow \infty$,  where $x_*$ is the unique solution to the regularized $\mbox{BP}_{\mbox{LASSO}}$ (\ref{eqn:BP_Lasso}) (resp. $\mbox{BP}_{\mbox{BPDN}}$ (\ref{eqn:BP_BPDN})).

   \item [(ii)] Let $\|\cdot \|_\star$ be the $\ell_1$-norm. Suppose $(y^k)$ has the convergence rate $O(\frac{1}{k^q})$ and $(x^s)$ has the convergence rate $O(\frac{1}{s^r})$. Then $(x^s)$ converges to $x_*$ in the rate of $ O(\frac{1}{k^q})+O(\frac{1}{s^r})$.
 \end{itemize}
\end{theorem}

\begin{proof}
  We consider the LASSO-like problem only; the similar argument holds for the BPDN-like problem.

  (i) For each $k$, let $\wh b^k:= b + y^k$, where $(y^k)$ is a sequence generated from the first stage  that converges to $y_*$. When $\wh b^k$ is used in the $\mbox{BP}_{\mbox{LASSO}}$ (\ref{eqn:BP_Lasso}) in the second stage, i.e., the constraint $A x = b+y_*$ is replaced by $A x =\wh b^k$, we have $\| x^s(\wh b^k) - x_*\|\le \| x^s(\wh b^k) - x_*(\wh b^k)\| + \| x_*(\wh b^k) - x_*\|$, where $x_*(\wh b^k)$ is the unique solution to the $\mbox{BP}_{\mbox{LASSO}}$ (\ref{eqn:BP_Lasso}) corresponding to the constraint $A x = \wh b^k$ (and $x \in \mathcal C$). Since $\big( x^s(\wh b^k) \big)$ converges to $x_*(\wh b^k)$ as $s \rightarrow \infty$ (for a fixed $k$), $\| x^s(\wh b^k) - x_*(\wh b^k)\|$ converges to zero. Further, note that $x_*=x_*(\wh b_*)$ with $\wh b_*:=b+ y_*$. Then it follows from the continuity property shown in Theorem~\ref{thm:continuity_x_*} that $\| x_*(\wh b^k) - x_*\|= \| x_*(\wh b^k) - x_*(\wh b_*)\|$ converges to zero as $k \rightarrow \infty$ in view of the convergence of $(y^k)$ to $y_*$. This establishes the convergence of the two-stage algorithm. 

  (ii) When $\|\cdot\|_\star$ is the $\ell_1$-norm, we deduce via Theorem~\ref{thm:continuity_x_*_convex_PA} that $x_*$ is Lipschitz continuous in $b$ on $A \mathcal C$, i.e., there exists a constant $L>0$ such that $\|x_*(b) - x_*(b')\|\le L\|b - b'\|$ for any $b, b' \in A \mathcal C$. Hence,
   $\| x^s(\wh b^k) - x_*\|\le \| x^s(\wh b^k) - x_*(\wh b^k)\| + \| x_*(\wh b^k) - x_*(\wh b_*)\| \le \| x^s(\wh b^k) - x_*(\wh b^k)\| + L \| \wh b^k - \wh b_*\| = \| x^s(\wh b^k) - x_*(\wh b^k)\| + L \| y^k - y_*\| = O(\frac{1}{s^r}) +O(\frac{1}{k^q}) $.
\end{proof}


%
\section{Numerical Results} \label{sect:numerical}

We present numerical results to demonstrate the performance of the proposed two-stage column partition based  distributed algorithms for LASSO, fused LASSO, BPDN, and their extensions. In each case, we consider a network of $p=40$ agents with two topologies: the first is a cyclic graph, and the second is a random graph satisfying $(i, i+1) \in \mathcal E, \forall \, i=1, \ldots, p-1$ (which is needed for the fused problems) shown in Figure~\ref{Fig:graph2},
which are referred to as Scenario 1 and Scenario 2, respectively.
%
%
The matrix $A \in \mathbb R^{10 \times 400}$ is a random normal matrix, and $b \in \mathbb R^{10}$ is a random normal vector. We consider even column partitioning, i.e., each agent has 10 columns, and use the distributed averaging scheme with optimal constant edge weight \cite[Section 4.1]{XiaoBoyd_SCL04} for consensus computation.
\begin{figure}[htbp]
\begin{centering}
\includegraphics[width=0.6\textwidth]{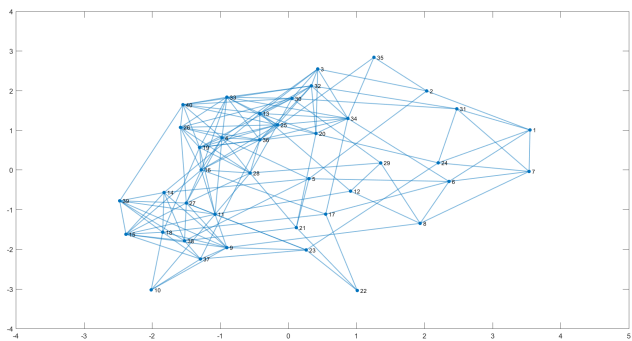}
\caption{The graph for Scenario 2.}
\label{Fig:graph2}
\end{centering}
\end{figure}

To evaluate the accuracy of the proposed schemes, let $J$ denote the objective function in each case, and $x^*_\text{dist}$ denote the numerical solution obtained using the proposed distributed schemes. Let $J^*_\text{dist}:=J(x^*_\text{dist})$, $J^*_\text{true}$ be the true optimal value obtained from a high-precision centralized scheme, and $J_\text{RE} := \frac{| J^*_{\text{dist}} - J^*_{\text{true}}|}{|J^*_{\text{true}}|}$ be the relative error of the optimal value.

\gap

%

\noindent $\bullet$ {\bf LASSO} \ The $\ell_1$-penalty parameter $\lambda =1.8$, and the regularization parameter in the second stage $\alpha = 0.18$.
%
%
When $\mathcal C=\mathbb R^N$ (resp. $\mathcal C =\mathbb R^N_+$), the termination tolerances for the first and second stage  are $10^{-7}$ (resp. $10^{-6}$) and $10^{-5}$ (resp.  $10^{-5}$) respectively.
When $\Ccal=\mathbb R^N$, the behaviors of $\|y^k - y_*\|_2$ in Stage one and Stage two over two graphs are shown in Figure~\ref{Fig:error_y_k}. It is observed that the trajectories of $\|y^k - y_*\|_2$ over two graphs coincide in both the stages.
%
%

\begin{multicols}{2}
\begin{tabular}{|p{1.5cm}|p{1.7cm}|c|c|}
\hline
 \multicolumn{2}{|c|}{Constraint}& $J^*_\text{dist}$ & $J_\text{RE}$ \\ \hline
\ \ $\Ccal =\mathbb R^N$ & Scenario 1  & 1.7211 &  $6.6\times 10^{-4}$ \\ 
\cline{2-4}
  & Scenario 2 & 1.7211 & $6.7\times 10^{-4}$  \\  \hline
\end{tabular}

\columnbreak

%

\begin{tabular}{|p{1.5cm}|p{1.7cm}|c|c|}
\hline
 \multicolumn{2}{|c|}{Constraint}& $J^*_\text{dist}$ & $J_\text{RE}$ \\ \hline
\ \ $\Ccal =\mathbb R^N_+$ & Scenario 1  & 1.9012 &  $1.6\times 10^{-5}$ \\ 
\cline{2-4}
  & Scenario 2 & 1.9012 & $1.2\times 10^{-5}$ \\  \hline
\end{tabular}

\end{multicols}
The scaled regularized BP is also applied to the second stage scheme of the LASSO (cf. Remark~\ref{remark:Lasso_scaled_BP}), which yields the similar performance and accuracy. Its details are omitted.

%
%
\begin{figure}[t]  \centering
\begin{tabular}{cc}
\includegraphics[width=0.47\textwidth]{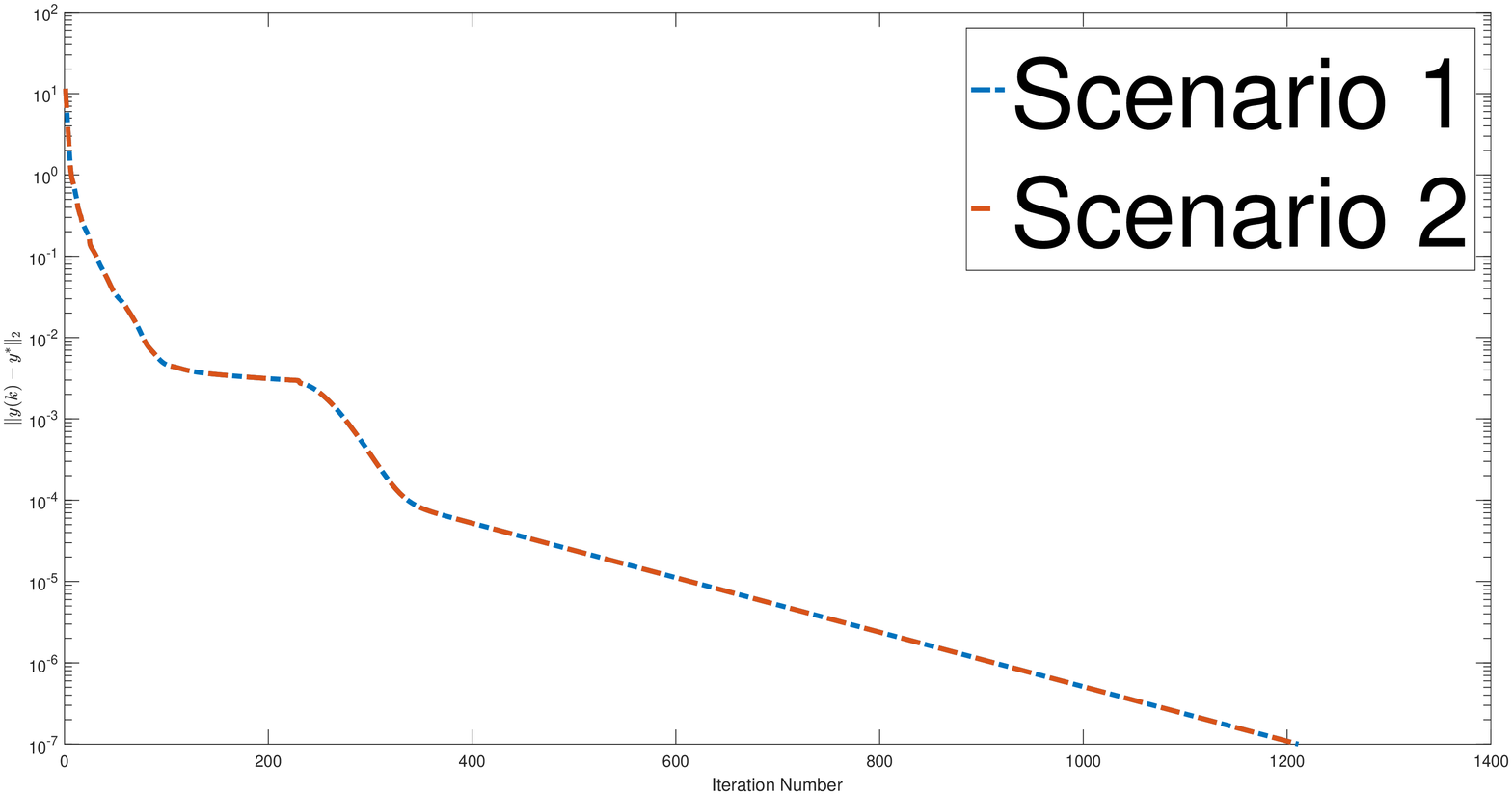} &  
\includegraphics[width=0.47\textwidth]{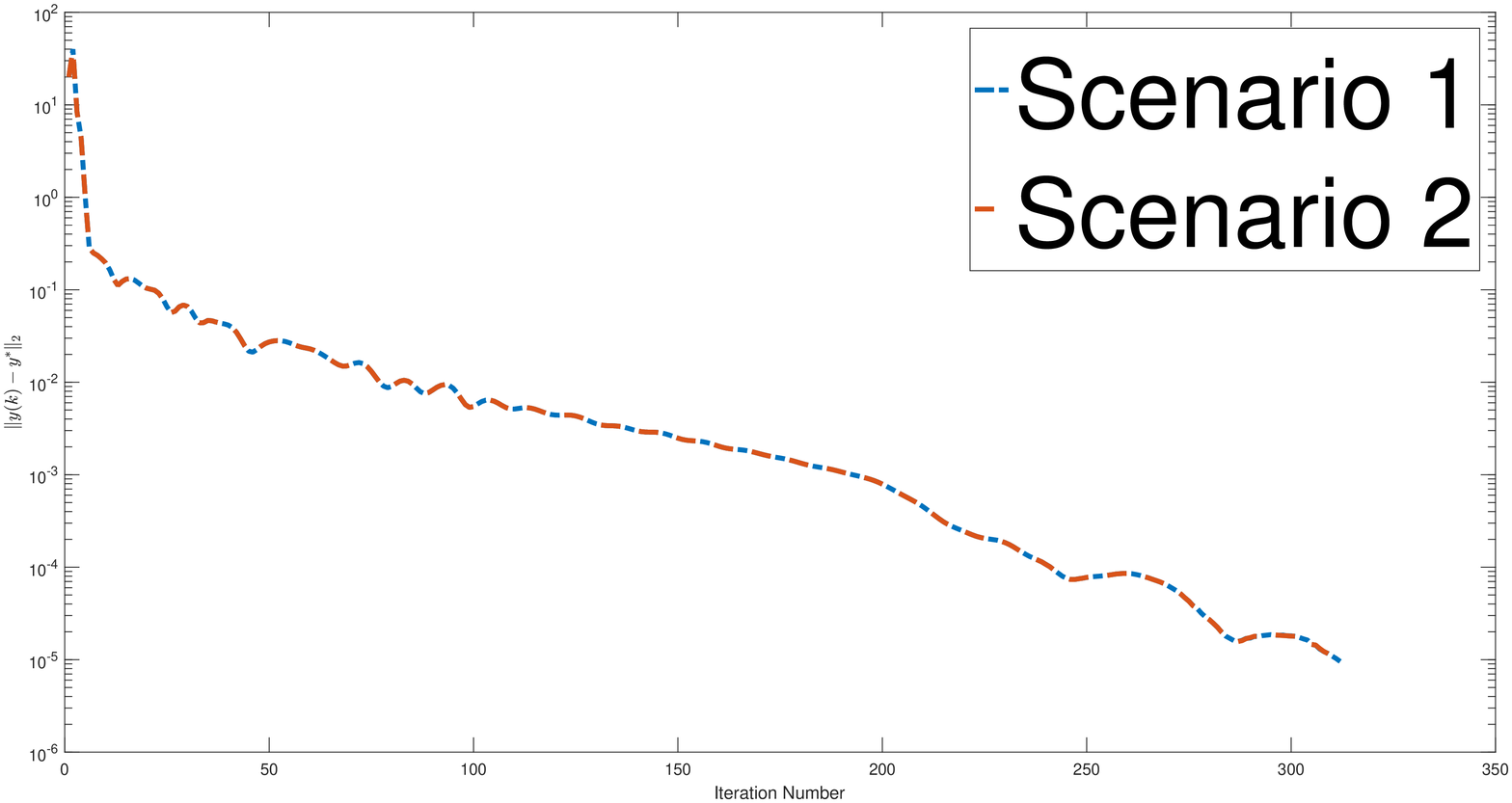}
\end{tabular}
\caption{Trajectories of $\|y^k - y_*\|_2$ in the LASSO for $\Ccal=\mathbb R^N$ (using 10 logarithmic scale for the vertical axis and linear scale for the horizontal axis). Left: Stage one; Right: Stage two.}
\label{Fig:error_y_k}
%
\end{figure}

\gap

\noindent $\bullet$ {\bf BPDN} \ The parameter $\sigma =0.2$, and the regularization parameter in the second stage $\alpha = 0.15$. Further, $\|b\|_2=2.9688$ such that $\|b\|_2>\sigma$.
When $\mathcal C=\mathbb R^N$ (resp. $\mathcal C =\mathbb R^N_+$), the termination tolerances for the first and second stage  are $10^{-7}$ (resp. $10^{-5}$) and $8\times 10^{-4}$ (resp.  $2\times 10^{-4}$) respectively.
%
%

\begin{multicols}{2}
%
%
\begin{tabular}{|p{1.5cm}|p{1.7cm}|c|c|}
\hline
 \multicolumn{2}{|c|}{Constraint}& $J^*_\text{dist}$ & $J_\text{RE}$ \\ \hline
\ \ $\Ccal =\mathbb R^N$ & Scenario 1  & 1.0271 &  $2.2\times 10^{-5}$ \\ 
\cline{2-4}
  & Scenario 2 & 1.0209 & $5.9\times 10^{-3}$  \\  \hline
\end{tabular}

\columnbreak

%

\begin{tabular}{|p{1.6cm}|p{1.8cm}|c|c|}
\hline
 \multicolumn{2}{|c|}{Constraint}& $J^*_\text{dist}$ & $J_\text{RE}$ \\ \hline
\ \ $\Ccal =\mathbb R^N_+$ & Scenario 1  & 1.1599  & $5.4\times 10^{-4}$   \\ 
\cline{2-4}
  & Scenario 2 & 1.1607 & $1.3\times 10^{-3}$  \\  \hline
\end{tabular}

\end{multicols}

%

\gap

\noindent $\bullet$ {\bf Fused LASSO} \ The matrix $E = \begin{bmatrix}
\lambda I \\
\gamma D_1
\end{bmatrix} $ with $\lambda =0.6$ and $\gamma=0.4$, and the regularization parameter $\alpha = 0.18$. For $\mathcal C=\mathbb R^N$ and $\mathcal C =\mathbb R^N_+$, the termination tolerances for the first and second stages  are $10^{-5}$  and $10^{-4}$, respectively.
%
%

\begin{multicols}{2}
%
%
\begin{tabular}{|p{1.5cm}|p{1.7cm}|c|c|}
\hline
 \multicolumn{2}{|c|}{Constraint}& $J^*_\text{dist}$ & $J_\text{RE}$ \\ \hline
\ \ $\Ccal =\mathbb R^N$ & Scenario 1  & 1.2431 & $7.2\times 10^{-3}$ \\ 
\cline{2-4}
  & Scenario 2 &  1.2376 &  $2.8 \times 10^{-3}$  \\  \hline
\end{tabular}

\columnbreak

%

\begin{tabular}{|p{1.5cm}|p{1.7cm}|c|c|}
\hline
 \multicolumn{2}{|c|}{Constraint}& $J^*_\text{dist}$ & $J_\text{RE}$ \\ \hline
\ \ $\Ccal =\mathbb R^N_+$ & Scenario 1  & 1.4346 &  $4.2\times 10^{-3}$ \\ 
\cline{2-4}
  & Scenario 2 & 1.4474  & $1.3\times 10^{-2}$ \\  \hline
\end{tabular}

\end{multicols}

%

\gap

\noindent $\bullet$ {\bf BPDN arising from Fused Problem} \ Consider the BPDN-like problem: $\min_{x \in \Ccal} \| E x \|_1$ subject to $\| A x -b \|_2 \le \sigma$, where the matrix $E = \begin{bmatrix}
\lambda I \\
\gamma D_1
\end{bmatrix} $ with $\lambda =0.6$ and $\gamma=0.4$, and the regularization parameter  $\alpha = 0.18$. Further, $\sigma=0.2$ and $\|b\|_2=2.9688$.
When $\mathcal C=\mathbb R^N$ (resp. $\mathcal C =\mathbb R^N_+$), the termination tolerances for the first and second stages  are $10^{-5}$ (resp. $10^{-4}$) and $10^{-5}$ (resp. $10^{-5}$) respectively.
%
%


%
%

\begin{multicols}{2}
%
%
\begin{tabular}{|p{1.5cm}|p{1.7cm}|c|c|}
\hline
 \multicolumn{2}{|c|}{Constraint}& $J^*_\text{dist}$ & $J_\text{RE}$ \\ \hline
\ \ $\Ccal =\mathbb R^N$ & Scenario 1  & 1.2823 & $6.5\times 10^{-3}$ \\ 
\cline{2-4}
  & Scenario 2 & 1.2841  & $7.9\times 10^{-3}$ \\  \hline
\end{tabular}

\columnbreak

%

\begin{tabular}{|p{1.5cm}|p{1.7cm}|c|c|}
\hline
 \multicolumn{2}{|c|}{Constraint}& $J^*_\text{dist}$ & $J_\text{RE}$ \\ \hline
\ \ $\Ccal =\mathbb R^N_+$ & Scenario 1  & 1.4815 & $4.9\times 10^{-4}$   \\ 
\cline{2-4}
  & Scenario 2 & 1.4817 & $5.9\times 10^{-4}$ \\  \hline
\end{tabular}

\end{multicols}

\gap

\noindent $\bullet$ {\bf Group LASSO} \ Consider $\Ccal=\mathbb R^N$ and a cyclic graph with the penalty parameter $\lambda =1.8$, and the regularization parameter $\alpha = 0.18$. The termination tolerances for the first and second stages are $10^{-5}$ and $8\times 10^{-6}$, respectively. The numerical tests show that $J^*_\text{dist}= 1.2208$ and $J_\text{RE}= 9.8\times 10^{-4}$.

The above results demonstrate the effectiveness of the proposed two-stage distributed algorithms.


%
\section{Conclusions} \label{sect:conclusions}

In this paper, column partition based distributed schemes are developed for a class of densely coupled convex sparse optimization problems, including BP, LASSO, BPDN and their extensions. By leveraging duality theory, exact regularization techniques, and solution properties of the aforementioned problems, we develop dual based fully distributed schemes via column partition. Sensitivity results are used to establish overall convergence of the two-stage distributed schemes for LASSO, BPDN, and their extensions. The proposed schemes and techniques shed light on the development of column partition based distributed schemes for a broader class of densely coupled problems, which will be future research topics.
%
%


\end{document}